%% file: manuscript.tex
\documentclass[1p]{elsarticle}

\usepackage{lineno,hyperref}
\usepackage{amsfonts}
\usepackage{amsmath}
\usepackage{amssymb}
\usepackage{amsthm}
\usepackage{float}
\usepackage{todonotes}
\usepackage{fonts,defs}
\usepackage{mathrsfs}
\usepackage{subfig}

\modulolinenumbers[5]
\hypersetup{
  bookmarksnumbered = true,
  bookmarksopen=false,
  pdfborder=0 0 0,         
  pdffitwindow=true,      
  pdfnewwindow=true, 
  colorlinks=true,           
  linkcolor=blue,            
  citecolor=magenta,    
  filecolor=magenta,     
  urlcolor=cyan              
}

\journal{Journal of Computational Physics}

\newtheorem{lemma}{Lemma}
\newtheorem{prop}{Proposition}
\newtheorem{rem}{Remark}

\begin{document}

\begin{frontmatter}

\title{Conservative DG Method for the Micro-Macro Decomposition of the Vlasov--Poisson--Lenard--Bernstein Model\tnoteref{support}\tnoteref{copyright}}
\tnotetext[support]{
This research has been supported by the DOE Office of Advance Scientific Computing Research through the SciDAC Partnership Center for High-fidelity Boundary Plasma Simulation under contract DE-AC05-00OR22725 with Oak Ridge National Laboratory (ORNL), managed by UT-Battelle, LLC for the U.S. Department of Energy.}
\tnotetext[copyright]{
This manuscript has been authored in part by UT-Battelle, LLC, under contract DE-AC05-00OR22725 with the US Department of Energy (DOE). The US government retains and the publisher, by accepting the article for publication, acknowledges that the US government retains a nonexclusive, paid-up, irrevocable, worldwide license to publish or reproduce the published form of this manuscript, or allow others to do so, for US government purposes. DOE will provide public access to these results of federally sponsored research in accordance with the DOE Public Access Plan (http://energy.gov/downloads/doe-public-access-plan).}


\author[ornl,utk-phys]{Eirik Endeve\corref{cor}}
\ead{endevee@ornl.gov}

\author[ornl,utk-math]{Cory D. Hauck}
\ead{hauckc@ornl.gov}

\cortext[cor]{Corresponding author. Tel.:+1 865 576 6349; fax:+1 865 241 0381}

\address[ornl]{Multiscale Methods and Dynamics Group, Oak Ridge National Laboratory, Oak Ridge, TN 37831 USA }

\address[utk-phys]{Department of Physics and Astronomy, University of Tennessee Knoxville, TN 37996-1200}

\address[utk-math]{Department of Mathematics, University of Tennessee Knoxville, TN 37996-1320}


\begin{abstract}
The micro-macro (mM) decomposition approach is considered for the numerical solution of the Vlasov--Poisson--Lenard--Bernstein (VPLB) system, which is relevant for plasma physics applications.  
In the mM approach, the kinetic distribution function is decomposed as $f=\cE[\bsrho_{f}]+g$, where $\cE$ is a local equilibrium distribution, depending on the macroscopic moments $\bsrho_{f}=\int_{\bbR}\be fdv=\vint{\be f}_{\bbR}$, where $\be=(1,v,\f{1}{2}v^{2})^{\rm{T}}$, and $g$, the microscopic distribution, is defined such that $\vint{\be g}_{\bbR}=0$.  
We aim to design numerical methods for the mM decomposition of the VPLB system, which consists of coupled equations for $\bsrho_{f}$ and $g$.  
To this end, we use the discontinuous Galerkin (DG) method for phase-space discretization, and implicit-explicit (IMEX) time integration, where the phase-space advection terms are integrated explicitly and the collision operator is integrated implicitly.  
We give special consideration to ensure that the resulting mM method maintains the $\vint{\be g}_{\bbR}=0$ constraint, which may be necessary for obtaining (i) satisfactory results in the collision dominated regime with coarse velocity resolution, and (ii) unambiguous conservation properties.  
The constraint-preserving property is achieved through a consistent discretization of the equations governing the micro and macro components.  
We present numerical results that demonstrate the performance of the mM method.  
The mM method is also compared against a corresponding DG-IMEX method solving directly for $f$.  
\end{abstract}


\end{frontmatter}

%
%


\input{intro}
\input{model}
\input{method}
\input{numerical}
\input{conclusion}

\section*{Acknowledgements}
We acknowledge fruitful conversations with A.~Hakim and G.~W. Hammett.  
We also thank A.~Hakim and M.~Francisquez for providing damping rate estimates used to evaluate results obtained for the collisional Landau damping problem in Section~\ref{sec:numerical}.  

\appendix

\input{appendix}


\bibliography{references}
\end{document}

%% file: intro.tex
\section{Introduction}
\label{sec:intro}

In this paper, we design discontinuous Galerkin (DG) methods to solve multiscale kinetic equations of electrostatic plasma, using the framework of micro-macro decomposition.  
In a multispecies plasma, the dynamics of particle species $s$ can be described by a kinetic equation (e.g., \cite{braginskii_1965,chapmanCowling_1970}) of the form 
\begin{equation}
  \p_{t}f_{s} + \bv\cdot\nabla_{\bx}f_{s} + \ba_{s}\cdot\nabla_{\bv}f_{s} = \sum_{s'}\cC(f_{s},f_{s'}),
  \label{eq:boltzmannGeneral}
\end{equation}
where the phase-space distribution function $f_{s}$, depending on velocity $\bv\in\bbR^{3}$ and position $\bx\in\bbR^{3}$, is defined such that, at time $t\in\bbR^{+}$, $f_{s}d\bv d\bx$ gives the number of particles of species $s$ in the phase-space volume element $d\bv d\bx$.  
In Eq.~\eqref{eq:boltzmannGeneral}, $\ba_{s}$ is the acceleration experienced by particles of species $s$ (e.g., due to electromagnetic forces), and the collision term $\cC(f_{s},f_{s'})$ describes collisional interactions between particles of species $s$ and $s'$ (including self collisions).  

Solving Eq.~\eqref{eq:boltzmannGeneral} numerically is challenging for several reasons, including (i) the dimensionality of phase-space, which demands large-scale computational resources \cite{germaschewski_etal_2021}; (ii) the multiple spatial and temporal scales introduced by the individual terms \cite{filbetJin_2010}; (iii) the fact that the kinetic equation gives rise to conservation laws that can be nontrivial to capture in a numerical method \cite{cheng_etal_2013,juno_2018}; and (iv) maintaining positivity of the distribution function \cite{rossmanithSeal_2011}.   
We focus in this paper on the single species case and therefore drop the species subscript $s$ from here on; i.e., $f_{s}\to f$, etc. 

In many applications of kinetic theory --- including fusion and astrophysical plasmas --- the physical system is characterized by regions with varying degrees of collisionality.  
In the \emph{fluid regime}, characterized by frequent interparticle collisions, the distribution function is driven towards a local equilibrium distribution $\cE[\bsrho_{f}]$, depending on a limited number of velocity moments --- $\bsrho_{f}(\bx,t)=\int_{\bbR^{3}}f\be d\bv$, where $\be=(1,\bv,\f{1}{2}|\bv|^{2})^{\rm{T}}$ --- representing the \emph{particle} number, momentum, and energy density of the plasma.  
In such cases fluid models (e.g., the Euler, Navier--Stokes, or magnetohydrodynamics equations) for these moments provide an adequate and much cheaper description.  
In the \emph{kinetic regime}, the particle mean free path exceeds other physical length scales, introducing nonlocal effects, and the kinetic equation must be solved to accurately capture the dynamics.  
Much work has been devoted to the development of asymptotic-preserving numerical methods that work well in both regimes, as well as in the transition between the two.  
An overview of these approaches, along with numerous additional references, can be found in \cite{dimarco2014numerical,jin2010asymptotic}

In the micro-macro (mM) formulation \cite{liuYu_2004,degond_etal_2006,bennoune_etal_2008,crestetto_etal_2012,xiong_etal_2015,gamba_etal_2019}, the distribution function is decomposed as $f=\cE[\bsrho_{f}]+g$, where the microscopic part $g$ is defined such that $\bsrho_{g}=\int_{\bbR^{3}}g\be d\bv=0$.  From Eq.~\eqref{eq:boltzmannGeneral}, an equivalent set of coupled equations can be derived for $\bsrho_{f}$ and $g$.  
One benefit of solving the mM system, as opposed to Eq.~\eqref{eq:boltzmannGeneral} directly for $f$, is that the fluid regime, where $g$ vanishes, is captured exactly by the fluid model evolving the macroscopic part $\bsrho_{f}$.  
It has been demonstrated that numerical methods for the mM decomposition of the neutral-particle BGK equation can capture the Euler and compressible Navier--Stokes limits as the Knudsen number tends to zero \cite{bennoune_etal_2008,xiong_etal_2015}.  
In \cite{gamba_etal_2019}, the application of the mM method was extended to kinetic equations with more complicated collision operators; i.e., the Boltzmann and Fokker--Planck--Landau collision operators.  
The mM method can potentially offer gains in computational efficiency (relative to direct methods) in collisional regimes because fewer degrees of freedom are typically needed to capture the correct dynamics; although phase-space adaptivity \cite{hittingerBanks_2013} or other reduced-memory techniques \cite{einkemmer_etal_2021,guo2016sparse} may be needed to realize the full potential of the mM method.  
As an example, the mM decomposition was used in \cite{crestetto_etal_2012} to solve the Vlason--Poisson--BGK equations, using a particle method for $g$, and a reduction in computational cost was demonstrated in the fluid regime, where fewer particles were needed.  
Another potential benefit of solving the mM system is that conservation laws for particle number, momentum, and energy are evolved directly, as opposed to indirectly when solving Eq.~\eqref{eq:boltzmannGeneral} for $f$, which suggests that these quantities are automatically conserved \cite{gamba_etal_2019} --- independent of the evolution of the microscopic part $g$.  
However, if the numerical method for $g$ violates the evolution constraint $\bsrho_{g}=0$, the velocity moments of $f$ differ from those of $\cE[\bsrho_{f}]$, and the conservation properties of the mM method become ambiguous.  

In this paper, we design numerical methods for the Vlasov--Poisson--Lenard--Bernstein (VPLB) system (see, e.g., \cite{hakim_etal_2020}) in one spatial and one velocity dimension.  
The VPLB system consists of the Vlasov--Poisson (VP) equations supplemented with the Lenard--Bernstein (LB) collision operator \cite{lenardBernstein_1958}, which includes velocity-space drift and diffusion terms.  
We consider two approaches to solving the VPLB system: one based on the mM decomposition (mM method; our primary focus), and one that solves directly for $f$ (direct method; mainly included as a means of comparison).  
For the approach based on the mM decomposition, in order to remove any ambiguity in the conservation properties, we aim to develop methods that maintain the evolution constraint $\bsrho_{g}=0$.  
We also seek to compare the performance of the mM method --- in terms of conservation properties and accuracy for a given phase-space resolution --- against the direct method.  

We use the DG method \cite{cockburnShu_2001} to discretize the kinetic equation in the VPLB system in phase-space.  
The DG method achieves high-order accuracy in phase-space by approximating the solution in each element by a local polynomial expansion of arbitrary degree.  
As a projection-based method, it is attractive for solving kinetic equations, in part because certain conservation properties can be built into the method with an appropriate choice of test functions; see, e.g., \cite{ayuso_etal_2011,cheng_etal_2013,hakim_etal_2019}.  
We use a nodal DG method \cite{hesthavenWarburton_2008} with upwind-type numerical fluxes for the phase-space advection (hyperbolic) parts, while, following \cite{hakim_etal_2020}, we use the recovery DG method of \cite{vanLeerNomura_2005} to approximate the diffusion part of the LB collision operator.  
Similar to \cite{hakim_etal_2020}, we show that the resulting discretization of the LB operator conserves particle number, momentum, and energy, but now with an implicit time discretization.

We use implicit-explicit (IMEX) time-stepping methods \cite{ascher_etal_1997,pareschiRusso_2005} to integrate the VPLB system forward in time.  
To avoid severe time step restrictions for stability with explicit methods in the fluid regime, we use implicit integration of the LB collision operator, while the phase-space advection part is integrated with explicit methods.  
In terms of the numerical methods employed, the work in \cite{xiong_etal_2015}, which used IMEX time-stepping and a DG discretization in space to solve the mM decomposition of the neutral-particle BGK equation, is closely related to the current work.  

We show that the DG-IMEX method for the mM decomposition of the VPLB system preserves the constraints $\bsrho_{g}=0$, provided they are satisfied initially.  
With this property, the conservation of number, momentum, and energy is governed by the macro component, and the conservation properties are as good as those achieved with the direct method.  
The maintenance of $\bsrho_{g}=0$ is achieved by carefully ensuring consistency of the discretization of the micro and macro components.  
(Due to finite velocity domain effects, we introduce a limiter to enforce $\bsrho_{g}=0$ after each explicit step in the time-stepping method.)  
Through numerical experiments, we find that the consistency required to maintain $\bsrho_{g}=0$ is necessary to achieve good results in the fluid regime when using coarse velocity meshes, and thereby better leverage the benefits of the mM method.  
We then demonstrate that the mM method, when compared with the direct method, can improve efficiency in the fluid regime.  

The remainder of this paper is organized as follows.  
In Section~\ref{sec:model} we present the VPLB model and the associated mM decomposition.  
In Section~\ref{sec:dg}, we detail the DG discretization of the direct and mM methods, together with the Poisson solver and IMEX time-stepping methods used in the respective approaches.
In Section~\ref{sec:conservation}, we investigate conservation properties.
In Section~\ref{sec:cleaningLimiter}, we present the cleaning limiter used with the mM method to enforce $\bsrho_{g}=0$.  
In Section~\ref{sec:numerical}, we present numerical results.  Conclusions are provided in Section~\ref{sec:conclusions}.  

%% file: model.tex
\section{Mathematical Model}
\label{sec:model}

\subsection{Vlasov--Poisson--Lenard--Bernstein  Model}

The Vlasov--Poisson--Lenard--Bernstein (VPLB) model describes the evolution of charged particles subject to a self-consistent electrostatic potential and collisional dynamics.  
It consists of a kinetic equation for a phase-space distribution function coupled to a Poisson equation for the potential.  
In one spatial and one velocity dimension (1D1V), the kinetic equation for ions of unit mass and charge takes the form
\begin{equation}
  \p_{t}f(v,x,t) + v\,\p_{x}f(v,x,t) + E(x,t)\,\p_{v}f(v,x,t) = \cC(f)(v,x,t),
  \label{eq:boltzmann}
\end{equation}
where the phase-space distribution $f$ depends on velocity $v\in\bbR$, position $x \in D^x \subset \bbR$, and time $t\in\bbR^{+}$.
The electric field $E=- \p_{x}\Phi$ is the derived from the potential $\Phi$, which satisfies the Poisson equation
\begin{equation}
  - \p_{xx}\Phi(x,t) = (n_{f}(x,t)-n_{\rm{e}}).
  \label{eq:poisson}
\end{equation}
Here $n_{f}=\vint{f}_{\bbR} \equiv \int_{\bbR}f dv$ is the ion density, and $n_{\rm{e}}$ is a constant electron density over the spatial domain $D^x$, which is assumed to be bounded. Throughout the manuscript, we use the notation $\vint{\ldots}_{\Omega}\equiv\int_{\Omega}\ldots dv$ for any $\Omega \subseteq \bbR$.

In the VPLB model, the collision operator $\cC$ on the right-hand side of Eq.~\eqref{eq:boltzmann} is the Lenard--Bernstein (LB) operator \cite{lenardBernstein_1958,hakim_etal_2020}, a  simplified approximation of the more realistic Fokker--Planck--Landau operator (see, e.g. \cite{hazeltine2018framework}).  
The LB operator takes the form
\begin{equation}
  \cC(f) := \cC_{\lb}[\bsrho_{f}](f) = \nu\,\p_{v}\big(\,(v-u_{f})\,f+\theta_{f}\,\p_{v}f\,\big).
  \label{eq:lenardBernstein}
\end{equation}
Here the constant $\nu>0$ is an effective collision frequency;  the moments
\begin{equation}
\label{eq:macroMoments}
\bsrho_{f} 
= \vint{\be f}_{\bbR}, \quad\text{where}\quad\be
=(e_{0},e_{1},e_{2})^{\rm{T}}\equiv(1,v,\frac{1}{2}v^{2})^{\rm{T}},
\end{equation}  
are the number, momentum, and energy densities, respectively; and 
\begin{equation}
u_{f}=\frac{1}{n_f}\vint{fv}_{\bbR }
	\quand
\theta_f = \frac{1}{n_f}\vint{f(v-u_{f})^{2}}_{\bbR}
\end{equation} 
are the bulk velocity and temperature, respectively, associated to $f$.  The map between the fluid variables $(n_f, u_f,\theta_f)$ and $\bsrho_f$ is a simple bijection:
\begin{equation}
\label{eq:momentFluidBijection}
  \bsrho_f = \Big(\,n_f,\, n_f u_f,\, \frac12 n_f (u_f^2 + \theta_f) \,\Big)^{\rm{T}}.
\end{equation} 

The VPLB model, given by Eqs.~\eqref{eq:boltzmann}-\eqref{eq:lenardBernstein}, must be supplemented with suitable initial and boundary conditions, which we leave unspecified for now.  

The following results hold for the VPLB model \cite{hakim_etal_2020}.
\begin{prop} 
	\label{prop:collisionProperties}
	The LB operator satisfies the following properties
\begin{enumerate}
  \item Conservation of number, momentum, and energy:
  \begin{equation}
    \vint{\cC_{\lb}[\bsrho_{w}](w)\,\be}_{\bbR}=0 \quad \forall w \in \operatorname{Dom}( \cC_{\lb})
    \label{eq:lbConservation}
  \end{equation}
  \item Dissipation of entropy:
  \begin{equation}
    \vint{\cC_{\lb}[\bsrho_{w}](w)\log w}_{\bbR}\le0 \quad \forall w \in \operatorname{Dom}( \cC_{\lb})
  \end{equation}
\item Characterization of equilbria (H-theorem):  For any $w \in \operatorname{Dom}(\cC_{\lb})$,
  \begin{equation}
	\vint{\cC_{\lb}[\bsrho_{w}](w\log w}_{\bbR} = 0
\end{equation}
if and only if 
\begin{equation} 
	w= M_w:= \frac{n_{w}}{\sqrt{2\pi\theta_{h}}}\exp\Big\{-\frac{(v-u_{w})^{2}}{2\theta_{w}}\Big\}.
	\label{eq:maxwellian}
\end{equation}
\end{enumerate}
A consequence of Eq.~\eqref{eq:lbConservation} is that 
\begin{equation}
\label{eq:momentModel}
  \p_{t}\bsrho_{f} + \p_{x} \vint{\be v f}_{\bbR} = ET\bsrho_{f},
\end{equation}
where the $3\times3$ matrix $T$ is given by
\begin{equation}
  T=
  \left(\begin{array}{ccc}
    0 & 0 & 0 \\
    1 & 0 & 0 \\
    0 & 1 & 0
  \end{array}\right).
\end{equation}
\end{prop}
Using Eqs.~\eqref{eq:poisson} and \eqref{eq:momentModel}, it is straightforward to show the following.  
\begin{prop} 
	The VPLB model gives rise to the following local conservation laws:
	\begin{itemize}
		\item Conservation of number:
		\begin{equation}
			\p_{t}\vint{e_{0}f}_{\bbR} + \p_{x}\vint{e_{0}vf}_{\bbR} = 0.
			\label{eq:particleConservation}
		\end{equation}
	\item Conservation of momentum:
	\begin{equation}
		\p_{t}\vint{e_{1}f}_{\bbR} + \p_{x}\big(\,\vint{e_{1}vf}_{\bbR}+n_{\rm{e}}\Phi-\f{1}{2}E^{2}\,\big) = 0.  
		\label{eq:momentumConservation}
	\end{equation}
	\item Conservation of energy:
	\begin{equation}
		\p_{t}\big(\vint{e_{2}f}_{\bbR}+\f{1}{2}E^{2}\big) + \p_{x}\vint{e_{2}vf}_{\bbR} = 0.
		\label{eq:totalEnergyConservation}
	\end{equation}
	\end{itemize}
In particular,  when $E=0$,  Eqs.~\eqref{eq:particleConservation}-\eqref{eq:totalEnergyConservation} reduce to local conservation laws for the moments $\bsrho_f$.
\end{prop}

\subsection{Micro-Macro Decomposition}

In a micro-macro (mM) decomposition \cite{liuYu_2004}, the distribution function is decomposed into an equilibrium (macro) component $\cE[\bsrho_f]$ and non-equilibrium (micro) component $g$; that is,
\begin{equation}
  f = \cE[\bsrho_{f}] + g,
  \label{eq:microMacroDecomposition}
\end{equation}
where $\cE[\bsrho_{f}]:=M_{f}$. Because $\vint{\be \cE[\bsrho_{f}]}_{\bbR}=\bsrho_{f} $, it follows that
\begin{equation}
  \bsrho_{g} = \vint{\be g}_{\bbR} = 0.
  \label{eq:gConstraints}
\end{equation}

Inserting the mM decomposition from Eq.~\eqref{eq:microMacroDecomposition} into Eq.~\eqref{eq:momentModel} gives 
the following equations for $\bsrho_f$:
\begin{equation}
  \p_{t}\bsrho_{f} + \p_{x}\bF(\bsrho_{f}) + \p_{x}\bff(g) = ET\bsrho_{f},
  \label{eq:macroModel}
\end{equation}
where $\bff(g)=\vint{\be v g}_{\bbR}$ and 
\begin{equation}
  \bF(\bsrho_{f}) = \vint{\be v \cE[\bsrho_{f}]}_{\bbR} 
  = 
  \left(\begin{array}{c}
    n_{f} \, u_{f} \\
    n_{f} \, (\, u_{f}^{2} + \theta_{f} \,) \\
    \frac{1}{2} \, n_{f} \, (\, u_{f}^{2} + 3 \theta_{f} \,) \, u_{f}
  \end{array}\right).
  \label{eq:macroFlux}
\end{equation}
From Eq.~\eqref{eq:gConstraints}, it follows that the first two components of $\bff(g)$ vanish.
However, we will still retain these in the numerical method presented in Section~\ref{sec:dg}. 

Inserting the mM decomposition from Eq.~\eqref{eq:microMacroDecomposition} into Eq.~\eqref{eq:boltzmann} gives the following evolution equation for $g$:
\begin{equation}
  \p_{t}g + v\p_{x}g + E\p_{v}g = \cC_{\lb}[\bsrho_{f}](g) 
  - \big\{\,\p_{t}\cE[\bsrho_{f}] + v\p_{x}\cE[\bsrho_{f}] + E\p_{v}\cE[\bsrho_{f}]\,\big\},
  \label{eq:microModel}
\end{equation}
where we have used the fact, taken from item 3 of Proposition \ref{prop:collisionProperties}, that
\begin{equation}
  \cC_{\lb}[\bsrho_{f}](f) = \cC_{\lb}[\bsrho_{f}](\cE[\bsrho_{f}]+g) = \cC_{\lb}[\bsrho_{f}](\cE[\bsrho_{f}]) + \cC_{\lb}[\bsrho_{f}](g) = \cC_{\lb}[\bsrho_{f}](g).  
  \label{eq:lbCollisionOperator_g}
\end{equation}
As noted in \cite{bennoune_etal_2008,crestetto_etal_2012}, the mM model in Eqs.~\eqref{eq:macroModel} and \eqref{eq:microModel}, with the electric field given by Eq.~\eqref{eq:poisson}, is equivalent to the original system in Eqs.~\eqref{eq:boltzmann}-\eqref{eq:lenardBernstein}, provided that the initial and boundary conditions are compatible.

In this paper, we focus on discretizations which maintain Eq.~\eqref{eq:gConstraints} at the numerical level.  
To illustrate the challenge in doing so, consider the moments of Eq.~\eqref{eq:microModel} after $\vint{\be \p_t \cE[\bsrho_f]}_\bbR = \p_t \bsrho_f$ has been eliminated using Eq.~\eqref{eq:macroModel}:
\begin{align}\label{eq:gMomentsChange}
	\p_{t}\vint{\be g}_{\bbR} 
		=&\underbrace{ \{ \p_x  \bF(\bsrho_f) - \vint{\be v \p_x \cE[\bsrho_f]}_\bbR  \} }_{\rm{I}}
		 + \underbrace{\{ \p_x \bff (g) - \vint{\be v\p_x g}_\bbR  \}}_{\rm{II}} \nonumber \\ 
		&-  \underbrace{\{ ET\bsrho_f - \vint{\be E \p_v \cE[\bsrho_f]}_\bbR \}}_{\rm{III}}
		+\underbrace{ \{\vint{\be E \p_v g}_\bbR \}}_{\rm{IV}}
\end{align}
Each of the integrated quantities above comes from the micro model.  
In the continuum, each of two terms in I, II, and III are equal by definition, while the term in IV is identically zero.  
However, at the discrete level, these indentities may no longer hold, unless the discretizations of Eqs.~\eqref{eq:macroModel} and \eqref{eq:microModel} are designed in a cohesive fashion.  

When solving Eq.~\eqref{eq:microModel} numerically, it is common to apply a projection technique before discretization \cite{bennoune_etal_2008,xiong_etal_2015,gamba_etal_2019}, in order to remove terms that do not contribute to the dynamics of $g$.    
However, because the projection introduces non-conservative products, we do not use it here.  
Instead, we solve Eq.~\eqref{eq:microModel} directly, as was done in \cite{degond_etal_2006} for a kinetic model of neutral gases.


%% file: method.tex
\section{Numerical Method}
\label{sec:dg}

In this section we present two different DG discretizations of the VPLB model.  
We begin in Section~\ref{sec:dgDirect} with the direct discretization of Eq.~\eqref{eq:boltzmann}, which we refer to as to as the \emph{direct method}.  In Section~\ref{sec:dgMM}, we focus on the discetization of the mM model in Eqs.~\eqref{eq:macroModel} and \eqref{eq:microModel}, which we refer to as the \emph{micro-Macro (mM) method}.  
While the direct method serves primarily as a reference for comparison, many elements of the discretization are reused in the mM method.

In both methods, the computational domain $D=D^{x}\times D^{v}$, where $D^{x}=[x_{\min},x_{\max}]$ and $D^{v}=[v_{\min},v_{\max}]$, is divided into non-overlapping elements $I_{ij}=I_{i}^{x}\times I_{j}^{v}$ for $i=1,\ldots,N^{x}$ and $j=1,\ldots,N^{v}$, so that $D=\cup_{i,j=1}^{N^{x},N^{v}}I_{ij}$.  
More specifically, given
\begin{alignat}{3}
  x_{\min} &= x_{\f{1}{2}} < \ldots < x_{i-\f{1}{2}} < x_{i+\f{1}{2}} < \ldots < x_{N^{x}+\f{1}{2}}& &= x_{\max} \\
  v_{\min} &=v_{\f{1}{2}} < \ldots < v_{j-\f{1}{2}} < v_{j+\f{1}{2}} < \ldots < v_{N^{v}+\f{1}{2}} & &= v_{\max},
\end{alignat}
we set
\begin{alignat}{3}
  I_{i}^{x}&=(x_{i-\f{1}{2}},x_{i+\f{1}{2}}), \quad
  \dx_{i}&&=x_{i+\f{1}{2}}-x_{i-\f{1}{2}}, \quad
  x_{i}&&=\f{1}{2}(x_{i-\f{1}{2}}+x_{i+\f{1}{2}})
  \\
    I_{j}^{v}&=(v_{j-\f{1}{2}},v_{j+\f{1}{2}}), \quad
  \dv_{j}&&=v_{j+\f{1}{2}}-v_{j-\f{1}{2}}, \quad
  v_{j}&&=\f{1}{2}(v_{j-\f{1}{2}}+v_{j+\f{1}{2}}).
\end{alignat}
On each phase-space element $I_{ij}$, we define the approximation space
\begin{equation}
  \bbV_{h}^{p} = \big\{\,\varphi_{h}\in L^{2}(D) : \varphi_{h} |_{I_{ij}} \in \bbQ^{p}(I_{ij}),\forall i = 1, \ldots, N^{x} ~ \mbox{and} ~ j = 1, \ldots, N^{v} \,\big\},
\end{equation}
where $\bbQ^{p}(I_{ij})$ is the tensor product space of one-dimensional polynomials of maximal degree $p$.  
More specifically, we employ a nodal DG method \cite{hesthavenWarburton_2008}, where the degrees of freedom are defined at the set of Gauss--Legendre (GL) quadrature points within each phase-space element, and the basis functions for the DG approximation are given by Lagrange polynomials associated with these GL points.  
Let $S_{i}^{x}=\{x_{q}\}_{q=1}^{p+1}\subset I_{i}^{x}$ and $S_{j}^{v}=\{v_{q}\}_{q=1}^{p+1}\subset I_{j}^{v}$ denote the set of GL quadrature points on $I_{i}^{x}$ and $I_{j}^{v}$, respectively.
Then, the degrees of freedom on $I_{ij}$ are defined on the point set $S_{ij}=S_{i}^{x}\otimes S_{j}^{v}\subset I_{ij}$.  
On $I_{i}^{x}$, we let $\{\ell_{q}^{i}(x)\}_{q=1}^{p+1}$ denote Lagrange polynomials of degree $p$ constructed from the points $S_{i}^{x}$.  
Similarly, on $I_{j}^{v}$, $\{\ell_{q}^{j}(v)\}_{q=1}^{p+1}$ denote Lagrange polynomials of degree $p$ constructed from the points $S_{j}^{v}$. 
Then, $\{\ell_{q}^{i}(x)\ell_{r}^{j}(v)\}_{q,r=1}^{p+1}\subset \bbQ_{h}^{p}(I_{ij})$.  

\subsection{Direct Method for the VPLB Model}
\label{sec:dgDirect}

In the direct method, we seek a semi-discrete  solution $f_{h}\in C([0,\infty);\bbV_{h}^{p})$ that satisfies
\begin{equation}
  (\p_{t}f_{h},\varphi_{h})_{ij}
  +B_{h}(f_{h},\bsrho_{f_{h}},E_{h},\varphi_{h})_{ij}=0
  \label{eq:boltzmannSemiDiscreteDG}
\end{equation}
for all $\varphi_{h}\in\bbV_{h}^{p}$ and all $I_{ij}\in D$.  
In Eq.~\eqref{eq:boltzmannSemiDiscreteDG}, 
\begin{equation}
  (a,b)_{ij} = \int_{I_{ij}}a\,b\,dx\,dv, \quad \forall a,b\in\bbV_{h}^{p};
\end{equation}
$ \bsrho_{f_{h}} = \vint{\be f_{h}}_{D^{v}}$;
$E_{h}$ is the approximation to the electric field (prescribed in Section \ref{sec:poisson}); and
\begin{equation}
  B_{h}(f_{h},\bsrho_{f_{h}},E_{h},\varphi_{h})_{ij}
  =B_{h}^{\vp}(f_{h},E_{h},\varphi_{h})_{ij} + \nu\,B_{h}^{\lb}(f_{h},\bsrho_{f_{h}},\varphi_{h})_{ij},
  \label{eq:vplbBLF}
\end{equation}
where the `Vlasov form' is given by
\begin{align}
  &B_{h}^{\vp}(f_{h},E_{h},\varphi_{h})_{ij} \label{eq:vpBLF} \\
  &=
  \int_{I_{j}^{v}}
  \big[\,
    \widehat{vf_{h}}(x_{i+\f{1}{2}},v)\varphi_{h}(x_{i+\f{1}{2}}^{-},v) 
    - \widehat{vf_{h}}(x_{i-\f{1}{2}},v)\varphi_{h}(x_{i-\f{1}{2}}^{+},v)
  \,\big]\,dv
  -(vf_{h},\p_{x}\varphi_{h})_{ij} \nonumber \\
  &\hspace{12pt}
  +\int_{I_{i}^{x}}
  \big[\,
    \widehat{E_{h}f_{h}}(x,v_{j+\f{1}{2}})\varphi_{h}(x,v_{j+\f{1}{2}}^{-})
    -\widehat{E_{h}f_{h}}(x,v_{j-\f{1}{2}})\varphi_{h}(x,v_{j-\f{1}{2}}^{+})
  \,\big]\,dx
  -(E_{h}f_{h},\p_{v}\varphi_{h})_{ij} \nonumber
\end{align}
and the `Lenard--Bernstein form' is given by
\begin{align}
  &B_{h}^{\lb}(f_{h},\bsrho_{f_{h}},\varphi_{h})_{ij} \label{eq:lbBLF} \\
  &=\int_{I_{i}^{x}}
  \big[\,
    \widehat{w_{f_{h}}f_{h}}(x,v_{j+\f{1}{2}})\,\varphi_{h}(x,v_{j+\f{1}{2}}^{-})
    -\widehat{w_{f_{h}}f_{h}}(x,v_{j-\f{1}{2}})\,\varphi_{h}(x,v_{j-\f{1}{2}}^{+})
  \,\big]\,dx
  - (w_{f_{h}}f_{h},\p_{v}\varphi_{h})_{ij} \nonumber \\
  &\hspace{12pt}
  -\int_{I_{i}^{x}}\theta_{f_{h}}\,
  \big\{\,
    \big[\,
      \p_{v}\mathfrak{f}_{h}(x,v_{j+\f{1}{2}})\,\varphi_{h}(x,v_{j+\f{1}{2}}^{-})
      -\mathfrak{f}_{h}(x,v_{j+\f{1}{2}})\,\p_{v}\varphi_{h}(x,v_{j+\f{1}{2}}^{-})
    \,\big] \nonumber \\
    &\hspace{32pt}
    -\big[\,
      \p_{v}\mathfrak{f}_{h}(x,v_{j-\f{1}{2}})\,\varphi_{h}(x,v_{j-\f{1}{2}}^{+})
      -\mathfrak{f}_{h}(x,v_{j-\f{1}{2}})\,\p_{v}\varphi_{h}(x,v_{j-\f{1}{2}}^{+})
    \,\big]
  \,\big\}\,dx
  -(\theta_{f_{h}}f_{h},\p_{vv}\varphi_{h})_{ij}.  \nonumber
\end{align}
Following \cite{hakim_etal_2020}, we arrive at Eq.~\eqref{eq:lbBLF} after integrating the diffusive term by parts twice.  
This strategy ensures that derivatives in the volume term (the last term on the right-hand side of Eq.~\eqref{eq:lbBLF}) act only on the test function $\varphi_{h}$ and, as a result, maintains energy conservation.  
(See Section \ref{sec:conservation} below for further details on conservation properties.)

We use the upwind flux to evaluate the numerical fluxes for the Vlasov operator in Eq.~\eqref{eq:vpBLF}.  
Specifically, we let
\begin{align}
  \widehat{vf_{h}}(x,v) 
  &= \hat{v}^{+}(v)\,f_{h}(x^{-},v)+\hat{v}^{-}(v)\,f_{h}(x^{+},v), \label{eq:upwindFluxPosition} \\
  \widehat{E_{h}f_{h}}(x,v) 
  &= \hat{E}_{h}^{+}(x)\,f_{h}(x,v^{-})+\hat{E}_{h}^{-}(x)\,f_{h}(x,v^{+}), \label{eq:upwindFluxVelocity}
\end{align}
where
\begin{equation}
  \hat{v}^{\pm}(v) = \frac{v\pm|v|}{2}
  \quad\text{and}\quad
  \hat{E}_{h}^{\pm}(x) = \frac{E_{h}(x)\pm|E_{h}(x)|}{2}  
  \label{eq:upwindVandE}
\end{equation}
and $x^{\pm}=x\pm\lim_{\delta\to0^{+}}\delta$ and $v^{\pm}=v\pm\lim_{\delta\to0^{+}}\delta$.  
Similarly, we use the upwind flux to evaluate the drift term in the LB operator in Eq.~\eqref{eq:lbBLF}
\begin{equation}
  \widehat{w_{f_{h}}f_{h}}(x,v)
  =\hat{w}_{f_{h}}^{+}(x,v)\,f_{h}(x,v^{-}) + \hat{w}_{f_{h}}^{-}(x,v)\,f_{h}(x,v^{+}),
\end{equation}
where $w_{f_{h}}=u_{f_{h}}-v$ and
\begin{equation}
  \hat{w}_{f_{h}}^{\pm}(x,v)= \frac{w_{f_{h}}(x,v)\pm|w_{f_{h}}(x,v)|}{2}.
  \label{eq:upwindW}
\end{equation}

Following \cite{vanLeerNomura_2005,hakim_etal_2020}, the diffusion term in the LB operator is evaluated with a \emph{recovery polynomial}, denoted by $\mathfrak{f}_{h}$ (i.e., with fraktur font), which is used to compute the distribution and its velocity gradient at an interface.
Let
\begin{equation}
  I_{i[j,j+1]} = I_{ij} \cup I_{ij+1}
  =\big\{\,(x,v) ~ | ~ x \in I_{i}^{x} ~ \text{and} ~ v \in I_{j}^{v} \cup I_{j+1}^{v} \equiv I_{[j,j+1]}^{v}\,\big\}.  
\end{equation}
The recovery polynomial on $I_{i[j,j+1]}$ is then constructed from the tensor product of one dimensional polynomials of maximal degree $p$ and $2p+1$ in the $x$- and $v$-dimensions, respectively.  
We denote this approximation space by
\begin{equation}
  \bbW_{h}^{2p+1}
  =\big\{\,\Psi_{h}\in L^{2}(D) : \Psi_{h}|_{I_{i[j,j+1]}}\in\bbP^{p}(I_{i}^{x})\otimes\bbP^{2p+1}(I_{[j,j+1]}^{v})\,\big\}, 
\end{equation}
where $\bbP^{q}$ is the space of one-dimensional polynomials of maximal degree $q$.  
On $I_{i[j,j+1]}$, $\mathfrak{f}_{h}\in C([0,\infty);\bbW_{h}^{2p+1})$ is  obtained from the DG solution $f_{h}$ by requiring that \cite{vanLeerNomura_2005}
\begin{equation}
  \int_{I_{ij}}f_{h}\,\varphi_{h}\,dz = \int_{I_{ij}}\mathfrak{f}_{h}\,\varphi_{h}\,dz
  \quad\text{and}\quad
  \int_{I_{ij+1}}f_{h}\,\varphi_{h}\,dz =  \int_{I_{ij+1}}\mathfrak{f}_{h}\,\varphi_{h}\,dz
  \label{eq:recoveryPolynomialEquivalence}
\end{equation}
for all $\varphi_{h}\in\bbV_{h}^{p}$.  
Then $\mathfrak{f}_{h}$ and $\p_{v}\mathfrak{f}_{h}$ are continous at the interface between $I_{j}^{v}$ and $I_{j+1}^{v}$.  

In this paper, when evaluating the Vlasov form in Eq.~\eqref{eq:vpBLF}, we impose the following zero flux conditions at the boundaries of the velocity domain (e.g., \cite{ayuso_etal_2011})
\begin{equation}
  \widehat{E_{h}f_{h}}\big|_{v=v_{\min}}=\widehat{E_{h}f_{h}}\big|_{v=v_{\max}}=0.  
  \label{eq:vpBoundaryCondition}
\end{equation}
When evaluating the LB form in Eq.~\eqref{eq:lbBLF}, we impose the following conditions at the boundaries of the velocity domain
\begin{equation}
  \big[\,\widehat{w_{f_{h}}f_{h}} - \theta_{f_{h}}\,\p_{v}\mathfrak{f}_{h}\,\big] \big|_{v=v_{\min}}
  =\big[\,\widehat{w_{f_{h}}f_{h}} - \theta_{f_{h}}\,\p_{v}\mathfrak{f}_{h}\,\big] \big|_{v=v_{\max}}=0
  \label{eq:lbBoundaryCondition_1}
\end{equation}
and
\begin{equation}
  \mathfrak{f}_{h}\big|_{v=v_{\min}} = \mathfrak{f}_{h}\big|_{v=v_{\max}} = 0.
  \label{eq:lbBoundaryCondition_2}
\end{equation}

\begin{rem}
  For simplicity, we will use a linear finite element method to solve Eq.~\eqref{eq:poisson} for the approximation to the electrostatic potential $\Phi_{h}$, and the electric field is obtained by direct differentiation $E_{h}=-\p_{x}\Phi_{h}$ (see Section~\ref{sec:poisson}).  
  Hence, $E_{h}$ in Eq.~\eqref{eq:boltzmannSemiDiscreteDG} is approximated as a constant within each spatial element.  
  \label{rem:constantElectricField}
\end{rem}

\subsection{DG Method for the Micro-Macro decomposition of the VPLB Model}
\label{sec:dgMM}

In this subsection we specify the DG discretization for the mM method.  
We start with the macro component, given by Eq.~\eqref{eq:macroModel} in Section~\ref{sec:dgMacro}, and then provide the DG discretization for the micro component, given by Eq.~\eqref{eq:microModel}, in Section~\ref{sec:dgMicro}.  
(Technically, the discretization of the two components go together because of coupling terms, but in this paper, as a matter of organization, we separate the specification of each component.)  

\subsubsection{DG Method for the Macro Component}
\label{sec:dgMacro}

To discretize Eq.~\eqref{eq:macroModel} in space with the DG method, we let the approximation space on the spatial element $I_{i}^{x}$ be denoted 
\begin{equation}
  \bbV_{h}^{x,p}=\big\{\,\varphi_{h}\in L^{2}(D^{x}) : \varphi_{h}|_{I_{i}^{x}}\in\bbP^{p}(I_{i}^{x}),\forall i = 1,\ldots,N^{x}\,\big\},
\end{equation}
where $\bbP^{p}(I_{i}^{x})$ is the space of piecewise polynomials on $I_{i}^{x}$ of maximal degree $p$.  
The semi-discrete DG problem is then to find $\bsrho_{f,h}\in C([0,\infty);[\bbV_{h}^{x,p}]^{3})$ such that
\begin{equation}
  (\p_{t}\bsrho_{f,h},\varphi_{h})_{i} + B_{h}^{\macro}(\bsrho_{f,h},g_{h},E_{h},\varphi_{h})_{i} = 0
  \label{eq:macroSemiDiscrete}
\end{equation}
holds for all $\varphi_{h}\in\bbV_{h}^{x,p}$, all $I_{i}^{x}\in D^{x}$, and where $E_{h}\in\bbV_{h}^{x,0}$ (see Remark~\ref{rem:constantElectricField}), and $g_{h}\in\bbV_{h}^{p}$.  
(Here it is understood that $g_{h}$ is obtained by solving Eq.~\eqref{eq:microSemiDiscrete}, given below.)  
The `macro form' $B_{h}^{\macro}$ in Eq.~\eqref{eq:macroSemiDiscrete} is given by
\begin{align}
  &B_{h}^{\macro}(\bsrho_{f,h},g_{h},E_{h},\varphi_{h})_{i} \nonumber \\
  &=
  \big[\,
    \widehat{\bF(\bsrho_{f,h})}(x_{i+\f{1}{2}})\,\varphi_{h}(x_{i+\f{1}{2}}^{-})
    -\widehat{\bF(\bsrho_{f,h})}(x_{i-\f{1}{2}})\,\varphi_{h}(x_{i-\f{1}{2}}^{+})
  \,\big]
  -(\bF(\bsrho_{f,h}),\p_{x}\varphi_{h})_{i} \nonumber \\
  &\hspace{12pt}
  +\big[\,
    \widehat{\bff(g_{h})}(x_{i+\f{1}{2}})\,\varphi_{h}(x_{i+\f{1}{2}}^{-})
    -\widehat{\bff(g_{h})}(x_{i-\f{1}{2}})\,\varphi_{h}(x_{i-\f{1}{2}}^{+})
  \,\big]
  -(\bff(g_{h}),\p_{x}\varphi_{h})_{i} \nonumber \\
  &\hspace{12pt}
  -(E_{h}T\bsrho_{f,h},\varphi_{h})_{i},
  \label{eq:macroBLF}
\end{align}
where
\begin{equation}
  (a,b)_{i} = \int_{I_{i}^{x}}a\,b\,dx,  \quad \forall a,b\in\bbV_{h}^{x,p}.  
\end{equation}

We note a slight abuse of notation in Eq.~\eqref{eq:macroSemiDiscrete}, where it is understood that the semi-discrete DG problem holds independently for each component of $\bsrho_{f,h}$.  

In Eq.~\eqref{eq:macroBLF}, the numerical flux $\widehat{\bF(\bsrho_{f,h})}$ is computed via upwinding at the kinetic level; that is,
\begin{align}
  \widehat{\bF(\bsrho_{f,h})}(x)
  &=\f{1}{2}\big[\,\bF\big(\bsrho_{f,h}(x^{-})\big)+\bF\big(\bsrho_{f,h}(x^{+})\big)\,\big] \nonumber \\
  &\hspace{12pt}
  -\f{1}{2}\int_{\bbR}\be\,|v|\,\big(\,\cE[\bsrho_{f,h}(x^{+})]-\cE[\bsrho_{f,h}(x^{-})]\,\big)\,dv, 
  \label{eq:numercalFluxMacroMacroA}
\end{align}
and the integrals in the dissipation term (the second line on the right-hand side of Eq.~\eqref{eq:numercalFluxMacroMacroA}) are evaluated analytically:
\begin{align}
    &\int_{\bbR}e_{0}\,|v|\,\cE[\bsrho_{f,h}(x)]\,dv 
    =\f{n_{f,h}}{\sqrt{2\pi\theta_{f,h}}}\,\Big\{\,2\,\theta_{f,h}\,e^{-\f{u_{f,h}^{2}}{2\theta_{f,h}}}+u_{f,h}\,\sqrt{2\pi\theta_{f,h}}\,\mbox{erf}\Big(\f{u_{f,h}}{\sqrt{2\theta_{f,h}}}\Big)\,\Big\}, \\
    &\int_{\bbR}e_{1}\,|v|\,\cE[\bsrho_{f,h}(x)]\,dv \nonumber \\
  &\quad=\f{n_{f,h}}{\sqrt{2\pi\theta_{f,h}}}\,\Big\{\,2\,u_{f,h}\,\theta_{f,h}\,e^{-\f{u_{f,h}^{2}}{2\theta_{f,h}}}+\big[\,u_{f,h}^{2}+\theta_{f,h}\,\big]\,\sqrt{2\pi\theta_{f,h}}\,\mbox{erf}\Big(\f{u_{f,h}}{\sqrt{2\theta_{f,h}}}\Big)\,\Big\}, \\
    &\int_{\bbR}e_{2}\,|v|\,\cE[\bsrho_{f,h}(x)]\,dv \nonumber \\
    &\quad =\f{n_{f,h}}{\sqrt{2\pi\theta_{f,h}}}\,\Big\{\,2\,\theta_{f,h}\,\big(\theta_{f,h}+\f{1}{2}\,u_{f,h}^{2}\big)\,e^{-\f{u_{f,h}^{2}}{2\theta_{f,h}}}+\f{1}{2}\big[\,u_{f,h}^{2}+3\theta_{f,h}\,\big]\,u_{f,h}\,\sqrt{2\pi\theta_{f,h}}\,\mbox{erf}\Big(\f{u_{f,h}}{\sqrt{2\theta_{f,h}}}\Big)\,\Big\}
    \end{align}
where $n_{f,h}$, $u_{f,h}$, and $\theta_{f,h}$ are obtained from $\bsrho_{f,h}(x)$ via the bijection defined in Eq.~\eqref{eq:momentFluidBijection}. 
Similarly, the numerical flux $\widehat{\bff(g_{h})}$ is given by (note that $g_{h}$ has compact support on $D^{v}$)
\begin{align}
  \widehat{\bff(g_{h})}(x) 
  = \int_{\bbR} \be\,\widehat{vg_{h}}(x,v)\,dv
  &= \int_{D^{v}} \be\,\widehat{vg_{h}}(x,v)\,dv \nonumber \\
  &= \int_{D^{v}} \be\,\hat{v}^{+}\,g_{h}(x^{-},v)\,dv + \int_{D^{v}} \be\,\hat{v}^{-}\,g_{h}(x^{+},v)\,dv,
  \label{eq:numercalFluxMacroMicro}
\end{align}
where $\hat{v}^{\pm}$ is defined as in Eq.~\eqref{eq:upwindVandE}.  
These integrals can be evaluated exactly with GL quadrature, as long as $v=0$ coincides with and interface (so that $\hat{v}^{\pm}$ is polynomial in each element).  

\subsubsection{DG Method for the Micro Component}
\label{sec:dgMicro}

The semi-discrete DG problem for the micro component can be formulated as follows: find $g_{h}\in C([0,\infty);\bbV_{h}^{p})$ such that
\begin{equation}
  (\p_{t}g_{h},\varphi_{h})_{ij} + B_{h}(g_{h},\bsrho_{f,h},E_{h},\varphi_{h}) + (\p_{t}\cE[\bsrho_{f,h}],\varphi_{h})_{ij} + B_{h}^{\micro}(\bsrho_{f,h},E_{h},\varphi_{h})_{ij} = 0
  \label{eq:microSemiDiscrete}
\end{equation}
for all $\varphi_{h}\in\bbV_{h}^{p}$, $I_{ij}\in D$, $E_{h}\in\bbV_{h}^{x,0}$, and $\bsrho_{f,h}\in\bbV_{h}^{x,p}$ (obtained from Eq.~\eqref{eq:macroSemiDiscrete}),
where $B_{h}$ is defined in Eq.~\eqref{eq:vplbBLF} and
\begin{align}
  &B_{h}^{\micro}(\bsrho_{f,h},E_{h},\varphi_{h})_{ij} \nonumber \\
  &=\int_{I_{j}^{v}}
  \big[\,
    \widehat{v\cE}[\bsrho_{f,h}](x_{i+\f{1}{2}},v)\varphi_{h}(x_{i+\f{1}{2}}^{-},v)
    -\widehat{v\cE}[\bsrho_{f,h}](x_{i-\f{1}{2}},v)\varphi_{h}(x_{i-\f{1}{2}}^{+},v)
  \,\big] dv \nonumber \\
  &\hspace{32pt}
  - (v\cE[\bsrho_{f,h}],\p_{x}\varphi_{h})_{ij} \nonumber \\
  &\hspace{12pt}
  +\int_{I_{i}^{x}}
  \big[\,
    \widehat{E_{h}\cE}[\bsrho_{f,h}](x,v_{j+\f{1}{2}})\varphi_{h}(x,v_{j+\f{1}{2}}^{-})
    -\widehat{E_{h}\cE}[\bsrho_{f,h}](x,v_{j-\f{1}{2}})\varphi_{h}(x,v_{j-\f{1}{2}}^{+})
  \,\big]\,dx \nonumber \\
  &\hspace{32pt}
  - (E_{h}\cE[\bsrho_{f,h}],\p_{v}\varphi_{h})_{ij}.  
  \label{eq:microBLF}
\end{align}
The numerical flux $\widehat{v\cE}[\bsrho_{f,h}]$ is prescribed using upwinding at the kinetic level:
\begin{align}
  \widehat{v\cE}[\bsrho_{f,h}](x,v)
  &=\hat{v}^{+}(v)\,\cE[\bsrho_{f,h}(x^{-})](v)+\hat{v}^{+}(v)\,\cE[\bsrho_{f,h}(x^{+})](v).  
  \label{eq:numericalFluxMicroMacroA}
\end{align}
Since the equilibrium distribution is continuous in velocity, the numerical flux $\widehat{E_{h}\cE}[\bsrho_{f,h}]$ is simply evaluated as
\begin{equation}
  \widehat{E_{h}\cE}[\bsrho_{f,h}](x,v) = E_{h}(x)\cE[\bsrho_{f,h}(x)](v).  
\end{equation}

In Eq.~\eqref{eq:microBLF}, we impose the following conditions at the boundary of the velocity domain
\begin{equation}
  \widehat{E_{h}\cE}[\bsrho_{f,h}]\big|_{v=v_{\min}} = \widehat{E_{h}\cE}[\bsrho_{f,h}]\big|_{v=v_{\max}} = 0.  
  \label{eq:velocityBoundaryConditionMaxwellian}
\end{equation}
Due to the exponential decay of $\cE[\bsrho_{f,h}]$ in velocity, these conditions are reasonable, provided the velocity domain is large enough.  

For the Vlasov and LB forms evaluated with the micro distribution $g_{h}$, represented by the second term on the left-hand side of Eq.~\eqref{eq:microSemiDiscrete}, we impose the conditions in Eq.~\eqref{eq:vpBoundaryCondition} and Eqs.~\eqref{eq:lbBoundaryCondition_1}-\eqref{eq:lbBoundaryCondition_2}, respectively, with $f_{h}$ and $\mathfrak{f}_{h}$ replaced by $g_{h}$ and $\mathfrak{g}_{h}$.  

\begin{rem}
 The subtle notational difference between $\bsrho_{f_{h}}$ and $\bsrho_{f,h}$ is important.  Both terms approximate the same quantity, but the former is computed by taking velocity moments of $f_{h}$ in the direct method, while  the latter is evolved  by  Eq.~\eqref{eq:macroSemiDiscrete} in the mM method.
\end{rem}

\begin{rem}
  To arrive at a numerical analogue of Eq.~\eqref{eq:gMomentsChange}, and achieve the equivalent of the vanishing of terms {\rm I} and {\rm III} to preserve the constraints in Eq.~\eqref{eq:gConstraints}, we use exact evaluation of velocity integrals in Eq.~\eqref{eq:microSemiDiscrete} --- importantly those involving Maxwellians --- to properly balance terms that emanate from Eq.~\eqref{eq:macroSemiDiscrete}.  
  (Spatial integrals involving Maxwellians in Eq.~\eqref{eq:microSemiDiscrete} are approximated with $(p+1)$-point GL quadratures.)  
  In Section~\ref{sec:conservation}, we discuss the importance of this balance for maintaining conservation laws.  
  In Section~\ref{sec:numerical}, we demonstrate numerical artifacts that can arise when instead using standard quadrature formulas.
\end{rem}


\subsection{Poisson Solver}
\label{sec:poisson}

We use a standard finite element method (FEM) to solve Eq.~\eqref{eq:poisson}.  
(Specifically, as noted in Remark~\ref{rem:constantElectricField}, mainly for simplicity, we use the linear FEM.)  
To this end, we let $V_{h}$ denote space of functions constructed from basis functions of the form
\begin{equation}
  \psi_{i}(x)
  =\left\{
  \begin{array}{cc}
  \f{(x-x_{i-\f{1}{2}})}{\dx_{i}} & x \in I_{i}^{x}, \\
  \f{(x_{i+\f{3}{2}}-x)}{\dx_{i+1}} & x \in I_{i+1}^{x}, \\
  0 & \mbox{otherwise},
  \end{array}
  \right.
  \quad
  i=1,\ldots,N^{x}-1.
  \label{eq:spaceFEM}
\end{equation}
For the mM method, let $S(\bsrho_{f,h})=n_{f,h}-n_{\rm{e}}$.  
We seek $\Phi_{h}\in V_{h}$ such that
\begin{equation}
  \int_{D^{x}}(\p_{x}\Phi_{h})\,(\p_{x}\psi_{h})\,dx
  =\int_{D^{x}}S(\bsrho_{f,h})\,\psi_{h}\,dx
  \label{eq:poissonFEM}
\end{equation}
holds for all $\psi_{h}\in V_{h}$.  
The electric field is then given by
\begin{equation}
  E_{h}|_{I_{i}^{x}}=-(\p_{x}\Phi_{h})|_{I_{i}^{x}} = - \Big(\f {\Phi_h(x_{i+1/2})-\Phi_h(x_{i-1/2})} {\dx_{i}}\Big)\in\bbV_{h}^{x,0}.  
  \label{eq:electricField}
\end{equation}

When solving the kinetic equation with the direct method, the source $S(\bsrho_{f,h})$ in Eq.~\eqref{eq:poissonFEM} is replaced with $S(\bsrho_{f_{h}})=n_{f_{h}}-n_{\rm{e}}$.  

\subsection{Time Integration}
\label{sec:timeIntegration}

Both the direct and mM methods yield systems of ordinary differential equations (ODEs) that must be solved numerically with a time-stepping method. 
We use explicit strong stability-preserving Runge--Kutta (SSP-RK) methods \cite{shuOsher_1988,gottlieb_etal_2001} for problems without collisions ($\nu=0$) and implicit-explicit Runge--Kutta (IMEX-RK) methods \cite{ascher_etal_1997,pareschiRusso_2005} for problems with collisions ($\nu>0$).  
In the latter case, the collision operator is evaluated implicitly to avoid severe stability restrictions on the time step when the collision frequency is large.  

%

\subsubsection{IMEX-RK Time Integration for the Direct Method}
\label{sec:imexrkDirect}

The general $s$-stage IMEX-RK method to evolve the VPLB system with the direct method in Section~\ref{sec:dgDirect} from $t^{k}$ to $t^{k+1}=t^{k}+\dt$, where $\dt$ is the time step, can be written as \cite{pareschiRusso_2005}: for $i=1,\ldots,N^{x}$, $j=1,\ldots,N^{v}$, and all $\varphi_{h}\in\bbV_{h}^{p}$
\begin{itemize}
  \item[1.] for $l=1,\ldots,s$ compute
  \begin{equation}
    (f_{h}^{(l)},\varphi_{h})_{ij} 
    =(f_{h}^{(l\star)},\varphi_{h})_{ij}
    - a_{ll}\dt\nu B_{h}^{\lb}(f_{h}^{(l)},\bsrho_{f_{h}}^{(l)},\varphi_{h})_{ij},
    \label{eq:imexImplicit}
  \end{equation}
  where
  \begin{align}
    &(f_{h}^{(l\star)},\varphi_{h})_{ij}
    =(f_{h}^{k},\varphi_{h})_{ij} \nonumber \\
    &\hspace{12pt}
    - \dt\sum_{m=1}^{l-1}
    \Big(
      \tilde{a}_{lm}B_{h}^{\vp}(f_{h}^{(m)},E_{h}^{(m)},\varphi_{h})_{ij} 
      +a_{lm}\nu B_{h}^{\lb}(f_{h}^{(m)},\bsrho_{f_{h}}^{(m)},\varphi_{h})_{ij}
    \Big),
    \label{eq:imexExplicit}
  \end{align}
  and solve Eq.~\eqref{eq:poissonFEM} to obtain $E_{h}^{(l)}|_{I_{i}^{x}}$ from $(f_{h}^{(l)},\varphi_{h})_{ij}$.
  \item[2.] Assemble
  \begin{align}
    &(f_{h}^{k+1},\varphi_{h})_{ij}
    =(f_{h}^{k},\varphi_{h})_{ij} \nonumber \\
    &\hspace{12pt}
    - \dt\sum_{m=1}^{s}
    \Big(
      \tilde{w}_{l}B_{h}^{\vp}(f_{h}^{(l)},E_{h}^{(l)},\varphi_{h})_{ij} 
      +w_{l}\nu B_{h}^{\lb}(f_{h}^{(l)},\bsrho_{f_{h}}^{(l)},\varphi_{h})_{ij}
    \Big),
    \label{eq:imexAssembly}
  \end{align}
  and solve Eq.~\eqref{eq:poissonFEM} to obtain $E_{h}^{k+1}|_{I_{i}^{x}}$ from $(f_{h}^{k+1},\varphi_{h})_{ij}$.
\end{itemize}
Here, the coefficients $\tilde{a}_{lm},a_{lm}$ are components of matrices $\tilde{A},A\in\bbR^{s\times s}$, while the coefficients $\tilde{w}_{l},w_{l}$ are components of vectors $\tilde{\bw},\bw\in\bbR^{s}$.  
IMEX-RK schemes are commonly represented by a double tableau of the form
\begin{equation}
  \begin{array}{c | c}
    \tilde{\bc} & \tilde{A} \\ \hline
    & \tilde{\bw}
  \end{array}
  \qquad
  \begin{array}{c | c}
    \bc & A \\ \hline
    & \bw
  \end{array},
\end{equation}
where the coefficients $\tilde{\bc}$ and $\bc$ are used for non-autonomous systems.  
Here we only consider diagonally implicit IMEX schemes, where $\tilde{a}_{lm}=0$ for $m\ge l$ and $a_{lm}=0$ for $m>l$ (obvious from the upper limit of the sum in Eq.~\eqref{eq:imexExplicit}).  
Each implicit solve from Eq.~\eqref{eq:imexImplicit} can be considered as a backward Euler update with initial state $f_{h}^{(l\star)}$ and step size $a_{ll}\dt$.  
For the collisionless case ($\nu=0$), we set the implicit coefficients, $a_{lm}$ and $w_{l}$, to zero, while we set the explicit coefficients, $\tilde{a}_{lm}$ and $\tilde{w}_{l}$, appropriately to obtain either the optimal second- or third-order accurate SSP-RK methods of \cite{shuOsher_1988} (henceforth referred to as SSP-RK2 and SSP-RK3, respectively).  
Finally, we note a special class of IMEX-RK schemes, called globally stiffly accurate (GSA) schemes, where $a_{sm}=w_{m}$ and $\tilde{a}_{sm}=\tilde{w}_{m}$, for $m=1,\ldots,s$.  
For GSA IMEX-RK schemes, which we use exclusively for problems involving collisions, the assembly step in Eq.~\eqref{eq:imexAssembly} can be dropped, and we have $f_{h}^{n+1}=f_{h}^{(s)}$.  

\begin{rem}
  The use of unknown moments $\bsrho_{f_{h}}^{(l)}$ in the LB operator in the implicit step in Eq.~\eqref{eq:imexImplicit} appears to imply that a nonlinear solve is needed.  
  However, because the conservation properties of the LB operator hold at the discrete level (i.e., $\bsrho_{f_{h}}^{(l)}=\bsrho_{f_{h}}^{(l\star)}$), the known moments $\bsrho_{f_{h}}^{(l\star)}$ can be used in place of $\bsrho_{f_{h}}^{(l)}$.  
  See Section~\ref{sec:conservation} for more details.  
\end{rem}



\subsubsection{IMEX-RK Time Integration for the mM Method}
\label{imexrkMM}

For the mM decomposition of the VPLB system, following the previous section, we employ an $s$-stage GSA IMEX-RK method, which can be formulated as: for $i=1,\ldots,N^{x}$, $j=1,\ldots,N^{v}$, all $\varphi_{h}\in\bbV_{h}^{p}$, and all $\psi_{h}\in\bbV_{h}^{x,p}$
\begin{itemize}
  \item[1.] For $l=1,\ldots,s$ compute
  \begin{align}
    (\bsrho_{f,h}^{(l)},\psi_{h})_{i}
    &= (\bsrho_{f,h}^{k},\psi_{h})_{i} - \dt\sum_{m=1}^{l-1}\tilde{a}_{lm}B_{h}^{\macro}(\bsrho_{f,h}^{(m)},g_{h}^{(m)},E_{h}^{(m)},\psi_{h})_{i}, \label{eq:imexMacroExplicit}\\
    (g_{h}^{(l)},\varphi_{h})_{ij}
    &=\Lambda\big\{(g_{h}^{(l\star)},\varphi_{h})_{ij}\big\} - a_{ll} \dt \nu B_{h}^{\lb}(g_{h}^{(l)},\bsrho_{f,h}^{(l)},\varphi_{h})_{ij}, \label{eq:imexMicroImplicit}
  \end{align}
  where
  \begin{align}
    &(g_{h}^{(l\star)},\varphi_{h})_{ij}
    =(g_{h}^{k},\varphi_{h})_{ij} + \big[\,(\cE[\bsrho_{f,h}^{k}],\varphi_{h})_{ij}-(\cE[\bsrho_{f,h}^{(l)}],\varphi_{h})_{ij}\,\big] \nonumber \\
    &\hspace{0pt}
    -\dt\sum_{m=1}^{l-1}
    \Big(\,
      \tilde{a}_{lm}B_{h}^{\vp}(g_{h}^{(m)},E_{h}^{(m)},\varphi_{h})_{ij}
      +a_{lm}\nu B_{h}^{\lb}(g_{h}^{(m)},\bsrho_{f,h}^{(m)},\varphi_{h})_{ij}
    \,\Big) \nonumber \\
    &\hspace{0pt}
    -\dt\sum_{m=1}^{l-1}\tilde{a}_{lm}B_{h}^{\micro}(\bsrho_{f,h}^{(m)},E_{h}^{(m)},\varphi_{h})_{ij},
    \label{eq:imexMicroExplicit}
  \end{align}
  and solve Eq.~\eqref{eq:poissonFEM} to obtain $E_{h}^{(l)}|_{I_{i}^{x}}$ from $\bsrho_{f,h}^{(l)}$.
  \item[2.] Set $(\bsrho_{h}^{k+1},\psi_{h})_{i}=(\bsrho_{h}^{(s)},\psi_{h})_{i}$, $(g_{h}^{k+1},\varphi_{h})_{ij}=(g_{h}^{(s)},\varphi_{h})_{ij}$, and $E_{h}^{k+1}=E_{h}^{(s)}|_{I_{i}^{x}}$.
\end{itemize}

In Eq.~\eqref{eq:imexMicroImplicit}, $\Lambda\big\{\big\}$ represents a `cleaning limiter', which, if needed, is applied to $g_{h}^{(l\star)}$ after the explicit push in Eq.~\eqref{eq:imexMicroExplicit} to enforce the orthogonality constraints in Eq.~\eqref{eq:gConstraints}.  
We discuss this limiter further in Section~\ref{sec:cleaningLimiter}.  

\section{Conservation Properties}
\label{sec:conservation}

We consider conservation properties of the discretized VPLB system in this section.  
For simplicity, we consider a first-order accurate time integration scheme consisting of a sequence of forward and backward Euler steps, which can be written in the standard IMEX-RK form with Butcher tables (scheme ARS111 from \cite{ascher_etal_1997})
\begin{equation}
  \begin{array}{c | c c}
  	\tilde{c}_{1} & \tilde{a}_{11} & \tilde{a}_{12}  \\
  	\tilde{c}_{2} & \tilde{a}_{21} & \tilde{a}_{22}  \\ \hline
  	                   & \tilde{w}_{1} & \tilde{w}_{2}
  \end{array}
  =
  \begin{array}{c | c c}
  	0 & 0 & 0  \\
  	1 & 1 & 0  \\ \hline
  	   & 1 & 0
  \end{array}
  \qquad
    \begin{array}{c | c c}
  	c_{1} & a_{11} & a_{12}  \\
  	c_{2} & a_{21} & a_{22}  \\ \hline
  	         & w_{1} & w_{2}
  \end{array}
  =
  \begin{array}{c | c c}
  	0 & 0 & 0          \\
  	1 & 0 & 1  \\ \hline
  	   & 0 & 1
  \end{array}.
  \label{eq:butcherARS111}
\end{equation}
(The extension of the results in this section to the more general IMEX schemes in Sections~\ref{sec:imexrkDirect} and \ref{imexrkMM} is relatively straightforward, but notationally tedious.)  
We demand that the polynomial degree used in the DG method is at least $2$, so that all the components of $\be$ can be represented exactly by the approximation space.  
For completeness, we consider the general case with $E\ne0$.  
However, the presence of the electric field introduces momentum and total energy conservation violations from the explicit step, and we only achieve exact momentum and energy conservation for the case with $E=0$.  
We will demonstrate conservation properties numerically in Section~\ref{sec:numerical}, where we also compare total energy conservation properties of the direct and mM methods.  

\subsection{Conservation Properties of the Direct Method for the VPLB System}
\label{sec:conservationDirect}

With the Butcher tables in Eq.~\eqref{eq:butcherARS111}, the IMEX-RK scheme in Eqs.~\eqref{eq:imexImplicit}-\eqref{eq:imexAssembly} becomes
\begin{align}
  (f_{h}^{\star},\varphi_{h})_{ij}
  &= (f_{h}^{k},\varphi_{h})_{ij} - \dt B_{h}^{\vp}(f_{h}^{k},E_{h}^{k},\varphi_{h})_{ij}, \label{eq:directExplicitArs111} \\
  (f_{h}^{k+1},\varphi_{h})_{ij}
  &=(f_{h}^{\star},\varphi_{h})_{ij} - \dt\nu B_{h}^{\lb}(f_{h}^{k+1},\bsrho_{f_{h}}^{k+1},\varphi_{h})_{ij}. \label{eq:directImplicitArs111}
\end{align}
This IMEX-RK scheme is GSA, so the assembly step in Eq.~\eqref{eq:imexAssembly} is not needed.  
Conservation properties of the system in Eqs.~\eqref{eq:directExplicitArs111}-\eqref{eq:directImplicitArs111} was considered in \cite{hakim_etal_2020} for the case with explicit time-stepping (see also \cite{hakim_etal_2019} for the collisionless case).  
Here, we consider the case where the LB operator is integrated with implicit time-stepping.  
Define the local and global moments, given by
\begin{equation}
(\bsrho_{f_{h}})_{i} = \sum_{j=1}^{N^{v}}(f_{h},\be)_{ij} = \int_{I_{i}^{x}}\vint{f_{h}\be}_{D^{v}}\,dx
\label{eq:conservedMoments_f}
\quand
\bM_{f_{h}} = \sum_{i=1}^{N^{x}}(\bsrho_{f_{h}})_{i},
\end{equation}
respectively.

We first consider the explicit step in Eq.~\eqref{eq:directExplicitArs111}:
\begin{prop}
  Suppose the DG approximation space used in the direct method consists of piecewise polynomials of degree at least two ($p\ge2$).  
  Then the explicit update given by Eq.~\eqref{eq:directExplicitArs111} conserves particle number, momentum, and energy for the case with vanishing electric field ($E_{h}^{k}=0$).  
  For the case with nonvanishing electric field, the explicit update conserves particle number.  
  \label{prop:explicitConservation_Direct}
\end{prop}

\begin{proof}
  Setting $\varphi_{h}=\be$ into Eq.~\eqref{eq:directExplicitArs111}, using Eqs.~\eqref{eq:conservedMoments_f} and \eqref{eq:vpBLF}, and imposing the boundary conditions in Eq.~\eqref{eq:vpBoundaryCondition} gives
  \begin{equation}
    (\bsrho_{f_{h}}^{\star})_{i}
    =(\bsrho_{f_{h}}^{k})_{i} - \dt\big[\,\vint{\be\,\widehat{vf_{h}^{k}}(x_{i+\f{1}{2}},v)}_{D^{v}} - \vint{\be\,\widehat{vf_{h}^{k}}(x_{i-\f{1}{2}},v)}_{D^{v}}\,\big]
    +\dt\,(\,E_{h}^{k}T\bsrho_{f_{h}}^{k}\,)_{i}.  
    \label{eq:localConservationDirectExplicit}
  \end{equation}
%
  Then, using Eq.~\eqref{eq:localConservationDirectExplicit} with $E_{h}^{k}=0$, we obtain the globally integrated moments
  \begin{equation}
    \bM_{f_{h}}^{\star} = \bM_{f_{h}}^{k} - \dt\,\big[\,\vint{\be\,\widehat{vf_{h}^{k}}(x_{\max},v)}_{D^{v}} - \vint{\be\,\widehat{vf_{h}^{k}}(x_{\min},v)}_{D^{v}}\,\big],
    \label{eq:globalConservationDirectExplicit}
  \end{equation}
  which implies conservation of all the components of $\bM_{f_{h}}$ in the sense that the change is only due to flow through the domain boundaries.  
  When $E_{h}^{k}\ne0$, Eq.~\eqref{eq:globalConservationDirectExplicit} holds for the first component (particle conservation), since the first component of $T\bsrho_{f_{h}}$ is zero.  
\end{proof}

\begin{rem}
  Momentum and energy conservation was considered in the context of DG methods for the Vlasov--Poisson system in \cite{hakim_etal_2019}, where the electrostatic potential was obtained using a continuous finite element method.  
  (See also related work in \cite{ayuso_etal_2011,cheng_etal_2013,cheng_etal_2014}.)  
  While momentum and energy were not preserved exactly in \cite{hakim_etal_2019}, it was demonstrated that momentum conservation violations were small, and improved with increasing spatial resolution, commensurate with the polynomial degree of the approximation space, and independent of the velocity resolution.  
  Moreover, exact total (particle plus field) energy conservation in the semi-discrete limit ($\dt\to0$) was proved, provided the approximate Hamiltonian, $H_{h}=\f{1}{2}v^{2}+\Phi_{h}$, is in the continuous subset of the DG approximation space (see Proposition~3.2 in \cite{hakim_etal_2019}).  
  For the fully discrete scheme, the total energy conservation property is independent of the phase-space discretization, and depends solely on the time-stepping method.  
  Since we also obtain $\Phi_{h}$ with a continuous finite element method we expect similar results when using $p\ge2$ in the DG approximation space, and this is demonstrated numerically in Section~\ref{sec:numerical}.  
  \label{rem:momentumEnergyConservation}
\end{rem}

We next focus on the implicit step in Eq.~\eqref{eq:directImplicitArs111}.  
In our nodal DG method, where the degrees of freedom are defined at the set of GL quadrature points $S_{ij}\subset I_{ij}$, the distribution function on $I_{ij}$ is approximated by the representation
\begin{equation}
  f_{h}(x,v)|_{I_{ij}} = \sum_{r=1}^{p+1}\sum_{s=1}^{p+1}f_{rs}^{ij}\,\ell_{r}^{i}(x)\,\ell_{s}^{j}(v) = \sum_{r=1}^{p+1}f_{r}^{ij}(v)\,\ell_{r}^{i}(x),
  \label{eq:fNodalXV}
\end{equation}
where $f_{rs}^{ij}=f_{h}(x_{r},v_{s})|_{I_{ij}}$ for all $x_{r}\in S_{i}^{x}$ and $v_{s}\in S_{j}^{v}$, and we have defined
\begin{equation}
  f_{r}^{ij}(v) = \sum_{s=1}^{p+1}f_{rs}^{ij}\,\ell_{s}^{j}(v).  
  \label{eq:fNodalV}
\end{equation}
(Recall that $\{\ell_{r}^{i}\}_{q=1}^{p+1}$ and $\{\ell_{s}^{j}\}_{r=1}^{p+1}$ are Lagrange polynomials of degree $p$, constructed from the points $S_{i}^{x}$ and $S_{j}^{v}$, respectively.)  
We aim to show that the implicit solve preserves the moments of $f_{h}$ in a pointwise fashion within each spatial element.  
To this end, we first consider specific velocity integrals of the terms $(f_{h},\varphi_{h})_{ij}$ appearing in Eq.~\eqref{eq:directImplicitArs111}.  
\begin{lemma}
  Let $\psi_{h}\in\bbV_{h}^{x,p}$.  
  Then
  \begin{equation}
    \sum_{j=1}^{N^{v}}(f_{h},\psi_{h}\be)_{ij} = \sum_{r=1}^{p+1}(\psi_{h},\ell_{r}^{i})_{i}(\bsrho_{f_{h}})_{r}^{i}, 
    \label{eq:fVelocityMoments}
  \end{equation}
  where $(\bsrho_{f_{h}})_{r}^{i}=\vint{\be f_{r}^{i}}_{D^{v}}\equiv\sum_{j=1}^{N^{v}}\int_{I_{j}^{v}}\be f_{r}^{ij}\,dv$.  
  \label{lem:fVelocityMoments}
\end{lemma}
\begin{proof}
  The result follows from inserting Eq.~\eqref{eq:fNodalXV} into the left-hand side of Eq.~\eqref{eq:fVelocityMoments}.  
\end{proof}

We let $\{W_{q}\}_{q=1}^{p+1}$ denote the GL quadrature weights associated with the points $S_{i}^{x}$.  
The $(p+1)$-point GL quadrature integrates polynomials of degree $\le 2p+1$ exactly.  
In particular, for $\psi_{h}\in\bbV_{h}^{x,p}$ and $x_{r}\in S_{i}^{x}$, we have $(\psi_{h},\ell_{r}^{i})_{i}=\dx_{i}W_{r}\psi_{h}(x_{r})$.  
Then, letting $\psi_{h}=\ell_{q}^{i}$ in Eq.~\eqref{eq:fVelocityMoments}, since $\ell_{q}^{i}(x_{r})=\delta_{qr}$, the right-hand side equals $\dx_{i}W_{q}(\bsrho_{f_{h}})_{q}^{i}$.  

Next, we consider the velocity integrated LB collision operator.  
\begin{lemma}
  Let $\psi_{h}\in\bbV_{h}^{x,p}$.  
  Then
  \begin{equation}
    \sum_{j=1}^{N^{v}}B_{h}^{\lb}(f_{h},\bsrho_{f_{h}},\psi_{h}\be)_{ij}
    = - (\,T\vint{\be w_{f_{h}} f_{h}}_{D^{v}},\psi_{h}\,)_{i} - (\,\theta_{f_{h}}TT\vint{\be f_{h}}_{D^{v}},\psi_{h}\,)_{i}.
    \label{eq:lbBLF_Integrated}
  \end{equation}
  \label{lem:lbBLF_Integrated}
\end{lemma}
\begin{proof}
  The result is obtained straightforwardly with $\varphi_{h}:=\psi_{h}\be\in\bbV_{h}^{p}$ in Eq.~\eqref{eq:lbBLF}, summing over velocity elements, imposing the boundary conditions in Eqs.~\eqref{eq:lbBoundaryCondition_1} and \eqref{eq:lbBoundaryCondition_2} so that the surface terms vanish, and noting that $\p_{v}\be=T\be$ and $\p_{vv}\be=TT\be$.   
\end{proof}

\begin{rem}
We note that the conditions in Eqs.~\eqref{eq:lbBoundaryCondition_1} and \eqref{eq:lbBoundaryCondition_2}, used to arrive at Eq.~\eqref{eq:lbBLF_Integrated}, differ from the conditions used in \cite{hakim_etal_2020}.  
In \cite{hakim_etal_2020}, in the context of explicit time-stepping, only the conditions corresponding to Eq.~\eqref{eq:lbBoundaryCondition_1} were used, while $\mathfrak{f}_{h}|_{v=v_{\min}} := f_{h}|_{v=v_{\min}^{+}}$ and $\mathfrak{f}_{h}^{k+1}|_{v=v_{\max}} := f_{h}|_{v=v_{\max}^{-}}$.  
Then, corrections to the moments used to evaluate the LB operator, $u_{f_{h}}$ and $\theta_{f_{h}}$, were introduced to recover momentum and energy conservation.  
We have opted for the conditions in Eqs.~\eqref{eq:lbBoundaryCondition_1} and \eqref{eq:lbBoundaryCondition_2}, because the conditions used in \cite{hakim_etal_2020} are not straightforwardly compatible with the IMEX time-stepping scheme used here.  
\end{rem}

\begin{prop}
  Suppose the DG approximation space used in the direct method consists of polynomials of degree at least two ($p\ge2$), and assume the spatial integrals in Eq.~\eqref{eq:lbBLF} are evaluated with an $N$-point quadrature with points $\tilde{S}_{i}^{x}=\{\tilde{x}_{q}\}_{q=1}^{N}\subset I_{i}^{x}$ and weights $\{\tilde{W}_{q}\}_{q=1}^{N}$.  
  Then, the implicit update given by Eq.~\eqref{eq:directImplicitArs111} conserves particle number, momentum, and energy; i.e., $(\bsrho_{f_{h}}^{k+1})_{q}^{i}=(\bsrho_{f_{h}}^{\star})_{q}^{i}$ for all $x_{q}\in S_{i}^{x}$.
  \label{prop:implicitConservation_Direct}
\end{prop}
\begin{proof}
  Since $p\ge2$, letting $\varphi_{h}=\ell_{q}^{i}\be$ in Eq.~\eqref{eq:directImplicitArs111}, summing over velocity elements, using Eq.~\eqref{eq:fVelocityMoments} from Lemma~\ref{lem:fVelocityMoments} and Eq.~\eqref{eq:lbBLF_Integrated} from Lemma~\ref{lem:lbBLF_Integrated} with $\psi_{h}=\ell_{q}^{i}$, and evaluating the spatial integrals on the right-hand side of Eq.~\eqref{eq:lbBLF_Integrated} with an $N$-point quadrature gives
  \begin{equation}
    (\bsrho_{f_{h}}^{k+1})_{q}^{i} = (\bsrho_{f_{h}}^{\star})_{q}^{i}
    +\f{\dt\nu}{W_{q}}\sum_{m=1}^{N}\tilde{W}_{m}\,\ell_{q}^{i}(\tilde{x}_{m})
    \Big[\,
      T\vint{\be w_{m}^{k+1} f_{m}^{k+1}}_{D^{v}}
      +\theta_{m}^{k+1}TT\vint{\be f_{m}^{k+1}}_{D^{v}}
    \,\Big],
    \label{eq:localConservationConstraintDirectImplicit}
  \end{equation}
  where $f_{m}(v)=f_{h}(\tilde{x}_{m},v)$, $w_{m}=u_{m}-v$, and $u_{m}=u_{f_{h}}(\tilde{x}_{m})$ and $\theta_{m}=\theta_{f_{h}}(\tilde{x}_{m})$.  
  
  It remains to show that the expression inside the square brackets on the right-hand side of Eq.~\eqref{eq:localConservationConstraintDirectImplicit} vanishes for all quadrature points $\tilde{x}_{m}$.  
  To this end, recall that $T\be=(0,1,v)^{\rm{T}}$ and $TT\be=(0,0,1)^{\rm{T}}$.  
  Then, direct evaluation of each term inside the square brackets gives
  \begin{equation}
    T\vint{\be w_{m} f_{m}}_{D^{v}}
    =-n_{m}\theta_{m}\left(\begin{array}{c}0\\0\\1\end{array}\right)
    \quad\text{and}\quad
    \theta_{m}TT\vint{\be f_{m}}_{D^{v}}
    =n_{m}\theta_{m}\left(\begin{array}{c}0\\0\\1\end{array}\right),
  \end{equation}
  where $n_{m}=n_{f_{h}}(\tilde{x}_{m})$.  
  This completes the proof.  
\end{proof}

\begin{rem}
  In Eq.~\eqref{eq:localConservationConstraintDirectImplicit} in Proposition~\ref{prop:implicitConservation_Direct}, we use a general $N$-point quadrature to evaluate the spatial integrals originating from the LB form in Eq.~\eqref{eq:lbBLF}.  
  In the numerical examples presented in Section~\ref{sec:numerical}, we use $(p+1)$-point GL quadratures to evaluate these spatial integrals.  
  That is, the quadrature points coincide with the interpolation points in our nodal DG scheme ($\tilde{S}_{i}^{x}=S_{i}^{x}$), similar to the spectral-type nodal collocation DG approximation in \cite{bassi_etal_2013}.  
  Since the spatial integrals involve polynomials of degree $3p$, this approximation results in under-integration.  
  The velocity integrals involve polynomials of degree at most $p+2$, which are evaluated exactly with the $(p+1)$-point GL quadrature.  
\end{rem}

\subsection{Conservation Properties of the mM Method for the VPLB System}
\label{sec:conservationMM}

With the Butcher tables in Eq.~\eqref{eq:butcherARS111}, Eqs.~\eqref{eq:imexMacroExplicit}-\eqref{eq:imexMicroExplicit} can be expressed as
\begin{align}
  (\bsrho_{f,h}^{k+1},\psi_{h})_{i}
  &= (\bsrho_{f,h}^{k},\psi_{h})_{i} - \dt B_{h}^{\macro}(\bsrho_{f,h}^{k},g_{h}^{k},E_{h}^{k},\psi_{h})_{i} \label{eq:macroExplicitARS111} \\
  (g_{h}^{\star},\varphi_{h})_{ij}
  &=(g_{h}^{k},\varphi_{h})_{ij} + \big[\,(\cE[\bsrho_{f,h}^{k}],\varphi_{h})_{ij}-(\cE[\bsrho_{f,h}^{k+1}],\varphi_{h})_{ij}\,\big] \label{eq:microExplicitARS111} \nonumber \\
  &\hspace{12pt}
  -\dt \big[\, B_{h}^{\vp}(g_{h}^{k},E_{h}^{k},\varphi_{h})_{ij} + B_{h}^{\micro}(\bsrho_{f,h}^{k},E_{h}^{k},\varphi_{h})_{ij}\,\big] \\
  (g_{h}^{k+1},\varphi_{h})_{ij}
  &=(g_{h}^{\star},\varphi_{h})_{ij} - \dt \nu B_{h}^{\lb}(g_{h}^{k+1},\bsrho_{f,h}^{k+1},\varphi_{h})_{ij}, \label{eq:microImplicitARS111}
\end{align}
where in Eq.~\eqref{eq:microImplicitARS111} we have deliberately left out the cleaning limiter mentioned in Section~\ref{imexrkMM} (cf. Eq.~\eqref{eq:imexMicroImplicit}).  

We start with the explicit step in Eq.~\eqref{eq:macroExplicitARS111}.  In analogy with Proposition~\ref{prop:explicitConservation_Direct} for the direct method, we have the following:
\begin{prop}
  The explicit update given by Eq.~\eqref{eq:macroExplicitARS111} conserves particle number, momentum, and energy for the case with vanishing electric field ($E_{h}^{k}=0$).  
  For the case with nonvanishing electric field, the explicit update conserves particle number.  
  \label{eq:explicitConservation_mM}
\end{prop}

\begin{proof}
  Setting $\psi_{h}=1$ in Eq.~\eqref{eq:macroExplicitARS111} and using Eq.~\eqref{eq:macroBLF} gives
  \begin{align}
    (\bsrho_{f,h}^{k+1})_{i}
    &= (\bsrho_{f,h}^{k})_{i} - \dt 
    \Big\{\,
      \big[\,
        \widehat{\bF(\bsrho_{f,h}^{k})}(x_{i+\f{1}{2}}) + \widehat{\bff(g_{h}^{k})}(x_{i+\f{1}{2}})
      \,\big] \nonumber \\
      &\hspace{72pt}
      -
      \big[\,
        \widehat{\bF(\bsrho_{f,h}^{k})}(x_{i-\f{1}{2}})+\widehat{\bff(g_{h}^{k})}(x_{i-\f{1}{2}})
      \,\big]
    \,\Big\}
    +\dt\,(E_{h}^{k}T\,\bsrho_{f,h}^{k})_{i}.
    \label{eq:localConservationMacroExplicit}
  \end{align}
  Eq.~\eqref{eq:localConservationMacroExplicit} corresponds to Eq.~\eqref{eq:localConservationDirectExplicit} for the direct method.  The  remainder of the proof then follows the logic of Proposition~\ref{prop:explicitConservation_Direct} and is omitted.
\end{proof}

For consistency it is important to maintain the constraints in Eq.~\eqref{eq:gConstraints} at the numerical level, in the sense that
\begin{equation}
  (\,\vint{\be\,g_{h}}_{D^{v}},\psi_{h}\,)_{i} = 0 \quad\forall\psi_{h}\in\bbV_{h}^{x,p}.
  \label{eq:gConstraintsWeak}
\end{equation}
We show that, under relatively mild conditions, these constraints are maintained pointwise, $(\bsrho_{g_{h}})_{q}^{i}=\vint{\be g_{q}^{i}}_{D^{v}}=0$, for all $x_{q}\in S_{i}^{x}$.   
Here, $(\bsrho_{g_{h}})_{q}^{i}$ is defined as in Eq.~\eqref{eq:fVelocityMoments}, with $f_{h}$ replaced by $g_{h}$.  
Then, using Eq.~\eqref{eq:microExplicitARS111} with $\varphi_{h}:=\psi_{h}\be$, where $\psi_{h}\in\bbV_{h}^{x,p}$, and Lemma~\ref{lem:fVelocityMoments}, the moments of the micro distribution after the explicit step can then be expressed as
\begin{align}
  \sum_{r=1}^{p+1}(\psi_{h},\ell_{r}^{i})_{i}
  (\bsrho_{g_{h}}^{\star})_{r}^{i}
  &=\sum_{r=1}^{p+1}(\psi_{h},\ell_{r}^{i})_{i} (\bsrho_{g_{h}}^{k})_{r}^{i} 
  + \big[\,(\vint{\be\cE[\bsrho_{f,h}^{k}]}_{D^{v}},\psi_{h})_{i}-(\vint{\be\cE[\bsrho_{f,h}^{k+1}]}_{D^{v}},\psi_{h})_{i}\,\big] \nonumber \\
  &\hspace{12pt}
  -\dt \sum_{j=1}^{N^{v}}\big[\, B_{h}^{\vp}(g_{h}^{k},E_{h}^{k},\psi_{h}\be)_{ij} + B_{h}^{\micro}(\bsrho_{f,h}^{k},E_{h}^{k},\psi_{h}\be)_{ij}\,\big].  
  \label{eq:localConservationMicroExplicitIntermediate}
\end{align}


Next, our goal is to insert Eq.~\eqref{eq:macroExplicitARS111} into the right-hand side of Eq.~\eqref{eq:localConservationMicroExplicitIntermediate}; specifically by replacing the velocity moments of the Maxwellian --- the second and third terms on the right-hand side of Eq.~\eqref{eq:localConservationMicroExplicitIntermediate} --- with $\bsrho_{f,h}$.  
This requires some further specification of how the Maxwellian is approximated within each spatial element to evaluate these terms.  
In the nodal DG scheme, the representation of the macro moments on element $I_{i}^{x}$ is given by the nodal expansion
\begin{equation}
  \bsrho_{f,h}(x)|_{I_{i}^{x}} = \sum_{k=1}^{p+1}(\bsrho_{f})_{k}^{i}\,\ell_{k}^{i}(x),
  \label{eq:momentsNodalExpansion}
\end{equation}
where $(\bsrho_{f})_{k}^{i}=\bsrho_{f,h}(x_{k})$, for all $x_{k}\in S_{i}^{x}$.  
The values $(\bsrho_{f})_{k}^{i}$ are then used to define the Maxwellian on $I_{i}^{x}$, i.e.,
\begin{equation}
  \cE[\bsrho_{f,h}](x,v)|_{I_{i}^{x}} := \sum_{k=1}^{p+1}\cE_{k}^{i}(v)\ell_{k}^{i}(x),~\mbox{where}~\cE_{k}^{i}(v)=\cE[(\bsrho_{f})_{k}^{i}](v),
  \label{eq:maxwellianApproximate}
\end{equation}
in order to evaluate the second and third terms on the right-hand side of Eq.~\eqref{eq:localConservationMicroExplicitIntermediate}.  
(The expansion in Eq.~\eqref{eq:maxwellianApproximate} is also used when evaluating the spatial integrals in Eq.~\eqref{eq:microBLF} with $(p+1)$-point GL quadratures.)  

\begin{lemma}
  Let the Maxwellian on $I_{i}^{x}$ be approximated by the expansion in Eq.~\eqref{eq:maxwellianApproximate}.  
  Then
  \begin{equation}
    (\vint{\be\cE[\bsrho_{f,h}]}_{\bbR},\psi_{h})_{i} = (\bsrho_{f,h},\psi_{h})_{i}\quad\forall\psi_{h}\in\bbV_{h}^{x,p}.
    \label{eq:momentsOfApproximateMaxwellian}
  \end{equation}
  \label{lem:momentsOfApproximateMaxwellian}
\end{lemma}
\begin{proof}
  The result follows by inserting Eq.~\eqref{eq:maxwellianApproximate} into the left-hand side of Eq.~\eqref{eq:momentsOfApproximateMaxwellian}. 
  Using the fact that $\vint{\be\cE_{k}^{i}}_{\bbR}=(\bsrho_{f})_{k}^{i}$ and the expansion in Eq.~\eqref{eq:momentsNodalExpansion} gives
  \begin{equation}
    (\vint{\be\cE[\bsrho_{f,h}]}_{\bbR},\psi_{h})_{i}
    =\sum_{k=1}^{p+1}\vint{\be\cE_{k}^{i}}_{\bbR}(\ell_{k}^{i},\psi_{h})_{i}
    =\sum_{k=1}^{p+1}(\bsrho_{f})_{k}^{i}\,(\ell_{k}^{i},\psi_{h})_{i}
    =(\bsrho_{f,h},\psi_{h})_{i}.
  \end{equation}
\end{proof}

\begin{lemma}
  Let $\psi_{h}\in\bbV_{h}^{x,p}$.  
  Then
  \begin{align}
    \sum_{j=1}^{N^{v}}B_{h}^{\vp}(g_{h},E_{h},\psi_{h}\be)_{ij}
    &=
    \big[\,
      \vint{\be\,\widehat{vg_{h}}(x_{i+\f{1}{2}},v)}_{D^{v}}\,\psi_{h}(x_{i+\f{1}{2}}^{-}) 
      - \vint{\be\,\widehat{vg_{h}}(x_{i-\f{1}{2}},v)}_{D^{v}}\,\psi_{h}(x_{i-\f{1}{2}}^{+})
    \,\big] \nonumber \\
    &\hspace{12pt}
    - (\,\vint{\be vg_{h}}_{D^{v}},\p_{x}\psi_{h}\,)_{i}
    - (\,E_{h}T\vint{\be g_{h}}_{D^{v}},\psi_{h}\,)_{i}.
    \label{eq:vpBLF_Integrated}
  \end{align}
  \label{lem:vpBLF_Integrated}
\end{lemma}
\begin{proof}
  The result is obtained by setting $\varphi_{h}:=\psi_{h}\be$ in Eq.~\eqref{eq:vpBLF}, summing over velocity elements, and applying Eq.~\eqref{eq:vpBoundaryCondition} (with $f_{h}$ replaced by $g_{h}$).  
\end{proof}

\begin{lemma}
  Let $\psi_{h}\in\bbV_{h}^{x,p}$.  
  Then
  \begin{align}
    \sum_{j=1}^{N^{v}}B_{h}^{\micro}(\bsrho_{f,h},E_{h},\psi_{h}\be)_{ij}
    &=
    \big[\,
      \vint{\be\,\widehat{v\cE}[\bsrho_{f,h}](x_{i+\f{1}{2}},v)}_{D^{v}}\,\psi_{h}(x_{i+\f{1}{2}}^{-})
      -\vint{\be\,\widehat{v\cE}[\bsrho_{f,h}](x_{i-\f{1}{2}},v)}_{D^{v}}\,\psi_{h}(x_{i-\f{1}{2}}^{+})
    \,\big] \nonumber \\
    &\hspace{12pt}
    - (\,\vint{\be v\cE[\bsrho_{f,h}^{k}]}_{D^{v}},\p_{x}\psi_{h}\,)_{i}
    - (\,E_{h}T\vint{\be\cE[\bsrho_{f,h}]}_{D^{v}},\psi_{h}\,)_{i}.  
    \label{eq:microBLF_Integrated}
  \end{align}
  \label{lem:microBLF_Integrated}
\end{lemma}
\begin{proof}
  The result is obtained by setting $\varphi_{h}:=\psi_{h}\be$ in Eq.~\eqref{eq:microBLF}, summing over velocity elements, and applying Eq.~\eqref{eq:velocityBoundaryConditionMaxwellian}.  
\end{proof}

\begin{prop}
  Suppose the DG approximation space used in the mM method consists of piecewise polynomials of degree at least two ($p\ge2$).  
  Assume that $D^{v}=\bbR$, and that $(\bsrho_{g_{h}}^{k})_{q}^{i}=0$ for all $x_{q}\in S_{x}^{i}$.  
  Then $g_{h}^{\star}$, obtained from explicit update given by Eq.~\eqref{eq:microExplicitARS111}, satisfies $(\bsrho_{g_{h}}^{\star})_{q}^{i}=0$ for all $x_{q}\in S_{i}^{x}$.
  \label{prop:zeroMicroMomentsExplicit}
\end{prop}

\begin{proof}
  Setting $\varphi_{h}:=\psi_{h}\be$ in Eq.~\eqref{eq:microExplicitARS111}, where $\psi_{h}\in\bbV_{h}^{x,p}$, results in Eq.~\eqref{eq:localConservationMicroExplicitIntermediate}.  
  Set $\psi_{h}:=\ell_{q}^{i}$ in Eq.~\eqref{eq:localConservationMicroExplicitIntermediate} to obtain the equation for the pointwise moments of $g_{h}^{\star}$
  \begin{align}
    \dx_{i}W_{q}(\bsrho_{g_{h}}^{\star})_{q}^{i}
    &=\dx_{i}W_{q}(\bsrho_{g_{h}}^{k})_{q}^{i} + \big[\,(\vint{\be\cE[\bsrho_{f,h}^{k}]}_{D^{v}},\ell_{q}^{i})_{i}-(\vint{\be\cE[\bsrho_{f,h}^{k+1}]}_{D^{v}},\ell_{q}^{i})_{i}\,\big] \nonumber \\
    &\hspace{12pt}
    -\dt \sum_{j=1}^{N^{v}}\big[\, B_{h}^{\vp}(g_{h}^{k},E_{h}^{k},\ell_{q}^{i}\be)_{ij} + B_{h}^{\micro}(\bsrho_{f,h}^{k},E_{h}^{k},\ell_{q}^{i}\be)_{ij}\,\big].  
    \label{eq:localConservationMicroExplicitPenultimate}
  \end{align}
  Set  $\psi_{h}:=\ell_{q}^{i}$ in Eq.~\eqref{eq:macroExplicitARS111} and use Lemma~\ref{lem:momentsOfApproximateMaxwellian} to replace the left-hand side.  
  The result is
  \begin{equation}
    \big[\,(\vint{\be\cE[\bsrho_{f,h}^{k}]}_{D^{v}},\ell_{q}^{i})_{i}-(\vint{\be\cE[\bsrho_{f,h}^{k+1}]}_{D^{v}},\ell_{q}^{i})_{i}\,\big]
    = \dt B_{h}^{\macro}(\bsrho_{f,h}^{k},g_{h}^{k},E_{h}^{k},\ell_{q}^{i})_{i}.  
    \label{eq:macroExplicitARS111Intermediate}
  \end{equation}
  We then write Eq.~\eqref{eq:localConservationMicroExplicitPenultimate}, after inserting Eq.~\eqref{eq:macroExplicitARS111Intermediate}, as
  \begin{equation}
    \dx_{i}W_{q}[(\bsrho_{g_{h}}^{\star})_{q}^{i}-(\bsrho_{g_{h}}^{k})_{q}^{i}]/\dt
    =\Gamma(g_{h}^{k},\bsrho_{f,h}^{k},E_{h}^{k},\ell_{i}^{q})_{i},
    \label{eq:nodalConservationMicroExplicit}
  \end{equation}
  where, for $\psi_{h}\in\bbV_{h}^{x,p}$, we have defined
  \begin{align}
    &\Gamma(g_{h},\bsrho_{f,h},E_{h},\psi_{h})_{i} \nonumber \\
    &=B_{h}^{\macro}(\bsrho_{f,h},g_{h},E_{h},\psi_{h})_{i}
    -\sum_{j=1}^{N^{v}}\big[\, B_{h}^{\vp}(g_{h},E_{h},\psi_{h}\be)_{ij} + B_{h}^{\micro}(\bsrho_{f,h},E_{h},\psi_{h}\be)_{ij}\,\big].  
    \label{eq:GammaDefinition}
  \end{align}
  It remains to show that $\Gamma(g_{h},\bsrho_{f,h},E_{h},\psi_{h})_{i}=0$.  
  Using Eq.~\eqref{eq:macroBLF}, and the results of Lemmas~\ref{lem:vpBLF_Integrated} and \ref{lem:microBLF_Integrated}, we can write
  \begin{align}
    &\Gamma(g_{h},\bsrho_{f,h},E_{h},\psi_{h})_{i} \nonumber \\
    &=
    \big[
      \underbrace{
      \widehat{\bF(\bsrho_{f,h})}(x_{i+\f{1}{2}})
      -\vint{\be\,\widehat{v\cE}[\bsrho_{f,h}](x_{i+\f{1}{2}},v)}_{D^{v}}}_{\text{I}}
    \big]\,\psi_{h}(x_{i+\f{1}{2}}^{-}) \nonumber \\
    &\hspace{12pt}
    -\big[
      \underbrace{
      \widehat{\bF(\bsrho_{f,h})}(x_{i-\f{1}{2}})
      -\vint{\be\,\widehat{v\cE}[\bsrho_{f,h}](x_{i-\f{1}{2}},v)}_{D^{v}}}_{\text{II}}
    \big]\,\psi_{h}(x_{i-\f{1}{2}}^{+}) 
    - (\,\underbrace{\bF(\bsrho_{f,h})-\vint{\be v\cE[\bsrho_{f,h}]}_{D^{v}}}_{\text{III}},\,\p_{x}\psi_{h}\,)_{i}
    \nonumber \\
    &\hspace{12pt}
    +\big[
      \underbrace{
      \widehat{\bff(g_{h})}(x_{i+\f{1}{2}})
      -\vint{\be\,\widehat{vg_{h}}(x_{i+\f{1}{2}},v)}_{D^{v}}}_{\text{IV}}
    \big]\,\psi_{h}(x_{i+\f{1}{2}}^{-}) \nonumber \\
    &\hspace{12pt}
    -\big[
      \underbrace{
      \widehat{\bff(g_{h})}(x_{i-\f{1}{2}})
      - \vint{\be\,\widehat{vg_{h}}(x_{i-\f{1}{2}},v)}_{D^{v}}}_{\text{V}}
    \big]\,\psi_{h}(x_{i-\f{1}{2}}^{+}) 
    - (\,\underbrace{\bff(g_{h})-\vint{\be vg_{h}}_{D^{v}}}_{\text{VI}},\,\p_{x}\psi_{h}\,)_{i} \nonumber \\
    &\hspace{12pt}
    - (\,E_{h}T\,[\,\underbrace{\bsrho_{f,h}-\vint{\be\cE[\bsrho_{f,h}]}_{D^{v}}}_{\text{VII}}\,],\,\psi_{h}\,)_{i}
    + (\,\underbrace{E_{h}T\vint{\be g_{h}}_{D^{v}}}_{\text{VIII}},\,\psi_{h}\,)_{i} = 0.
    \label{eq:GammaElaborate}
  \end{align}
  In Eq.~\eqref{eq:GammaElaborate}, terms emanating from the discretized macro model have been paired with terms emanating from the discretized micro model (terms I-VII, where in each term the leading expression emanates from the discretized macro model), and these have been designed to cancel individually in order to prove the proposition.  
  Specifically, the definition of the numerical flux for the macro component in Eq.~\eqref{eq:numercalFluxMacroMacroA} --- together with the numerical flux for the micro component in Eq.~\eqref{eq:numericalFluxMicroMacroA} --- ensures that terms I and II vanish.  
  Exact evaluation of the velocity integrals $\vint{\be v\cE[\bsrho_{f,h}]}_{D^{v}}$, which emanate from the term containing $\p_{x}\varphi_{h}$ in Eq.~\eqref{eq:microBLF}, ensures that term III vanishes.  
  The definition of the numerical flux in Eq.~\eqref{eq:numercalFluxMacroMicro} ensures that terms IV and V vanish, while term~VI is zero because $\bff(g_{h}):=\vint{\be vg_{h}}_{D^{v}}$ in Eq.~\eqref{eq:macroBLF}.  Eq.~\eqref{eq:momentsOfApproximateMaxwellian} in Lemma \ref{lem:momentsOfApproximateMaxwellian} ensures that term VII vanishes.  
 By the assumption $(\bsrho_{g_{h}}^{k})_{q}^{i}=0$, term VIII vanishes when $g_h = g_h^k$.  
 Since all the terms vanish, it follows that $(\bsrho_{g_{h}}^{\star})_{q}^{i}=(\bsrho_{g_{h}}^{k})_{q}^{i}=0$ for all $x_{q}\in S_{i}^{x}$.  
\end{proof}

\begin{rem}
  Eq.~\eqref{eq:nodalConservationMicroExplicit} is a discrete analogue of Eq.~\eqref{eq:gMomentsChange}.  
  The proof of Proposition~\ref{prop:zeroMicroMomentsExplicit} illustrates that the discretization of the macro and micro components must be designed in a cohesive fashion (i.e., the terms in the discretization of the two components are consistent), so that the right-hand side of Eq.~\eqref{eq:nodalConservationMicroExplicit} vanishes and the moments of $g_{h}$ are preserved.  
  \label{rem:consistentDiscretization}
\end{rem}

\begin{rem}
  Proposition~\ref{prop:zeroMicroMomentsExplicit} assumes that $D^{v}=\bbR$.  
  However, in practical applications $D^{v}$ is a bounded domain.  
  Thus in practice many of the terms in Eq.~\eqref{eq:GammaElaborate} do not exactly cancel; specifically terms {\rm I}, {\rm II}, {\rm III}, and {\rm VII}.  
  The residuals can be made arbitrarily small by extending the velocity domain $D^{v}$, but such an approach may not be practical because the time step restriction for explicit integration scales as $1/\max(|v_{\min}|,|v_{\max}|)$ (see Eq.~\eqref{eq:cflTimeStep}).  
  The truncated velocity domain gives rise to additional conservation errors for the moments of $g_{h}$.  
  The main reason for introducing the cleaning limiter in Section~\ref{sec:cleaningLimiter} is to rectify these small non-zero contributions to $\vint{\be g_{h}}_{D^{v}}$ that emanate from the explicit step.  
  Alternatively, since the velocity integrals in terms {\rm I}, {\rm II}, {\rm III}, and {\rm VII} involve the Maxwellian, which we evaluate analytically, we have found that these terms can be made to vanish exactly by letting the first and last velocity element be artificially extended for terms involving the Maxwellian in Eq.~\eqref{eq:microSemiDiscrete}.  
  Specifically, by replacing $I_{j}^{v}$ in the last two terms on the left-hand side of Eq.~\eqref{eq:microSemiDiscrete} with
  \begin{equation}
    \tilde{I}_{j}^{v} = 
    \left\{\begin{array}{ll}
      \left(-\infty,\,v_{\min}+\dv_{j}\right] & \text{if } j = 1, \\
      \left[v_{\max}-\dv_{j},\,+\infty\right) & \text{if } j = N^{v},\\
      I_{j}^{v} & \text{otherwise.} 
    \end{array} \right.
    \label{eq:infiniteVelocityDomain}
  \end{equation}
  \label{rem:microMomentsConservationErrors}
\end{rem}

Next, we consider the implicit update in the mM method given by Eq.~\eqref{eq:microImplicitARS111}.  

\begin{prop}
  Suppose the DG approximation space used in the mM method consists of polynomials of degree at least two ($p\ge2$).  
  Let the spatial integrals in the LB form in Eq.~\eqref{eq:lbBLF} be evaluated using $(p+1)$-point GL quadrature, and assume that $g_{h}^{\star}$ satisfies $(\bsrho_{g_{h}}^{\star})_{q}^{i}=0$ for all $x_{q}\in S_{i}^{x}$, so that Eq.~\eqref{eq:gConstraintsWeak} holds.  
  Then $(\bsrho_{g_{h}}^{k+1})_{q}^{i} = 0$.
\end{prop}
\begin{proof}
  Letting $\varphi_{h}=\ell_{q}^{i}\be$ in Eq.~\eqref{eq:microImplicitARS111} and invoking Lemma~\ref{lem:lbBLF_Integrated}, with spatial integrals in Eq.~\eqref{eq:lbBLF_Integrated} evaluated using $(p+1)$-point GL quadrature, gives (cf. Eq.~\eqref{eq:localConservationConstraintDirectImplicit})
  \begin{align}
    (\bsrho_{g_{h}}^{k+1})_{q}^{i}
    &=(\bsrho_{g_{h}}^{\star})_{q}^{i} 
    	+ \dt\nu\,\big[\,T\vint{\be w_{q}^{k+1}g_{q}^{k+1}}_{D^{v}} 
    	+ \theta_{q}^{k+1}TT\vint{\be g_{q}^{k+1}}_{D^{v}}\,\big],
    \label{eq:localConservationMicroImplicit}
  \end{align}
  where $g_{q}(v)=g_{h}(v,x_{q})$, $w_{q}=u_{q}-v$, and $u_{q}=u_{f,h}(x_{q})$ and $\theta_{q}=\theta_{f,h}(x_{q})$.  
  
  A direct calculation gives
  \begin{equation}
    T\vint{\be w_{q}^{k+1} g_{q}^{k+1}}_{D^{v}}  
    = \begin{pmatrix}
         0 & 0 & 0 \\
  	 u_{q}^{k+1} & - 1 & 0  \\
  	 0 & u_{q}^{k+1} & - 2
    \end{pmatrix}(\bsrho_{g_{h}}^{k+1})_{q}^{i}
  \end{equation}
  and
  \begin{equation}
    \theta_{q}^{k+1}TT\vint{\be g_{q}^{k+1}}_{D^{v}}
    = \begin{pmatrix}
         0 & 0 & 0 \\
  	 0 & 0 & 0  \\
  	 \theta_{q}^{k+1} & 0 & 0
    \end{pmatrix}(\bsrho_{g_{h}}^{k+1})_{q}^{i}.  
  \end{equation}
  Thus \eqref{eq:localConservationMicroImplicit} can be written as 
  \begin{equation}
    (\bsrho_{g_{h}}^{k+1})_{q}^{i} = \big(\,I - \dt \nu M\,\big)^{-1} (\bsrho_{g_{h}}^{\star})_{q}^{i},
  \end{equation}
  where 
  \begin{equation}
    M = \begin{pmatrix}
		0 & 0 & 0 \\
		u_{q}^{k+1} & - 1 & 0 \\
		\theta_{q}^{k+1} & u_{q}^{k+1} & - 2
	\end{pmatrix},
  \end{equation}
  and $\big(\,I - \dt \nu M\,\big)$ is invertible for $\dt \nu \geq0$.  
  Therefore, $(\bsrho_{g_{h}}^{\star})_{q}^{i}=0$ implies $(\bsrho_{g_{h}}^{k+1})_{q}^{i}=0$.  
\end{proof}

\section{Cleaning Limiter}
\label{sec:cleaningLimiter}

As noted in Remark~\ref{rem:microMomentsConservationErrors}, the explicit update of the micro distribution given by Eq.~\eqref{eq:microExplicitARS111}, and consequently the update in Eq.~\eqref{eq:imexMicroExplicit}, can, due to finite velocity domain effects, give rise to small violations of the conservation constraints in Eq.~\eqref{eq:gConstraints}.  
To eliminate these violations, we introduce a `cleaning' limiter, to be applied to $g_{h}$ before the implicit solves in the IMEX scheme in Eq.~\eqref{eq:imexMicroImplicit}.  

In the  nodal DG scheme the global representation of $g_{h}$ in velocity space can, for arbitrary $x\in D^{x}$, be written as
\begin{equation}
  g_{h}(x,v)=\sum_{j=1}^{N^{v}}\chi_{I_{j}^{v}}(v)\sum_{k=1}^{p+1}g_{k}^{j}(x)\ell_{k}^{j}(v),
  \label{eq:nodalVelocityRepresentation}
\end{equation}
where $\chi_{I_{j}^{v}}(v)$ is the indicator function on $I_{j}^{v}$, and $g_{k}^{j}$ is the solution in the GL quadrature point $v_{k}\in S_{j}^{v}\subset I_{j}^{v}$.  
The cleaning limiter is global in velocity space, but is applied independently for each spatial point in $D^{x}$; i.e., for each $x\in S_{i}^{x}$, $i=1,\ldots,N^{x}$.  
Then, the cleaned solution $\tilde{g}_{h}\in\bbV_{h}^{p}$ is obtained by solving the linearly constrained least squares problem 
\begin{align}
  \min_{\tilde{g}_{h}}\f{1}{2}\int_{D^{v}}(\tilde{g}_{h}-g_{h})^{2}\,dv
  \quad
  \mbox{subject to} \quad \int_{D^{v}}\tilde{g}_{h}\be dv=0.  
  \label{eq:constrainedLeastSquaresCleaning}
\end{align}
(We solve Eq.~\eqref{eq:constrainedLeastSquaresCleaning} using the subroutine DGGLSE in LAPACK \cite{anderson1999lapack}, which expresses the least squares solution in terms of a generalized QR decomposition \cite{anderson1992generalized}. )  

We denote the application of the cleaning limiter to obtain $\tilde{g}_{h}$ from $g_{h}$ by solving the optimization problem in Eq.~\eqref{eq:constrainedLeastSquaresCleaning} simply by
\begin{equation}
  \tilde{g}_{h} := \Lambda\big\{g_{h}\big\}.  
\end{equation}

%% file: numerical.tex
\section{Numerical Experiments}
\label{sec:numerical}


In this section we apply the mM method developed in the previous sections to a variety of standard test problems for the VPLB system.  
Our goal is to document the performance of the mM method under collisionality conditions ranging from kinetic ($\nu$ small) to fluid ($\nu$ large).  
The conservation properties of the mM method is a major focus.  
We also aim to compare the performance of the mM and direct methods.  
Specifically, we investigate the extent to which the mM method provides improved accuracy over the direct method in collision dominated regimes, in the sense that the dynamics is well captured by the macro component, and simulations can be performed with coarser resolution in velocity space.  

We use explicit, implicit, and IMEX time-stepping methods in the tests presented.  
We use simple backward Euler time-stepping for the relaxation problem in Section~\ref{sec:numerical_Relaxation}.  
For the Riemann problem in Section~\ref{sec:numerical_Riemann} and the collisional Landau damping problem in Section~\ref{sec:numerical_Landau}, we use the IMEX method from \cite{chu_etal_2019}, with Butcher tables given by
\begin{equation}
  \begin{array}{c | c c c}
  	\tilde{c}_{1} & \tilde{a}_{11} & \tilde{a}_{12} & \tilde{a}_{13} \\
  	\tilde{c}_{2} & \tilde{a}_{21} & \tilde{a}_{22} & \tilde{a}_{23} \\
	\tilde{c}_{3} & \tilde{a}_{31} & \tilde{a}_{32} & \tilde{a}_{33} \\ \hline
  	                   & \tilde{w}_{1} & \tilde{w}_{2} & \tilde{w}_{3}
  \end{array}
  =
  \begin{array}{c | c c c}
  	0 & 0 & 0 & 0 \\
  	1 & 1 & 0 & 0 \\ 
	1 & 0.5 & 0.5 & 0 \\ \hline
  	   & 0.5 & 0.5 & 0
  \end{array}
  \quad
    \begin{array}{c | c c c}
  	c_{1} & a_{11} & a_{12} & a_{13} \\
  	c_{2} & a_{21} & a_{22} & a_{23} \\
	c_{3} & a_{31} & a_{32} & a_{33} \\ \hline
  	         & w_{1} & w_{2} & w_{3}
  \end{array}
  =
  \begin{array}{c | c c c}
  	0 & 0 & 0 & 0 \\
  	1 & 0 & 1 & 0 \\
	1 & 0 & 0.5 & 0.5 \\ \hline
  	   & 0 & 0.5 & 0.5 
  \end{array},
  \label{eq:butcherPDARS}
\end{equation}
which is GSA and formally only first-order accurate, but is SSP with a time step restriction for stability determined solely by the explicit part, and, when $\nu=0$, it reduces to the optimal explicit SSP-RK2 scheme from \cite{shuOsher_1988}.  
For the two-stream instability problem in Section~\ref{sec:numerical_TwoStream}, we use the optimal explicit SSP-RK3 scheme from \cite{shuOsher_1988}, which in the Butcher table format used Sections~\ref{sec:imexrkDirect} and \ref{imexrkMM} takes the form
\begin{equation}
  \begin{array}{c | c c c}
  	\tilde{c}_{1} & \tilde{a}_{11} & \tilde{a}_{12} & \tilde{a}_{13} \\
  	\tilde{c}_{2} & \tilde{a}_{21} & \tilde{a}_{22} & \tilde{a}_{23} \\
	\tilde{c}_{3} & \tilde{a}_{31} & \tilde{a}_{32} & \tilde{a}_{33} \\ \hline
  	                   & \tilde{w}_{1} & \tilde{w}_{2} & \tilde{w}_{3}
  \end{array}
  =
  \begin{array}{c | c c c}
  	0 & 0 & 0 & 0 \\
  	1 & 1 & 0 & 0 \\ 
	1/2 & 1/4 & 1/4 & 0 \\ \hline
  	   & 1/6 & 1/6 & 2/3,
  \end{array}
\end{equation}
for the explicit coefficients, while all the implicit coefficients are set to zero.  

With the exception of the the relaxation problem, which is purely implicit, we let the time step be given by
\begin{equation}
  \dt = \f{C_{\rm{CFL}}}{(2\,p+1)}\times \min_{i\in\{1,\ldots,N^{x}\}}\f{\dx_{i}}{\max(|v_{\min}|,|v_{\max}|)},
  \label{eq:cflTimeStep}
\end{equation}
where, unless otherwise specified, we use $C_{\rm{CFL}}=0.75$.  
For all the tests, we use polynomial degree $p=2$, which, when combined with third-order accurate time-stepping (e.g., SSP-RK3) results in a third-order accurate method for sufficiently smooth problems.  

\subsection{Relaxation}
\label{sec:numerical_Relaxation}


Here we consider the space-homogeneous problem
\begin{equation}
  \p_{t}f = C_{\lb}[\bsrho_{f}](f), \label{eq:homogeneousDirect}
\end{equation}
with mM decomposition
\begin{subequations}
  \label{eq:homogeneousMM}
  \begin{align}
    \p_{t}\bsrho_{f} &= 0, \label{eq:homogeneousMacro} \\
    \p_{t}g &= C_{\lb}[\bsrho_{f}](g), \label{eq:homogeneousMicro}
  \end{align}
\end{subequations}
and investigate conservation properties of the DG discretization of the LB collision operator in the context of the direct discretization and the discretization of the mM model.  
We discretize the velocity domain $D^{v}=[v_{\min},v_{\max}]=[-12,12]$ with $N^{v}=48$ elements and let the initial condition be a double Maxwellian
\begin{equation}
  f(v,t=0)=f_{0}(v) := M[\bsrho_{1}](v) + M[\bsrho_{2}](v), 
\end{equation}
with
\begin{equation}
  M[\bsrho_{i}](v) = \f{n_{i}}{\sqrt{2\pi \theta_{i}}} \exp\Big\{\,-\f{(v-u_{i})^{2}}{2\theta_{i}}\,\Big\}, 
\end{equation}
where we set $\{n_{1},u_{1},\theta_{1}\}=\{1.0,-1.5,0.5\}$ and $\{n_{2},u_{2},\theta_{2}\}=\{1.0,2.5,0.5\}$, so that $\bsrho_{f,0}=\vint{f_{0}\be}=(2,1,4.75)^{\rm{T}}$.  
When solving the mM system in Eq.~\eqref{eq:homogeneousMM}, we set $g_{0}=f_{0}-M[\bsrho_{f,0}]$ and apply the cleaning limiter from Section~\ref{sec:cleaningLimiter} to the initial condition, so that $\vint{g_{0}\be}=0$.  
After the application to the initial condition, do not use the cleaning limiter during the integration of Eq.~\eqref{eq:homogeneousMM} to the final time.  
We set the collision frequency to $\nu=10^{3}$, $\dt=10^{-2}$, and evolve until $t=1.0$ with the backward Euler time-stepping method.  

\begin{figure}[H]
	\begin{centering}
	\captionsetup[subfigure]{justification=centering}
	\subfloat[Initial and final distributions]
	{\begin{minipage}{0.5\textwidth}
			\includegraphics[width=\linewidth]{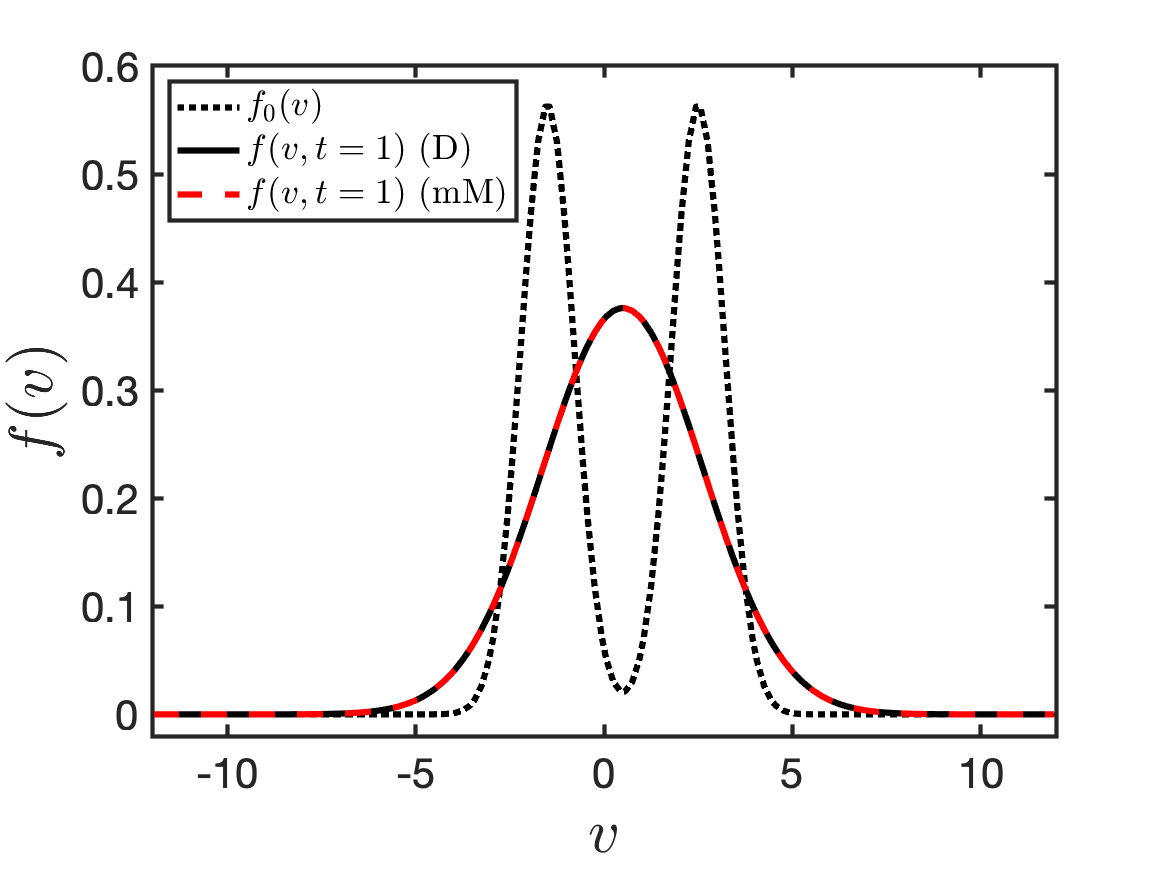}
			\label{fig:Relaxation.fInitialFinal}
		\end{minipage}
	} \\
	\subfloat[Relative change in moments of $f$ versus time]
	{\begin{minipage}{0.5\textwidth}
			\includegraphics[width=\linewidth]{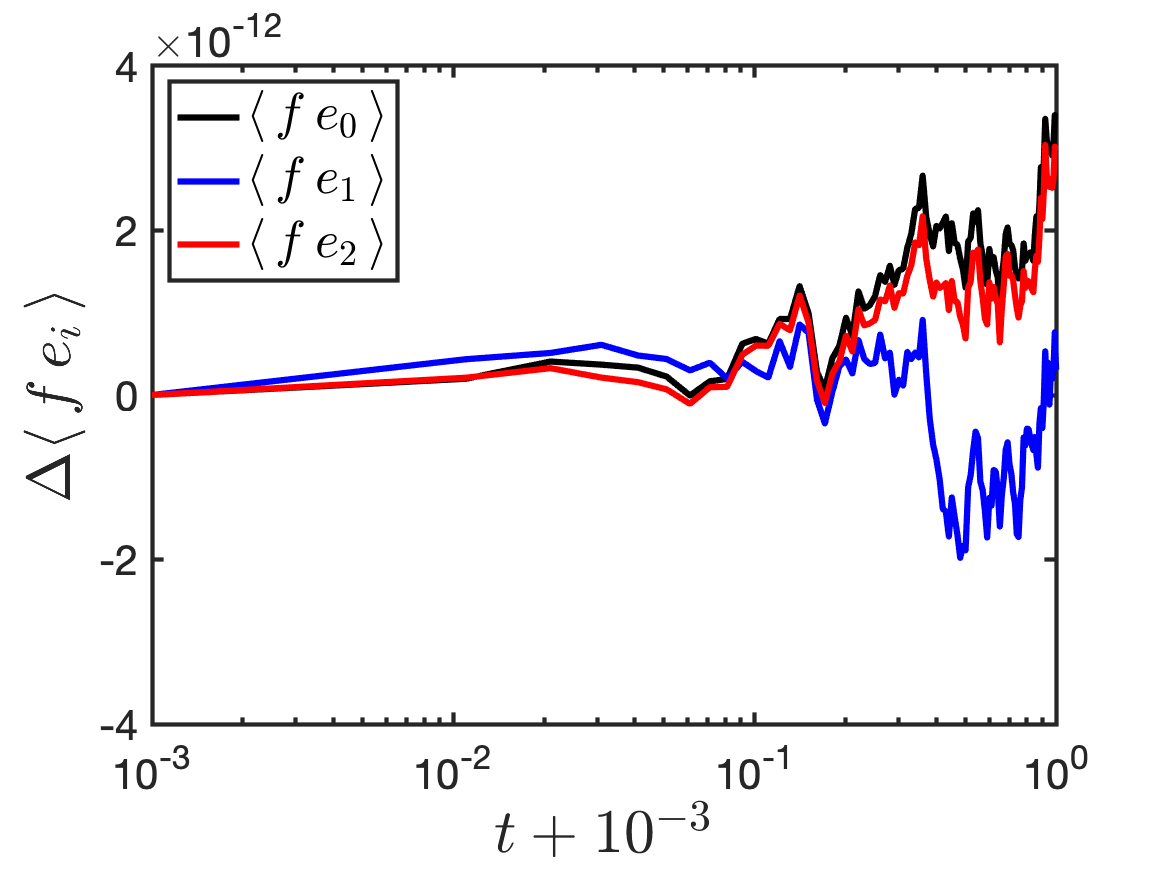}
			\label{fig:Relaxation.fMoments}
		\end{minipage}
	}
	\subfloat[Moments of $g$ versus time]
	{\begin{minipage}{0.5\textwidth}
			\includegraphics[width=\linewidth]{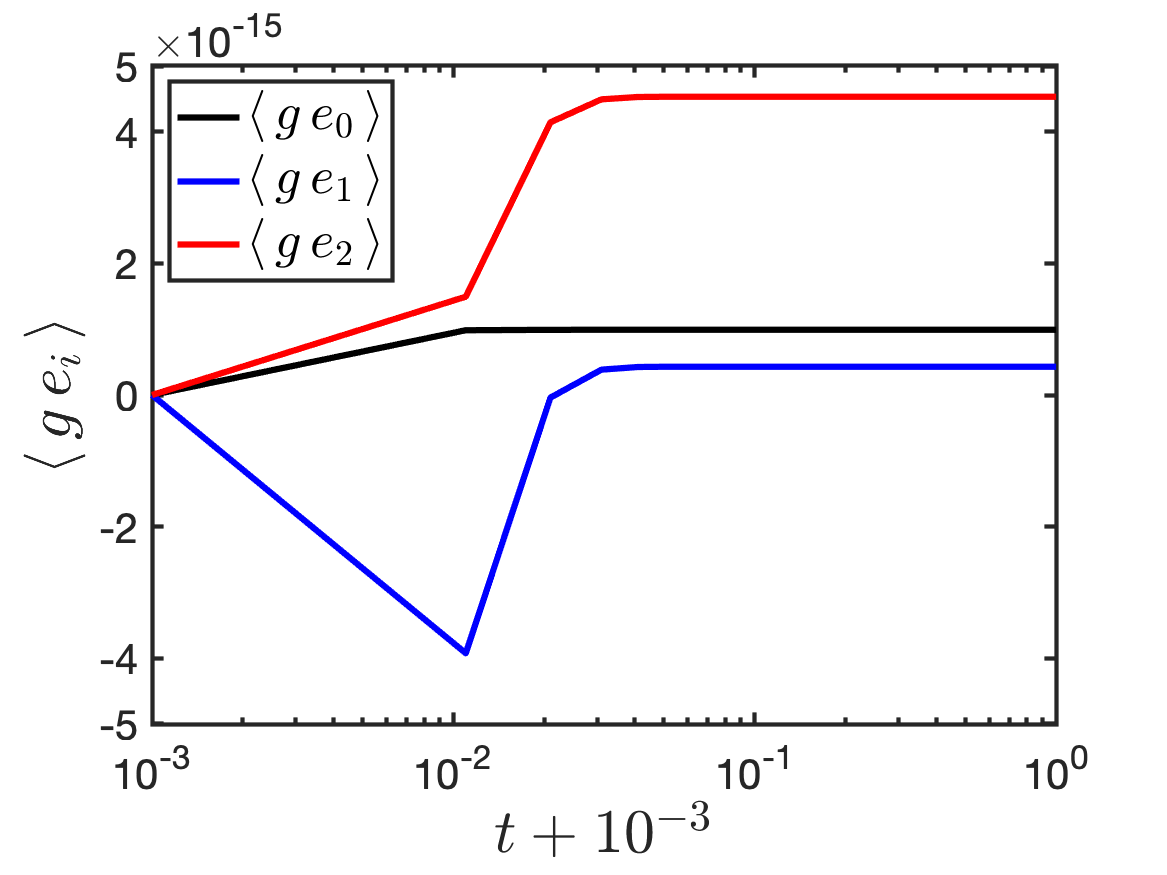}
			\label{fig:Relaxation.gMoments}
		\end{minipage}
	}
	\caption{Numerical results for the space-homogeneous problems in Eqs.~\eqref{eq:homogeneousDirect} and \eqref{eq:homogeneousMM}.  
	The top panel shows the initial distribution (dotted black line), and the final distribution obtained with the direct and mM methods (solid black and dashed red lines, respectively).  
	The lower left panel shows the time evolution of the relative change in the first three moments of the distribution function $f$ as obtained when using the direct method.  
	The lower right panel shows the time evolution of the first three moments of the micro distribution $g$ as obtained when using the mM method.  
	In the lower panels, the time axis has been offset by $10^{-3}$ to enable plotting on a logarithmic scale.}
	\label{fig:Relaxation}
	\end{centering}
\end{figure}

Figure~\ref{fig:Relaxation} shows results from solving the space-homogeneous systems in Eqs.~\eqref{eq:homogeneousDirect} and \eqref{eq:homogeneousMM}.  
The top panel shows the initial distribution, and final distributions obtained with the Direct and mM methods.  
The final distribution function, $f(v,t=1)$, is consistent with a Maxwellian evaluated with $n=2$, $u=0.5$, and $\theta=4.5$, and the direct and mM methods give practically identical results.  
The lower left panel shows the relative change in first three moments of the distribution, $\vint{f\be}$, versus time as obtained with the direct discretization of the LB collision operator.  
For all three moments, the relative change is of the order of $10^{-12}$ over the duration of the computation.  
(We find that the relative change in the moments decreases with decreasing values of the collision frequency.)  
The lower right panel shows the time evolution of the first three moments of the micro distribution, $\vint{g\be}$, obtained using the same discretization of the LB collision operator, but in the context of the mM model.  
All three components of $\vint{g\be}$ remain small (a few times $10^{-15}$) for the duration of the simulation.  
From these results, we conclude that the conservation properties of the discretized LB collision operator are satisfactory, and consistent with expectations from Section~\ref{sec:conservation}.  

\subsection{Riemann Problem}
\label{sec:numerical_Riemann}


We consider a Riemann problem in this section.  
The test involves both the collision operator and the transport operator (with the electric field set to zero), and we use the IMEX time-stepping schemes discussed in Sections~\ref{sec:imexrkDirect} and \ref{imexrkMM} for the direct and mM methods, respectively.  
Our goal is (1) to demonstrate conservation properties of the mM method, and (2) to compare the efficiency, in terms of accuracy for a given phase-space resolution, of the direct and mM methods in kinetic and fluid regimes.  

Unless stated otherwise,  the computational domain is given by $D^{x}=[-1.0,1.0]$ and $D^{v}=[-6,6]$.  
Following \cite{hakim_etal_2020}, we let the initial distribution function be given by a Maxwellian: $f(v,x,t=0)=f_{0}(v,x)=M[\bsrho](v)$, where
\begin{equation}
  (\,n,\,u,\,\theta\,)
  =\begin{cases}
    (\,1.0,\,0.0,\,1.0\,), & x \le 0 \\
    (\,0.125,\,0.0,\,0.8\,), & x> 0.
  \end{cases}
\end{equation}
These initial conditions are similar to the classical Riemann problem due to Sod \cite{sod_1978}, and in the fluid regime ($\nu\to\infty$) the solution consists of a rarefaction wave propagating to the left and a shock wave propagating to the right, followed by a contact discontinuity (also propagating to the right).  
In our numerical experiments, since the spatial boundaries are placed far enough away from the initial discontinuity at $x=0$, we use asymptotic boundary conditions at the spatial boundaries (i.e., the boundary values are given by the initial condition), while we use zero-flux conditions at the boundaries in the velocity domain.  
We evolve the Riemann problem until $t=0.1$.  

\begin{figure}[H]
	\begin{centering}
	\captionsetup[subfigure]{justification=centering}
	\subfloat[Density]
	{\begin{minipage}{0.5\textwidth}
			\includegraphics[width=\linewidth]{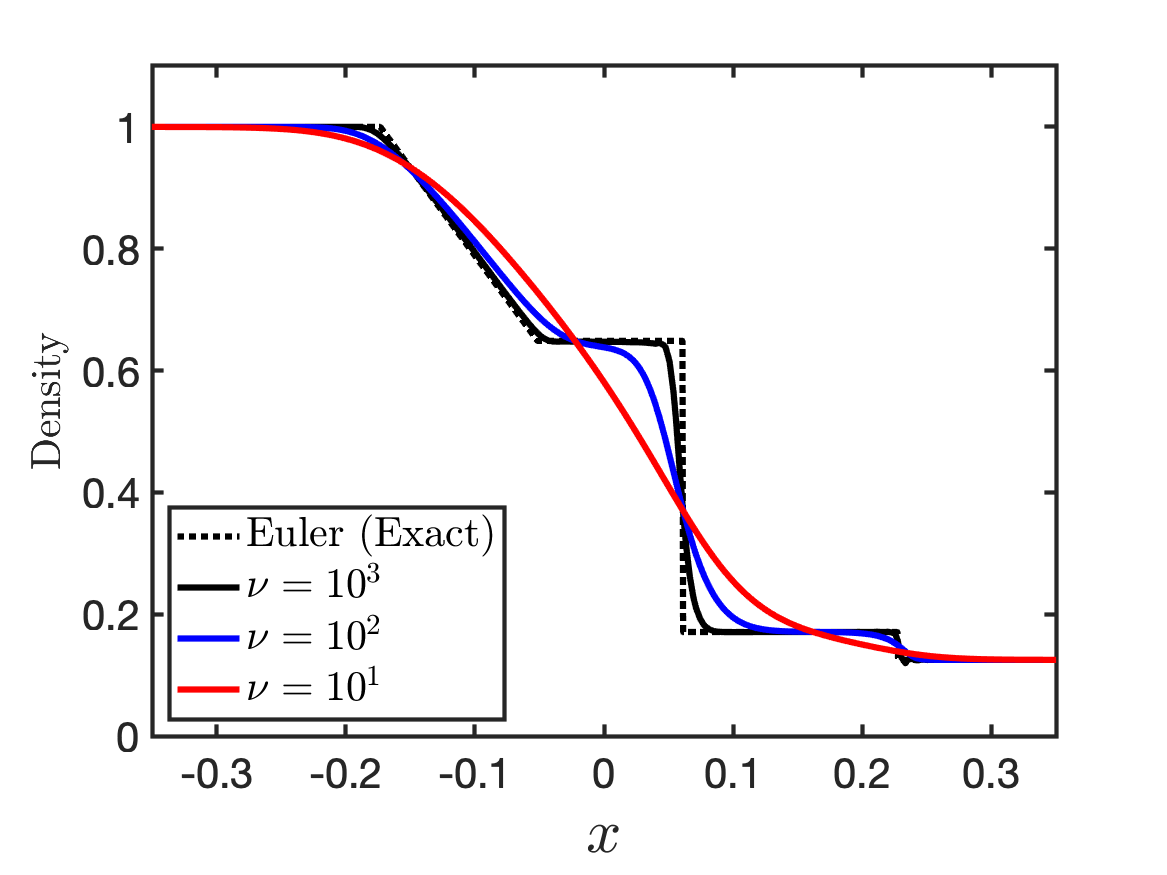}
			\label{fig:RiemannProblem_MM.Density}
		\end{minipage}
	}
	\subfloat[Velocity]
	{\begin{minipage}{0.5\textwidth}
			\includegraphics[width=\linewidth]{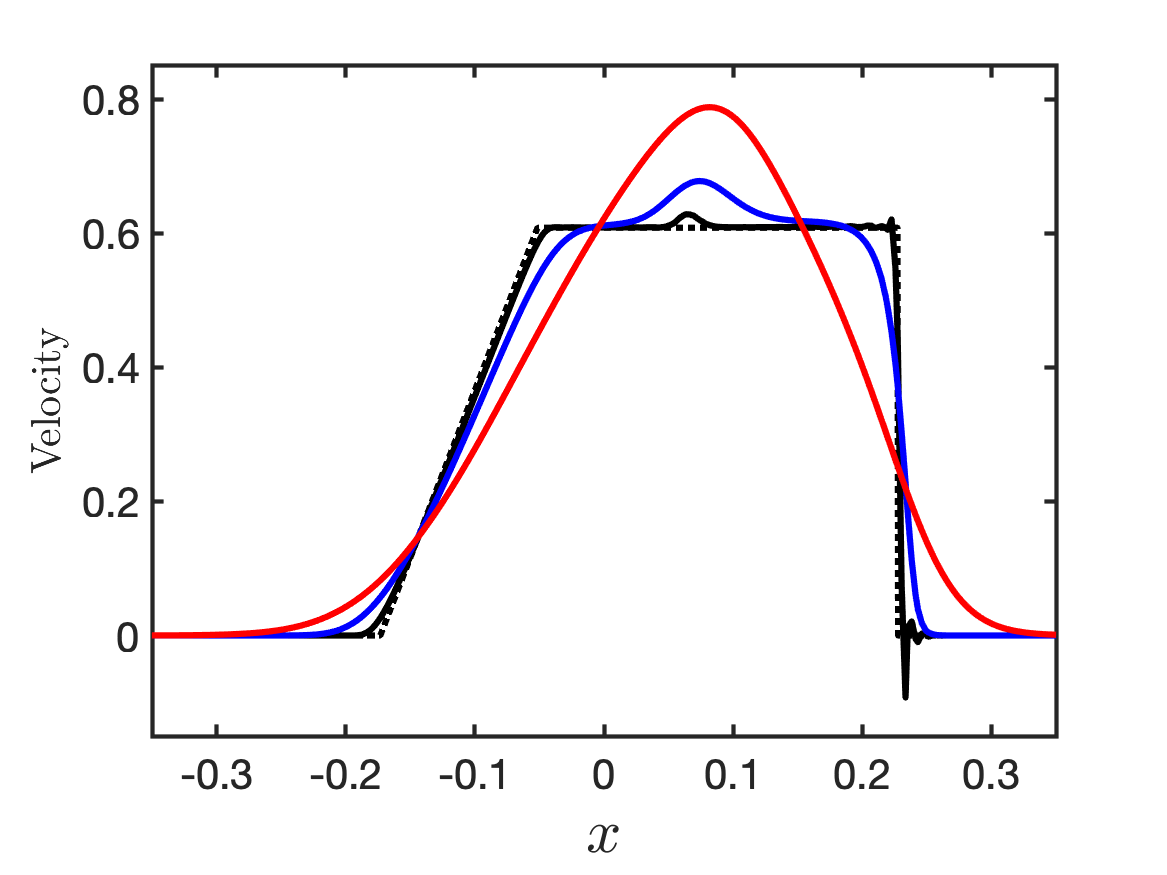}
			\label{fig:RiemannProblem_MM.Velocity}
		\end{minipage}
	} \\
	\subfloat[Temperature]
	{\begin{minipage}{0.5\textwidth}
			\includegraphics[width=\linewidth]{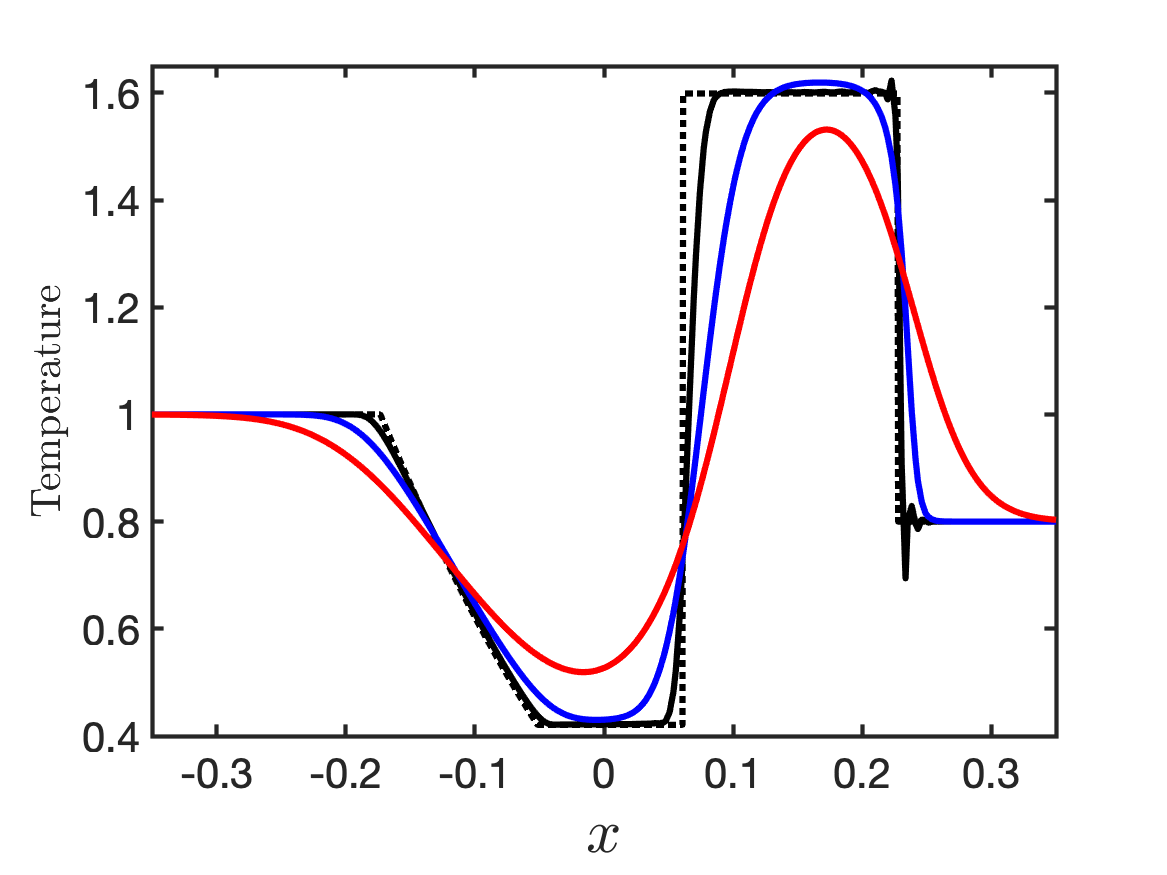}
			\label{fig:RiemannProblem_MM.Temperature}
		\end{minipage}
	}
	\subfloat[Conservation ($\nu=10^{3}$)]
	{\begin{minipage}{0.5\textwidth}
			\includegraphics[width=\linewidth]{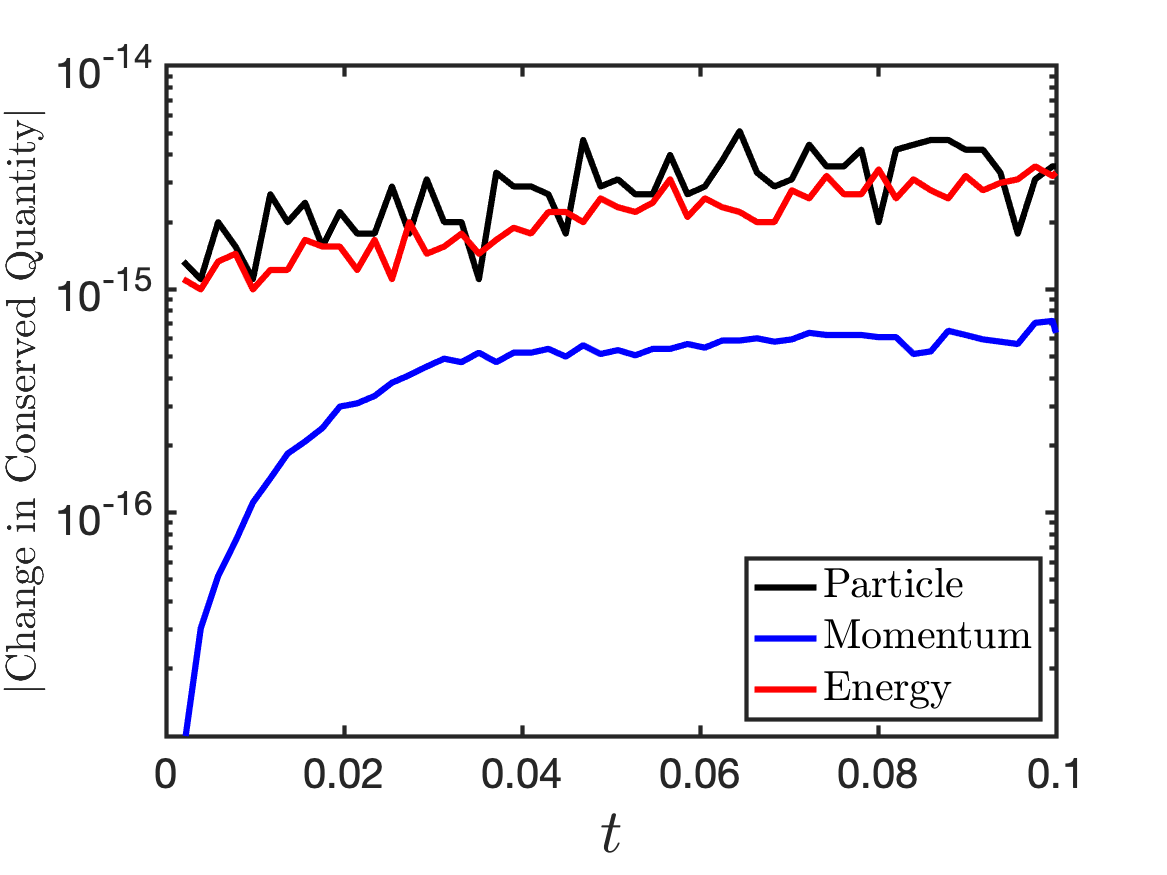}
			\label{fig:RiemannProblem_MM.conservation}
		\end{minipage}
	}
	\caption{Numerical results for the Riemann problem for various values of the collision frequency $\nu$, obtained with the mM method using $N^{x}\times N^{v}=256\times16$.  The density, velocity, and temperature at $t=0.1$ are plotted versus position $x$ in panels (a), (b), and (c), respectively.  In each of these panels, we plot results for $\nu=10^{1}$ (solid red), $\nu=10^{2}$ (solid blue), and $\nu=10^{3}$ (solid black).  For reference, we also plot the exact solution to the Riemann problem in the inviscid (Euler) limit ($\nu\to\infty$; dotted black).  In panel (d) the absolute change in particle number (black), momentum (blue), and energy (red) are plotted versus time for the case with $\nu=10^{3}$.}
	\label{fig:RiemannProblem_MM}
	\end{centering}
\end{figure}

Figure~\ref{fig:RiemannProblem_MM} shows results obtained with the mM method for various values of the collision frequency, using $N^{x}=256$ and $N^{v}=16$.  
These results show that the solutions obtained with the mM method tend to the inviscid Euler solution (dotted black line) as the collision frequency increases.  
For $\nu=10^{3}$, the mM solution is quite close to the Euler solution.  
We also note that oscillations are present in the solution around the shock ($x\approx0.23$) when $\nu=10^{3}$, as can be expected when limiters designed to suppress such oscillations, e.g., local projection limiters \cite{cockburn_etal_1989}, are not applied.  
However, the cleaning limiter from Section~\ref{sec:cleaningLimiter} is applied so that the integrated moments of the micro distribution $\int_{D^{x}}\vint{\be g}_{D^{v}}dx$ remain zero (to machine precision).  
Then, as seen in the lower right panel of Figure~\ref{fig:RiemannProblem_MM}, for $\nu=10^{3}$, the change in particle number, momentum, and energy are also at the level of machine precision.  
(We find similar results for the other values of $\nu$.)  

\begin{figure}[H]
	\begin{centering}
	\captionsetup[subfigure]{justification=centering}
	\subfloat[Difference between direct and mM methods ($\nu=10^{1}$)]
	{\begin{minipage}{0.5\textwidth}
			\includegraphics[width=\linewidth]{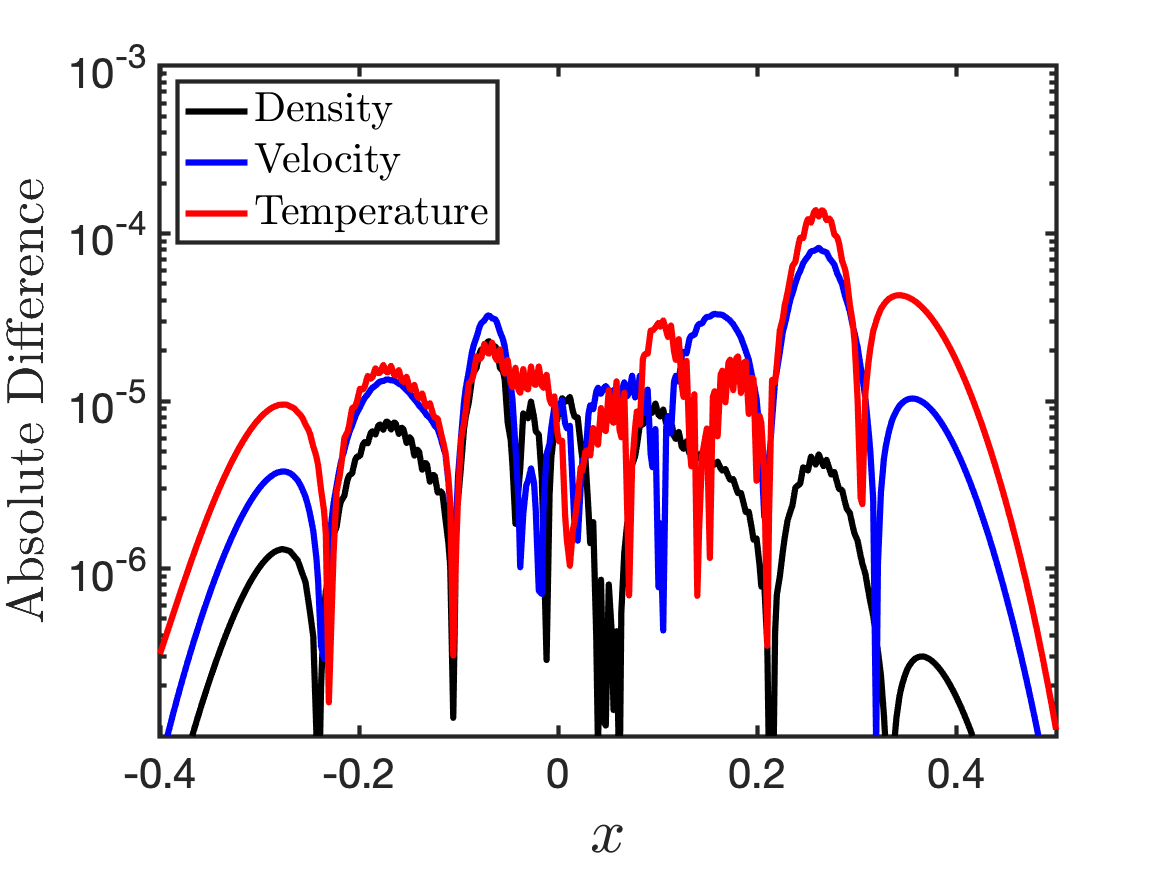}
			\label{fig:RiemannProblem_Direct_mM_Nu_1e1.Comparison}
		\end{minipage}
	}
	\subfloat[Difference between direct and mM methods ($\nu=10^{3}$)]
	{\begin{minipage}{0.5\textwidth}
			\includegraphics[width=\linewidth]{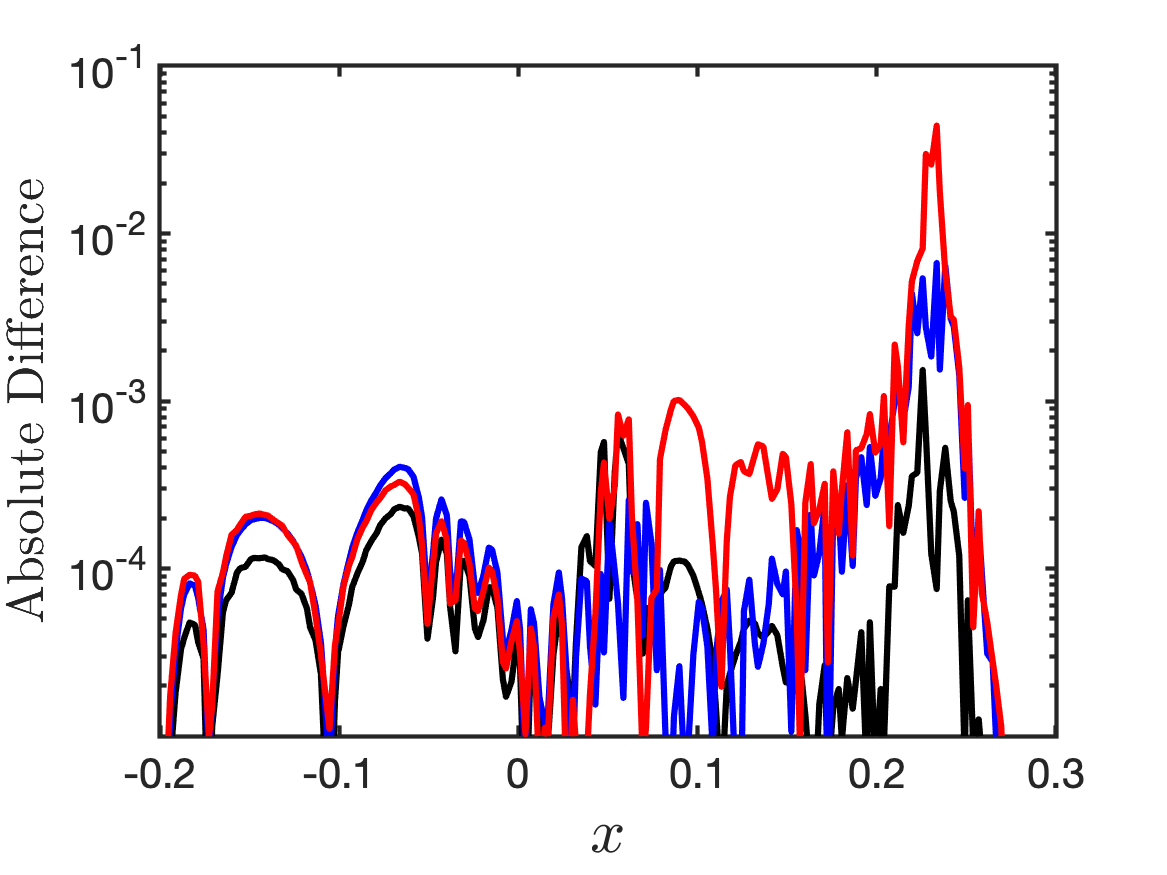}
			\label{fig:RiemannProblem_Direct_mM_Nu_1e3.Comparison}
		\end{minipage}
	}
	\caption{Numerical results for the Riemann problem at $t=0.1$, obtained with the direct and mM (with cleaning limiter applied) methods using $N^{x}\times N^{v}=256\times16$.  
	We plot the absolute difference in density (black), velocity (blue), and temperature (red), for $\nu=10^{1}$ and $\nu=10^{3}$.}
	\label{fig:RiemannProblem_Direct_mM}
	\end{centering}
\end{figure}

Figure~\ref{fig:RiemannProblem_Direct_mM} shows a comparison of results obtained with the mM method versus results obtained with the direct method, for $\nu=10^{1}$ and $\nu=10^{3}$.  
The results indicate good consistency between the two methods.  
For the case with $\nu=10^{1}$, the absolute difference in any quantity is less than $2\times10^{-4}$, anywhere in the spatial domain.  
The absolute difference is larger for the case with $\nu=10^{3}$, with the largest difference reaching a few~$\times10^{-2}$ around the shock at $x\approx0.23$.  
(We have verified that the differences decrease with increasing phase-space resolution.)  
We find that the conservation properties of the direct method are similar to that of the mM method: the absolute change in the particle number and energy are on the level of $10^{-14}$ to $10^{-13}$ for both values of the collision frequency.  

Next, we investigate in further detail the discretization of the mM model and the effect of applying the cleaning limiter after the explicit steps in the IMEX time integration scheme.  
As discussed in Section~\ref{sec:conservationMM} (see Proposition~\ref{prop:zeroMicroMomentsExplicit} and Remarks~\ref{rem:consistentDiscretization} and \ref{rem:microMomentsConservationErrors}), the ability to maintain the zero moment constraints in Eq.~\eqref{eq:gConstraints} relies on a \emph{consistent} discretization of the terms in Eqs.~\eqref{eq:macroSemiDiscrete} and \eqref{eq:microSemiDiscrete}, so as to achieve the necessary cancellation of terms in Eq.~\eqref{eq:GammaElaborate}.  
To demonstrate the importance of a consistent discretization, we introduce an \emph{inconsistent} discretization, where the velocity integrals involving the Maxwellian in Eq.~\eqref{eq:microBLF} are approximated with a 3-point LG quadrature, instead of analytically as is done in the consistent discretization.  
(For these integrals, the 3-point LG quadrature would be exact if the Maxwellian was replaced by a polynomial of degree $\le p$, but when integrating the Maxwellian it is not.  
Moreover, this quadrature approximation to the Maxwellian integrals becomes increasingly worse as the velocity elements are coarsened.)  

By design, the inconsistent discretization results in nonzero moments of the micro distribution, which then leads to an ambiguous interpretation of the conservation properties of the mM method.  
On one hand, the macro model evolves the conserved quantities $\bsrho_{f}$, and, analytically, these are equal to the moments of the kinetic distribution $\vint{f\be}=\vint{\cE[\bsrho_{f}]\be}+\vint{g\be}=\bsrho_{f}$.  
However, if $\vint{g\be}\ne0$, then $\bsrho_{f}\ne\vint{f\be}$.  
The ambiguity in conservation properties can be removed with the cleaning limiter.  

Figure~\ref{fig:RiemannProblem_cleanVsNoClean_Nu_1e4.Quadrature.Conservation} shows the time evolution of the change in the moments $\bsrho_{f}$ and $\vint{f\be}$, integrated over the spatial domain, for two models using the inconsistent discretization: one without the cleaning limiter, and one with the cleaning limiter.  
These models were computed with $\nu=10^{4}$, using a coarse velocity grid ($N^{x}\times N^{v}=256\times4$).  
For the model without cleaning, conservation as measured by $\bsrho_{f}$ is at the level of machine precision, while it is substantially worse when measured by $\vint{f\be}$.  
In this measure, particle conservation is on the order of $10^{-7}$, while momentum and energy conservation are on the order of $10^{-5}$, and growing at the end of the simulation.  
Similar ambiguous conservation properties were reported in \cite{gamba_etal_2019}, although for a different (smooth) problem.  
With cleaning, the ambiguity is removed, and conservation is at the level of machine precision in both measures.  

\begin{figure}[H]
	\begin{centering}
	\captionsetup[subfigure]{justification=centering}
	\vspace{-24pt}
	\subfloat[Inconsistent discretization without cleaning.]
	{\begin{minipage}{0.8\textwidth}
			\includegraphics[width=\linewidth]{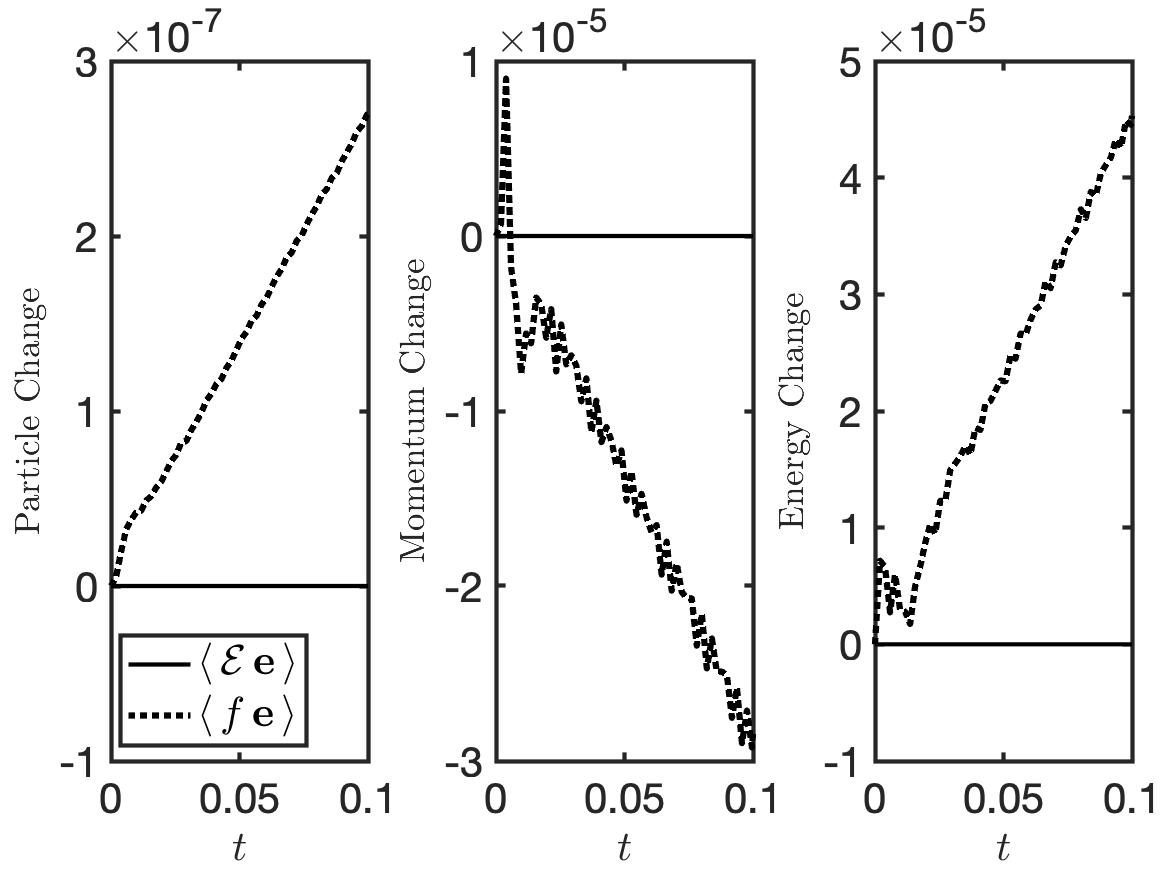}
			\label{fig:RiemannProblem_cleanVsNoClean_Nu_1e4.Quadrature.noClean.Conservation}
		\end{minipage}
	} \\
	\subfloat[Inconsistent discretization with cleaning.]
	{\begin{minipage}{0.8\textwidth}
			\includegraphics[width=\linewidth]{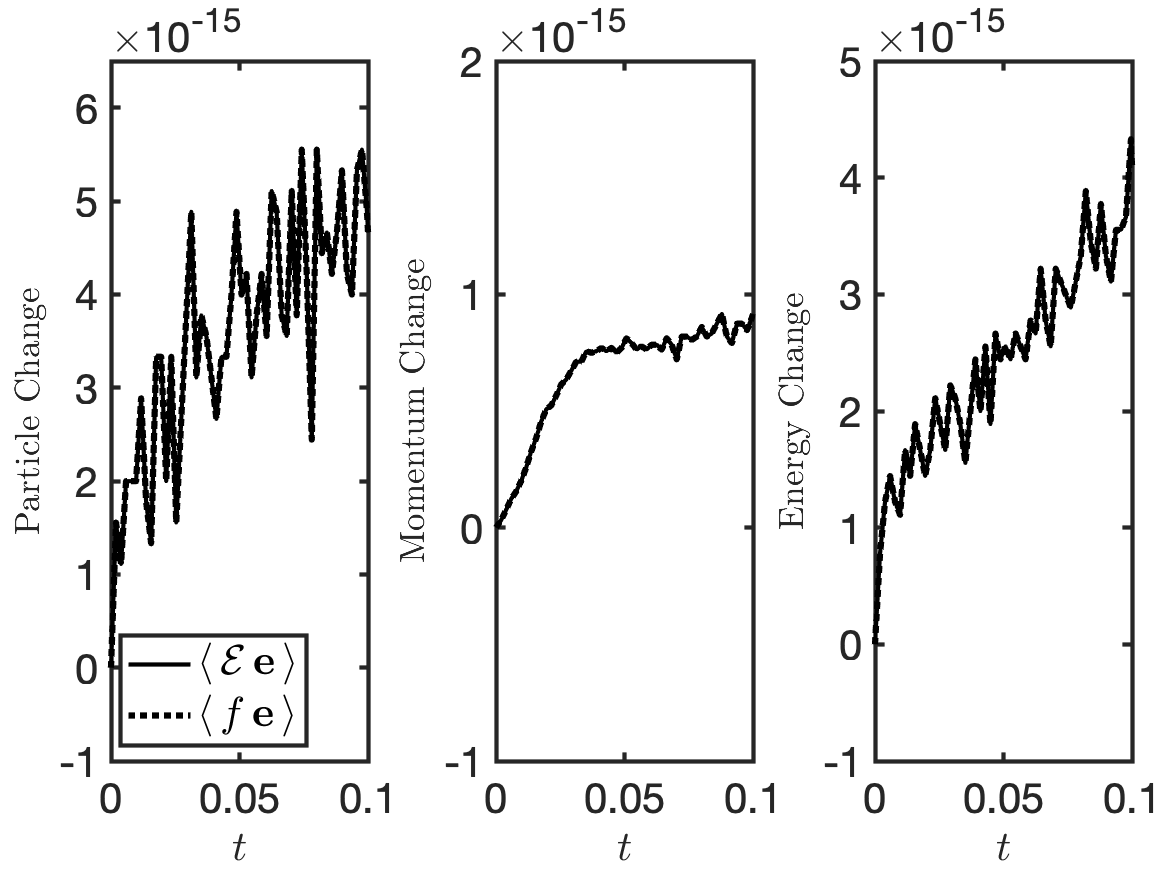}
			\label{fig:RiemannProblem_cleanVsNoClean_Nu_1e4.Quadrature.Clean.Conservation}
		\end{minipage}
	}
	\caption{Conservation properties for a Riemann problem with $\nu=10^{4}$ and $N^{x}\times N^{v}=256\times4$, as obtained with inconsistent discretization of the mM method (see text for details).  
	The change in particle number (left panels), momentum (middle panels), and energy (right panels) is plotted versus time.  
	Results obtained with and without the cleaning limiter are displayed in the bottom and top panels, respectively.  
	In each panel, we plot the change in the conserved quantity as obtained by the macro fields, $\int_{D^{x}}\bsrho_{f} dx$ (solid lines), and the kinetic distribution, $\int_{D^{x}}\vint{f\be}_{D^{v}}dx$ (dotted lines).  }
	\label{fig:RiemannProblem_cleanVsNoClean_Nu_1e4.Quadrature.Conservation}
	\end{centering}
\end{figure}

The seemingly small conservation inconsistency displayed by the model without cleaning in Figure~\ref{fig:RiemannProblem_cleanVsNoClean_Nu_1e4.Quadrature.Conservation} is due to uncontrolled growth in $\vint{g \be}$, induced by the inconsistent discretization.  
We then find that, when the velocity space resolution is coarse, the moments of $g$ may become too large and the solution accuracy adversely impacted.  

Figure~\ref{fig:RiemannProblem_cleanVsNoClean} shows results for the Riemann problem at $t=0.1$ from runs performed with the inconsistent discretization shown in Figure~\ref{fig:RiemannProblem_cleanVsNoClean_Nu_1e4.Quadrature.Conservation}.  
Figure~\ref{fig:RiemannProblem_cleanVsNoClean_Nu_1e4.Quadrature.noClean} shows the particle density and $\vint{g e_{0}}$ versus position for the model without the cleaning limiter.  
In this case, $\vint{g e_{0}}$ develops a substantial magnitude (up to $6.6\times10^{-2}$), which in turn introduces artifacts into the particle density.
Figure~\ref{fig:RiemannProblem_cleanVsNoClean_Nu_1e4.Quadrature} shows the same quantities as in Figure~\ref{fig:RiemannProblem_cleanVsNoClean_Nu_1e4.Quadrature.noClean} for the model with the cleaning limiter applied.  
In this case, the magnitude of $\vint{g e_{0}}$ remains small ($\sim10^{-18}$, or smaller), and the accuracy of the particle density relative to the exact solution is restored.  
We note that with higher resolution in velocity space, the results obtained using inconsistent discretization without cleaning do improve.  

\begin{figure}[H]
	\begin{centering}
	\captionsetup[subfigure]{justification=centering}
	\subfloat[Inconsistent discretization without cleaning.]
	{\begin{minipage}{0.5\textwidth}
			\includegraphics[width=\linewidth]{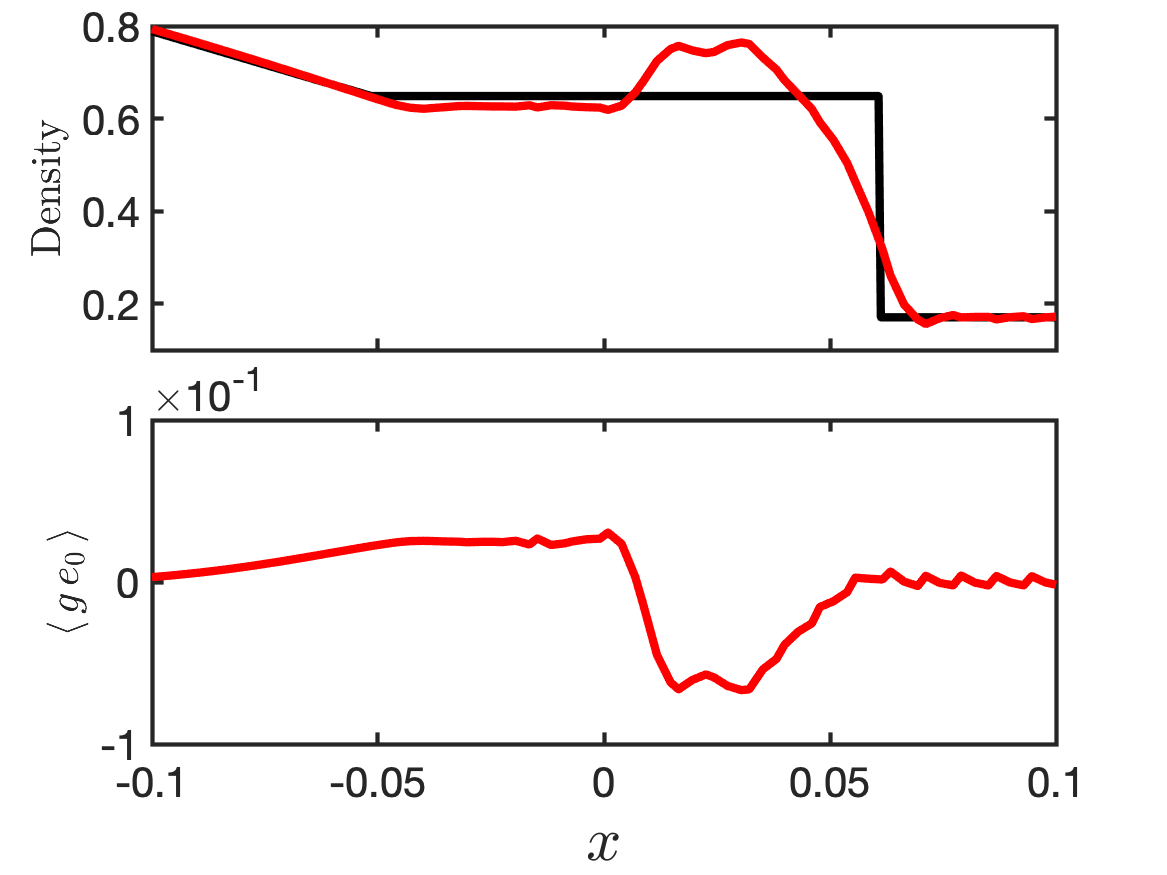}
			\label{fig:RiemannProblem_cleanVsNoClean_Nu_1e4.Quadrature.noClean}
		\end{minipage}
	}
	\subfloat[Inconsistent discretization with cleaning.]
	{\begin{minipage}{0.5\textwidth}
			\includegraphics[width=\linewidth]{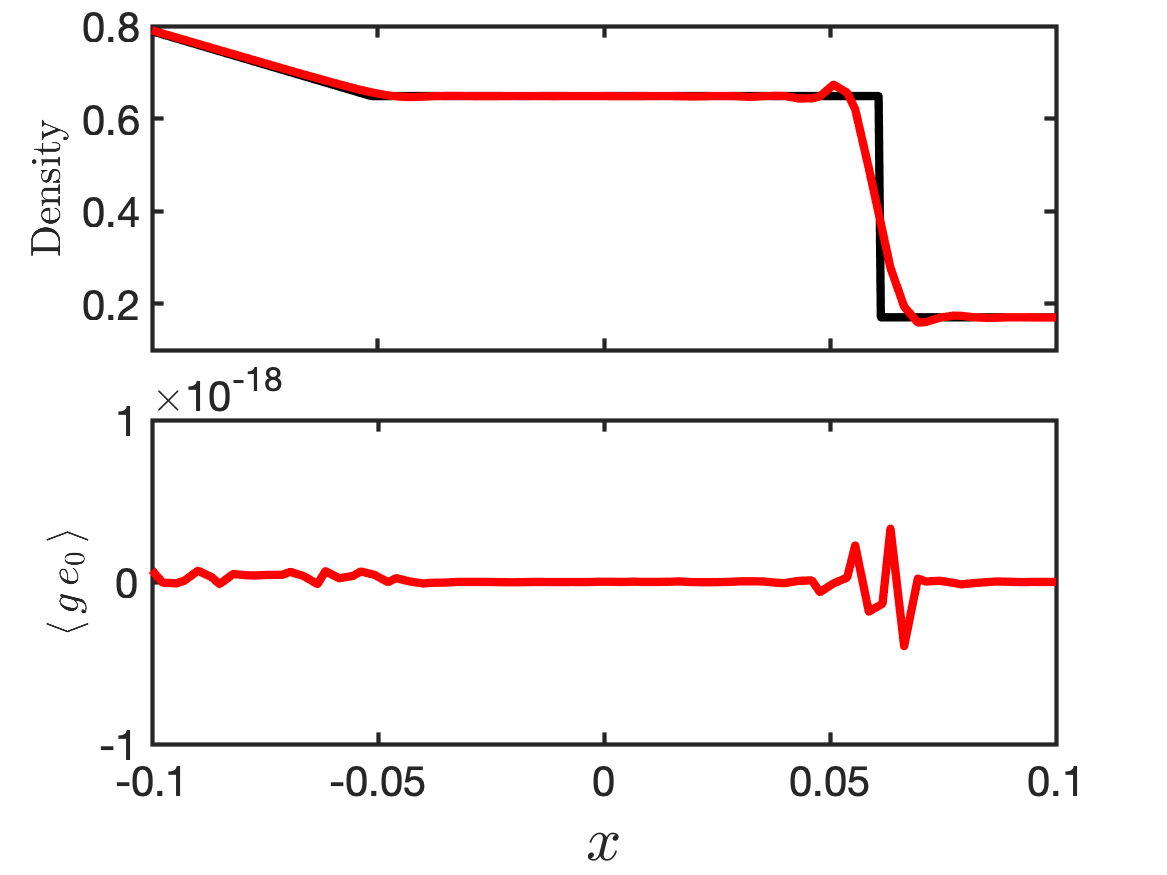}
			\label{fig:RiemannProblem_cleanVsNoClean_Nu_1e4.Quadrature}
		\end{minipage}
	}
	\caption{Numerical results for the Riemann problem at $t=0.1$, obtained with the \emph{inconsistent} mM methods also displayed in Figure~\ref{fig:RiemannProblem_cleanVsNoClean_Nu_1e4.Quadrature.Conservation}.  
	In the left panels we plot the density ($n_{f}$; red line, top panel) and the zeroth moment of $g$ ($\vint{g e_{0}}$; bottom panel) versus position for the model computed \emph{without} the cleaning limiter.  
	For reference, the exact solution to the Riemann problem in the Euler limit is also plotted (black line).  
	The right panels show the same quantities for the model computed \emph{with} the cleaning limiter.}
	\label{fig:RiemannProblem_cleanVsNoClean}
	\end{centering}
\end{figure}

Figure~\ref{fig:RiemannProblem_cleanVsNoClean_Nu_1e4.Consistent} shows results obtained using the \emph{consistent} discretization without the cleaning limiter for the case with $\nu=10^{4}$.  
As noted in Remark~\ref{rem:microMomentsConservationErrors}, the consistent discretization can still give rise to nonzero $\vint{g\be}$ due to finite velocity domain effects.  
Because of this, we have computed two models using the consistent discretization: one using the fiducial domain $D^{v}=[-6,6]$, with $N^{x}\times N^{v}=256\times4$ (same as for the results displayed in Figures~\ref{fig:RiemannProblem_cleanVsNoClean}), and one using an extended velocity domain $D^{v}=[-12,12]$, with $N^{x}\times N^{v}=256\times8$ so that $\dv$ is unchanged.  
These results are displayed in Figure~\ref{fig:RiemannProblem_cleanVsNoClean_Nu_1e4.Consistent.noClean}.  
The model computed with the consistent discretization and the fiducial velocity domain results in a significant reduction in $\vint{g e_{0}}$ (down to about $2\times10^{-4}$), when compared to the results obtained with the inconsistent discretization in Figure~\ref{fig:RiemannProblem_cleanVsNoClean_Nu_1e4.Quadrature.noClean}.  
With the extended velocity domain, $\vint{g e_{0}}$ is further reduced to a magnitude of less than $2\times10^{-13}$ across the spatial domain.  
Moreover, for both velocity domains, the accuracy relative to the exact solution is comparable to that displayed in Figure~\ref{fig:RiemannProblem_cleanVsNoClean_Nu_1e4.Quadrature}.  
As was also noted in Remark~\ref{rem:microMomentsConservationErrors}, the finite velocity domain effects can be eliminated by replacing $I_{j}^{v}$ with $\tilde{I}_{j}^{v}$, defined in Eq.~\eqref{eq:infiniteVelocityDomain}, so that Maxwellians in Eq.~\eqref{eq:microSemiDiscrete} are defined on the infinite velocity domain and terms I, II, III, and VII in Eq.~\eqref{eq:GammaElaborate} vanish --- even when $g_{h}$ is defined on a finite velocity domain.  
Figure~\ref{fig:RiemannProblem_cleanVsNoClean_Nu_1e4.Consistent.noClean.infiniteDomain} shows results obtained with $D^{v}=[-6,6]$ and $I_{j}^{v}\to\tilde{I}_{j}^{v}$.  
For this run, the density is practically identical to that displayed in Figure~\ref{fig:RiemannProblem_cleanVsNoClean_Nu_1e4.Consistent.noClean}, but $\vint{g e_{0}}$ is reduced to the level of machine precision.  
These results illustrate the importance of controlling/preventing spurious growth in the moments of the micro distribution, in order to obtain accurate results in the fluid regime when using coarse grids in velocity space.  

\begin{figure}[H]
	\begin{centering}
	\captionsetup[subfigure]{justification=centering}
	\subfloat[Consistent discretization without cleaning on fiducial and extended velocity domains.]
	{\begin{minipage}{0.5\textwidth}
			\includegraphics[width=\linewidth]{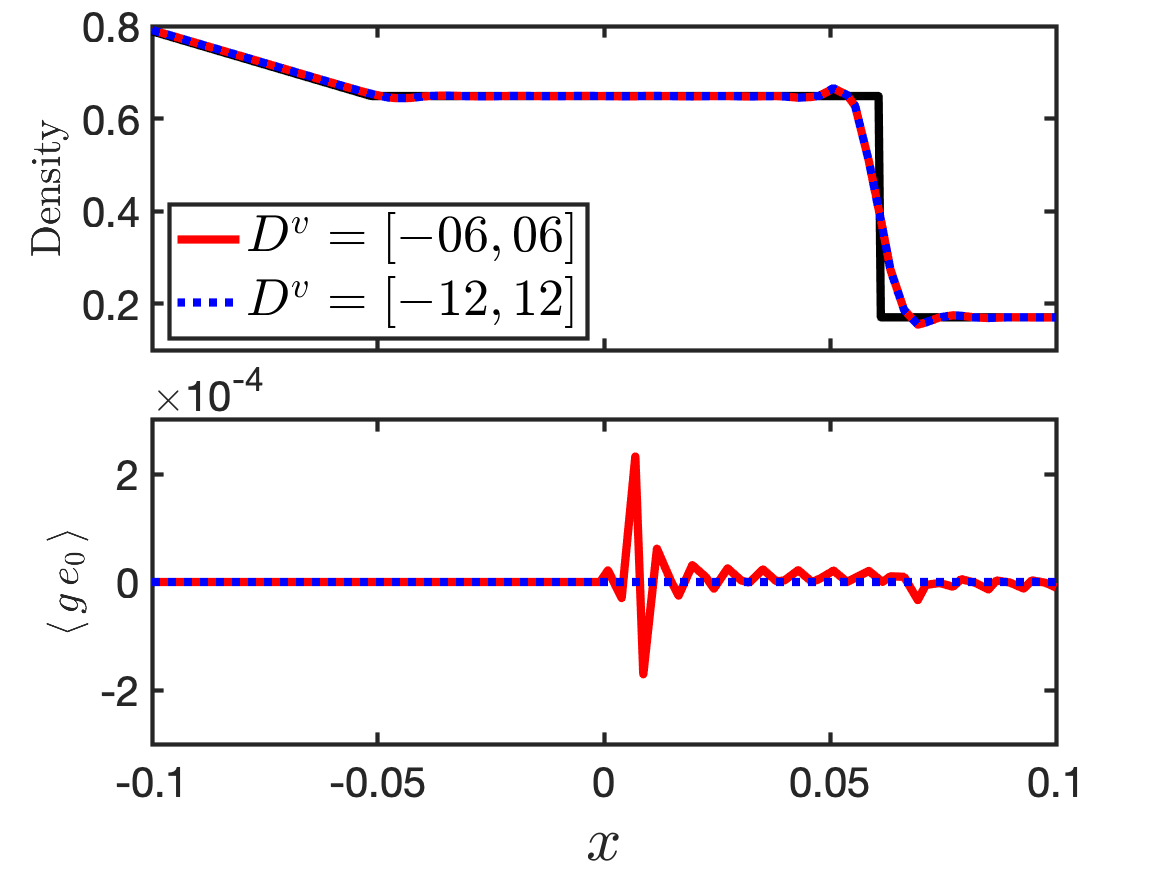}
			\label{fig:RiemannProblem_cleanVsNoClean_Nu_1e4.Consistent.noClean}
		\end{minipage}
	}
	\subfloat[Consistent discretization without cleaning on the fiducial velocity domain and $I_{j}^{v}\to\tilde{I}_{j}^{v}$.]
	{\begin{minipage}{0.5\textwidth}
			\includegraphics[width=\linewidth]{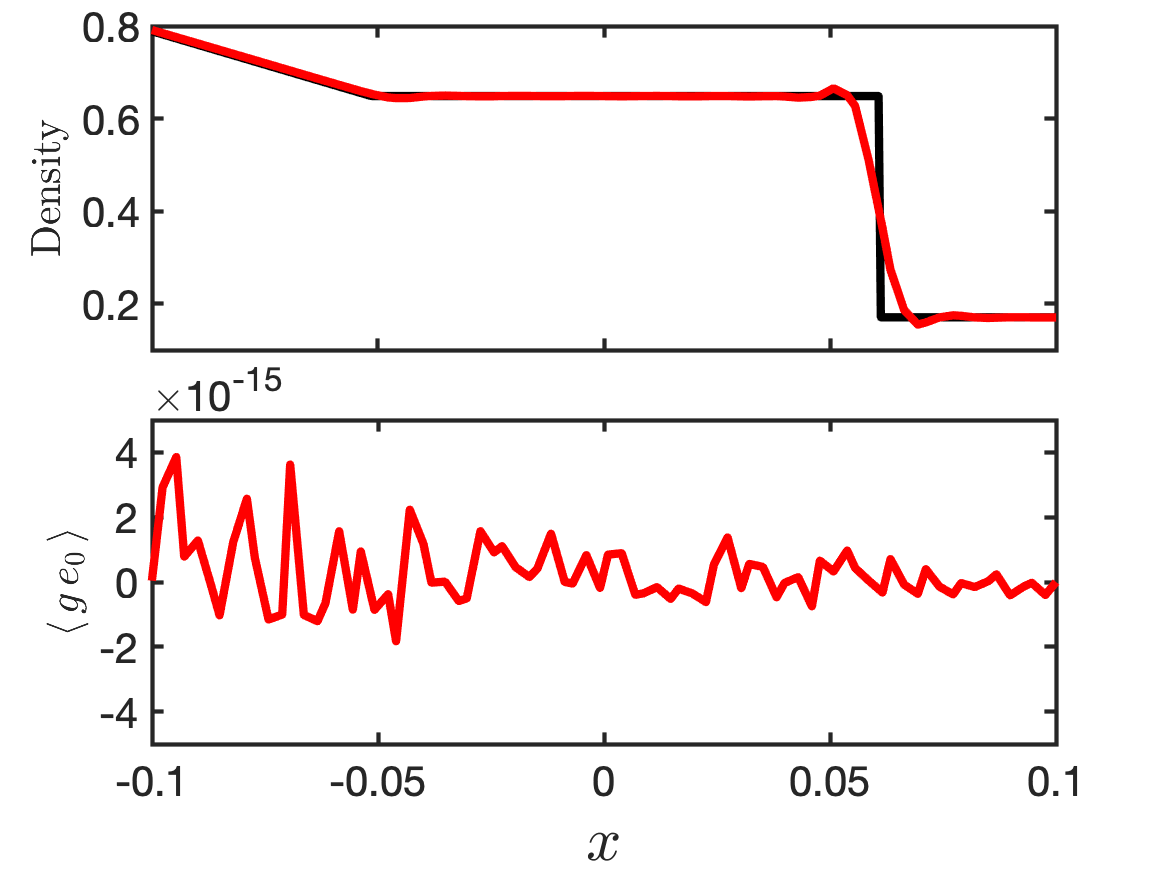}
			\label{fig:RiemannProblem_cleanVsNoClean_Nu_1e4.Consistent.noClean.infiniteDomain}
		\end{minipage}
	}
	\caption{Numerical results for the Riemann problem with $\nu=10^{4}$ at $t=0.1$, obtained with the consistent mM method without the cleaning limiter.  
	We plot the same quantities as in Figure~\ref{fig:RiemannProblem_cleanVsNoClean}.  
	In the left panel we plot results for a model computed with the fiducial velocity domain ($D^{v}=[-6,6]$; solid red lines) and a model computed with an extended velocity domain ($D^{v}=[-12,12]$; dotted blue lines).  
	In the right panel we plot results for a model computed with the fiducial velocity domain, but with $I_{j}^{v}$ replaced with $\tilde{I}_{j}^{v}$ (see Eq.~\eqref{eq:infiniteVelocityDomain}), as suggested in Remark~\ref{rem:microMomentsConservationErrors}.}
	\label{fig:RiemannProblem_cleanVsNoClean_Nu_1e4.Consistent}
	\end{centering}
\end{figure}

\begin{figure}[H]
	\captionsetup[subfigure]{justification=centering}
	\subfloat[$\nu=10^{1}$]
	{\begin{minipage}{0.5\textwidth}
			\includegraphics[width=\linewidth]{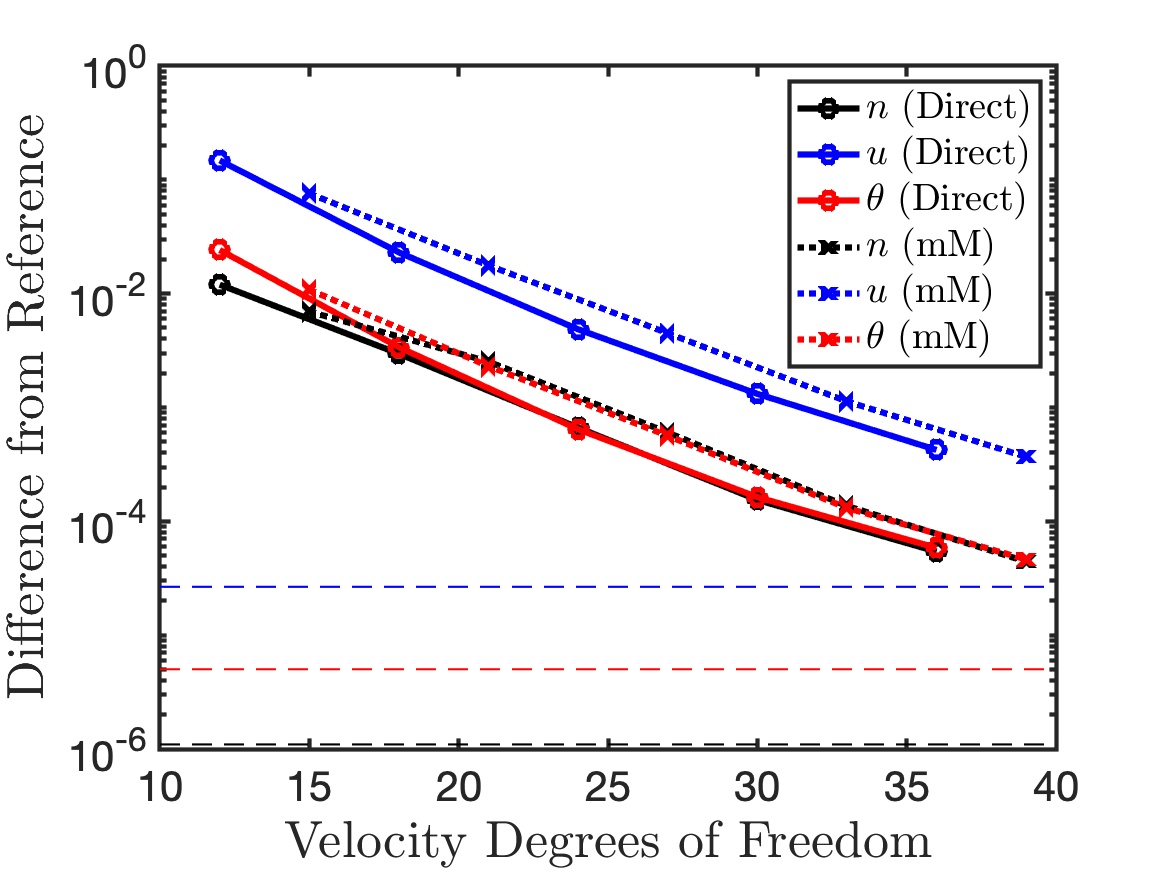}
			\label{fig:RiemannProblem_CompareMoments.1e1}
		\end{minipage}
	}
	\subfloat[$\nu=10^{2}$]
	{\begin{minipage}{0.5\textwidth}
			\includegraphics[width=\linewidth]{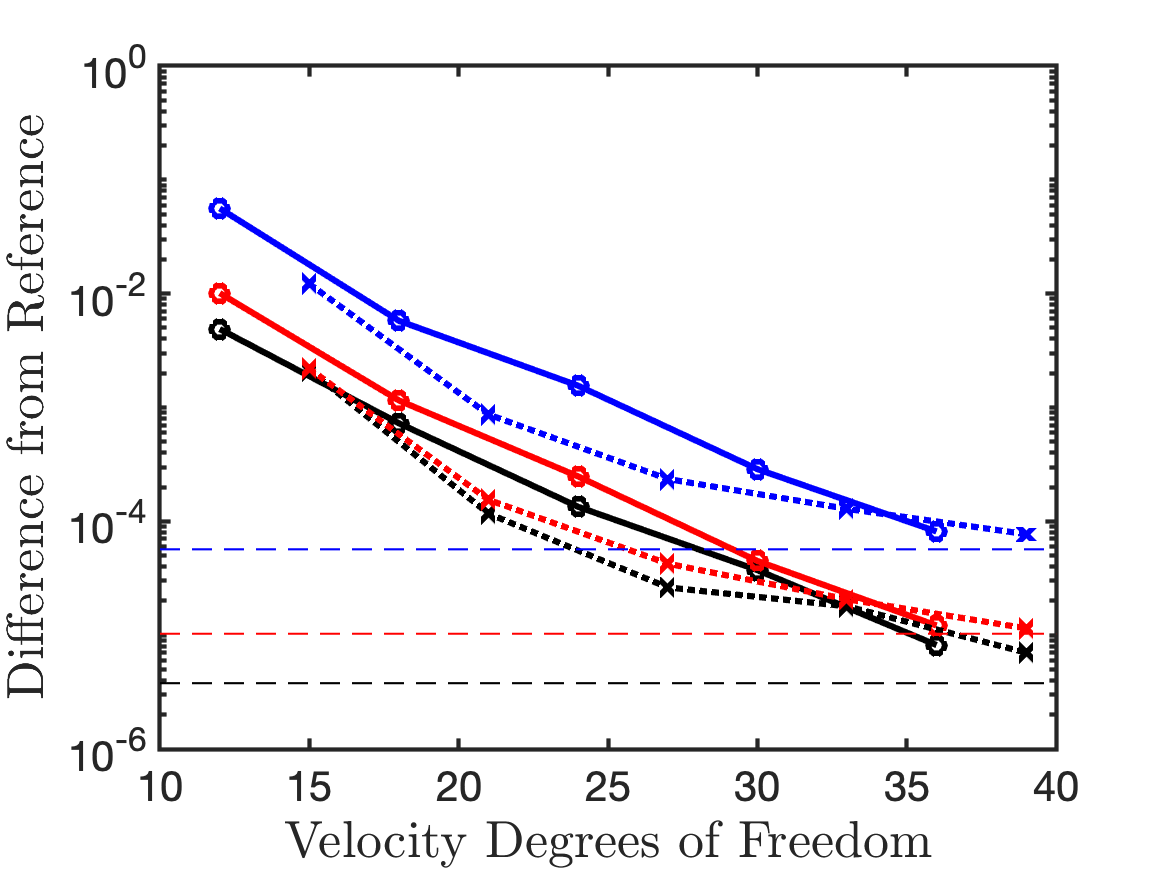}
			\label{fig:RiemannProblem_CompareMoments.1e2}
		\end{minipage}
	} \\
	\subfloat[$\nu=10^3$]
	{\begin{minipage}{0.5\textwidth}
			\includegraphics[width=\linewidth]{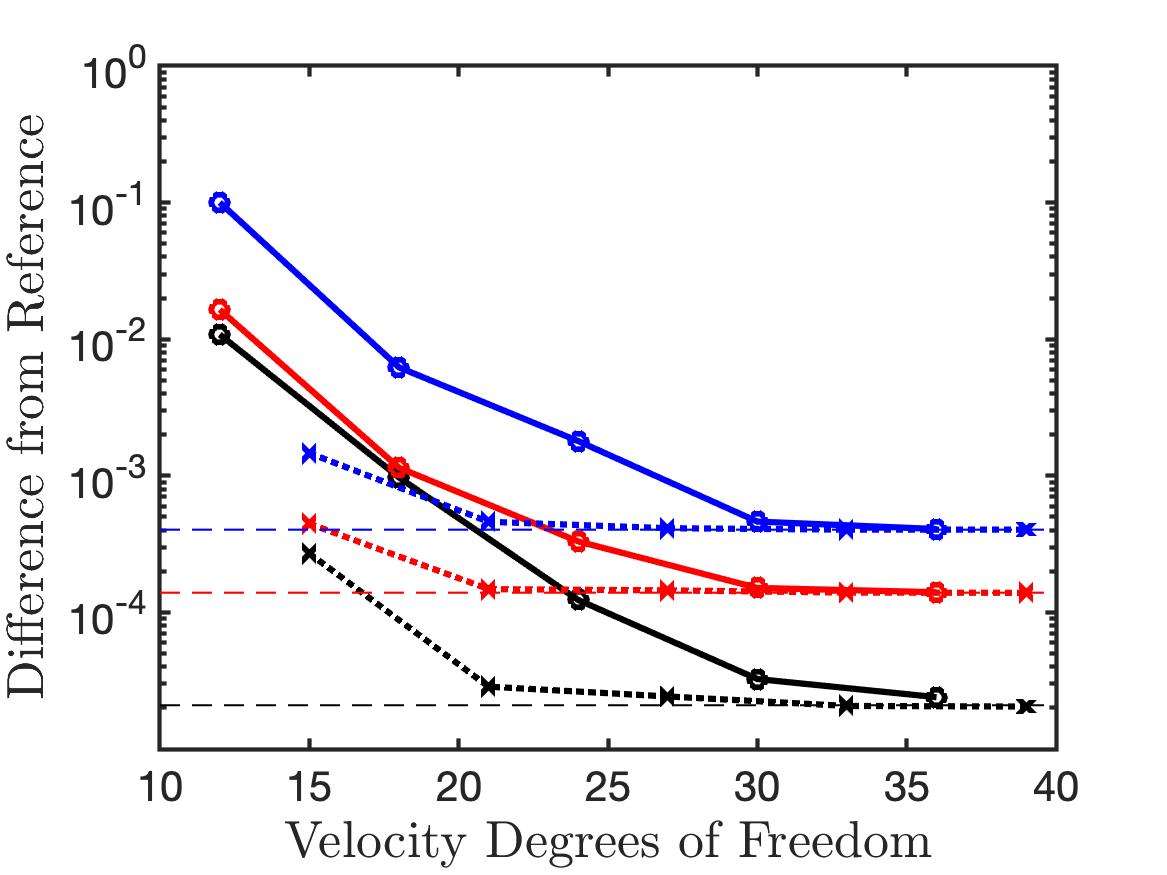}
			\label{fig:RiemannProblem_CompareMoments.1e3}
		\end{minipage}
	}
	\subfloat[$\nu=10^{4}$]
	{\begin{minipage}{0.5\textwidth}
			\includegraphics[width=\linewidth]{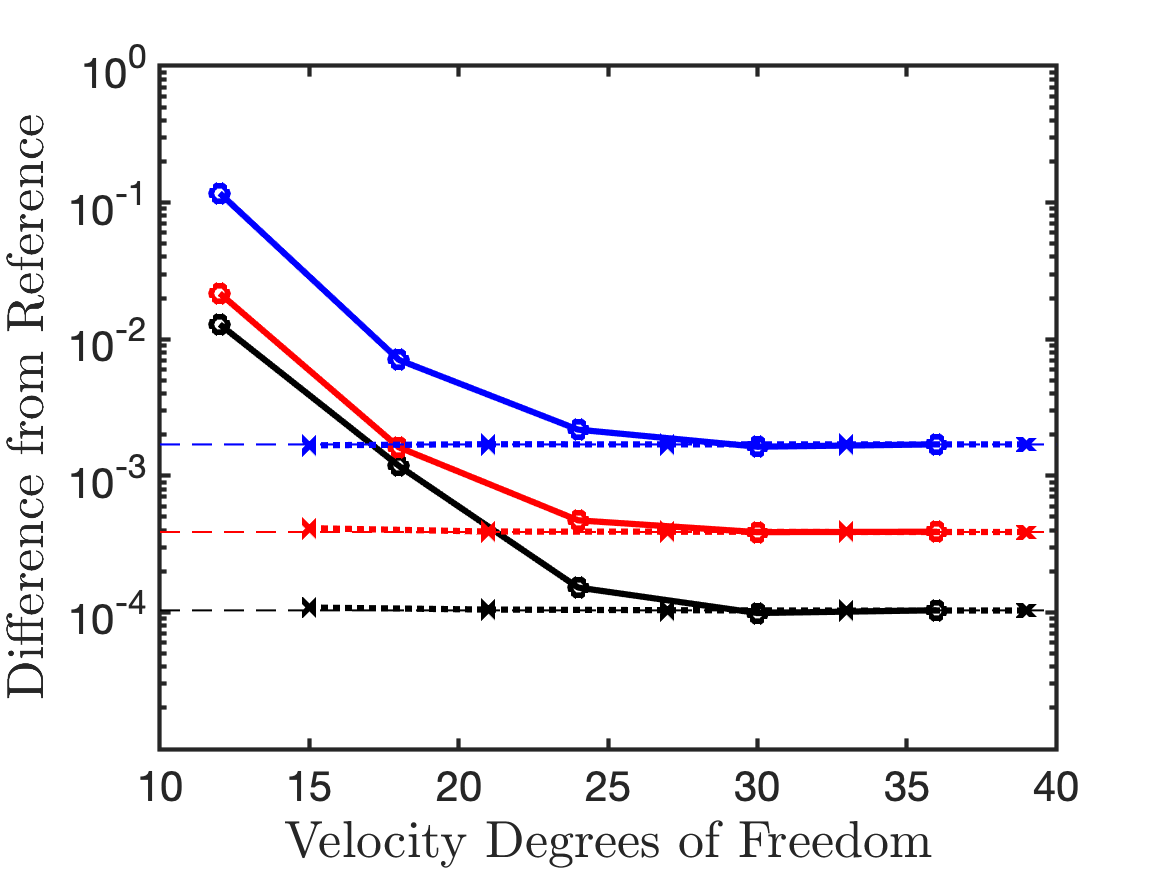}
			\label{fig:RiemannProblem_CompareMoments.1e4}
		\end{minipage}
	}
	\caption{Comparison of the direct and mM methods for the Riemann problem at $t=0.1$, computed with $N^{x}=256$, various values of $N^{v}\in\{4,6,8,10,12\}$, and various values of the collision frequency: $\nu=10^{1}$ (upper left), $\nu=10^{2}$ (upper right), $\nu=10^{3}$ (lower left), $\nu=10^{4}$ (lower right).  
	In each panel we plot the difference in density (black lines), velocity (blue lines), and temperature (red lines) relative to a reference solution (see~Eq.~\eqref{eq:deviationFromReference}), computed with $N^{x}\times N^{v}=256\times64$, versus the number of degrees of freedom in velocity space (see~Eq.~\eqref{eq:velocityDOFs}).  
	Results obtained with the direct method are plotted with solid lines, while results obtained with the mM method are plotted with dotted lines.  
	The horizontal dashed lines (see~Eq.~\eqref{eq:referenceDifference}), where saturation is expected, represent half the difference between the reference solutions obtained with the direct and mM method for density (dashed black lines), velocity (dashed blue lines), and temperature (dashed red lines).}
	\label{fig:RiemannProblem_CompareMoments}
\end{figure}

We use the Riemann problem to do a more quantitative comparison between the direct and mM methods.  
Specifically, we seek to compare the efficiency of the direct and mM methods in terms of accuracy for a given velocity space resolution.  
To this end, we let $X_{N^{v}}\in\{\,n_{N^{v}},\,u_{N^{v}},\,\theta_{N^{v}}\,\}$ denote a velocity moment (i.e., density, velocity, or temperature) computed with $N^{v}$ elements in velocity space, and define the difference of this moment relative to a reference solution $X_{\rm{ref}}$ in the $1$-norm
\begin{equation}
  \delta X_{N^{v}} = \f{|| X_{N^{v}} - X_{\rm{ref}}||_{1}}{|| X_{\rm{ref}} ||_{1}}.  
  \label{eq:deviationFromReference}
\end{equation}
We fix $N^{x}=256$, and compute results with both methods using samples with $N^{v}=4,6,8,10$, and $12$, and $\nu=10^{1},10^{2},10^{3}$, and $10^{4}$.  
Since we do not know the exact solution to the Riemann problem for arbitrary values of the collision frequency, we use solutions obtained with the direct and mM methods, computed using $N^{x}\times N^{v}=256\times64$, to define the reference values $X_{\rm{ref}}$.  
(We are interested in comparing the methods in terms of velocity space resolution, especially when the velocity resolution is coarse, and keep the spatial resolution of the reference the same as in the samples.)  
To reduce bias towards one method in the comparison, the reference solution is obtained by averaging the solutions obtained with the two methods
\begin{equation}
  X_{\rm{ref}}=\f{1}{2}\big(X_{\rm{ref,D}}+X_{\rm{ref,mM}}\big),
\end{equation}
where $X_{\rm{ref,D}}$ and $X_{\rm{ref,mM}}$ denote numerical solutions obtained with the direct and mM methods using $N^{x}\times N^{v}=256\times64$, respectively.  
As the solution obtained with each method converges to its respective reference (i.e., $X_{\rm{ref,D}}$ and $X_{\rm{ref,mM}}$), we expect the difference relative to the reference in Eq.~\eqref{eq:deviationFromReference} to saturate at $\frac{1}{2}\delta X_{\rm{ref}}$, where
\begin{equation}
  \delta X_{\rm{ref}}=||X_{\rm{ref,D}}-X_{\rm{ref,mM}}||_{1}/||X_{\rm{ref}}||_{1}.  
  \label{eq:referenceDifference}
\end{equation}
In this comparison, we use the cleaning limiter with the mM method.  

In Figure~\ref{fig:RiemannProblem_CompareMoments}, we plot the difference from the reference, defined in Eq.~\eqref{eq:deviationFromReference}, for the direct and mM methods versus $N_{\mbox{\tiny DOF}}^{v}$, the number of velocity degrees of freedom per spatial point:
\begin{equation}
  N_{\mbox{\tiny DOF}}^{v} 
  =\left\{
  \begin{array}{cc}
    N^{v} \times (p+1) & (\text{direct method}), \\
    N^{v} \times (p+1) + 3 & (\text{mM method}),
  \end{array}
  \right.
  \label{eq:velocityDOFs}
\end{equation}
that, for the mM method, includes the three moments evolved by the macro equation; i.e., Eq.~\eqref{eq:macroModel}.
Results obtained with various values of the collision frequency are plotted.  
In each panel, we plot the difference in density, velocity, and temperature.  
We also plot half the difference between the two reference solutions $\frac{1}{2}\delta X_{\rm{ref}}$ as horizontal dashed lines in each panel.  

From Figure~\ref{fig:RiemannProblem_CompareMoments}, we observe that the difference $\delta X_{N^{v}}$ decreases (close to exponentially) with increasing velocity resolution for all values of $N_{\mbox{\tiny DOF}}^{v}$ (for both methods) when $\nu=10^{1}$.  
Then, for a given value of $N_{\mbox{\tiny DOF}}^{v}$, the difference is smaller for the direct method due to the additional three degrees of freedom evolved by the mM method.  
For $\nu=10^{2}$, the mM method provides somewhat better accuracy for intermediate values of $N_{\mbox{\tiny DOF}}^{v}$ (compare solid and dotted lines of the same color), while the dotted lines --- representing the mM method --- begin to flatten when the difference approaches $\frac{1}{2}\delta X_{\rm{ref}}$, for larger values of $N_{\mbox{\tiny DOF}}^{v}$.  
The mM method provides substantially improved accuracy for small $N_{\mbox{\tiny DOF}}^{v}$ as the collision frequency is further increased.  
For $\nu=10^{3}$, the mM method compares well with the reference solution when $N^{v}=6$, while for $\nu=10^{4}$ the mM method captures the reference solution well with $N^{v}=4$.  
The direct method requires more velocity degrees of freedom to reach the same level of accuracy.  
For large $N_{\mbox{\tiny DOF}}^{v}$, the difference relative to the reference saturates, as expected, at $\delta X_{N^{v}}=\f{1}{2}\delta X_{\rm{ref}}$ for both methods.  

As seen in Figure~\ref{fig:RiemannProblem_CompareMoments} for small $N_{\mbox{\tiny DOF}}^{v}$, the mM method provides better accuracy than the direct method in the fluid regime, while the direct method catches up as $N_{\mbox{\tiny DOF}}^{v}$ is increased.  
To further emphasize this point, Figure~\ref{fig:RiemannProblem_CompareMoments.Density.1e4} shows the particle density versus position for the case with $\nu=10^{4}$, computed with the direct and mM methods using $N^{x}\times N^{v}=256\times4$ and $N^{x}\times N^{v}=256\times8$.  
When $N^{v}=4$, the mM method compares favorably to the exact solution of the Riemann problem in the Euler limit (solid red line).  
However, when $N^{v}=8$, the direct and mM methods are practically indistinguishable.  

\begin{figure}[H]
	\begin{centering}
	\captionsetup[subfigure]{justification=centering}
	\subfloat[$N^{x}\times N^{v}=256\times4$]
	{\begin{minipage}{0.5\textwidth}
			\includegraphics[width=\linewidth]{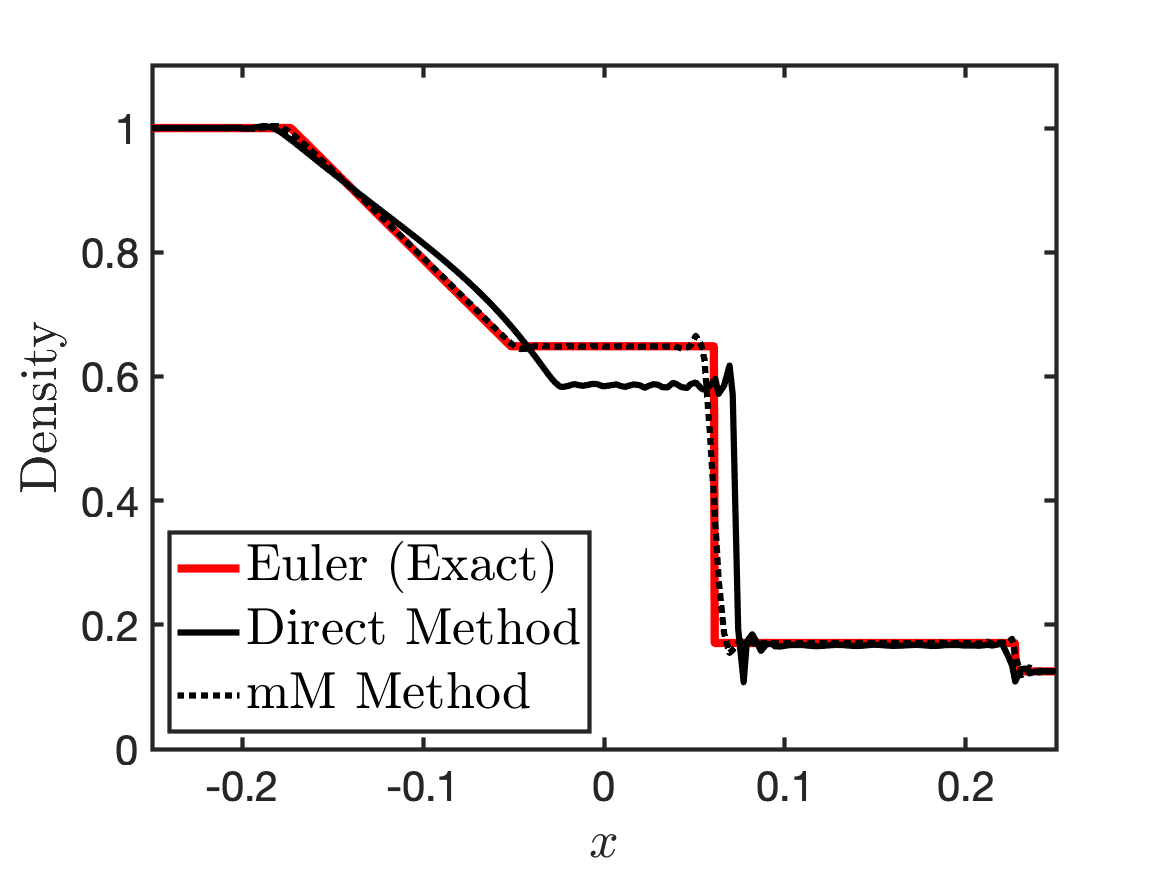}
			\label{fig:RiemannProblem_CompareMoments.Density.1e4.128x04}
		\end{minipage}
	}
	\subfloat[$N^{x}\times N^{v}=256\times8$]
	{\begin{minipage}{0.5\textwidth}
			\includegraphics[width=\linewidth]{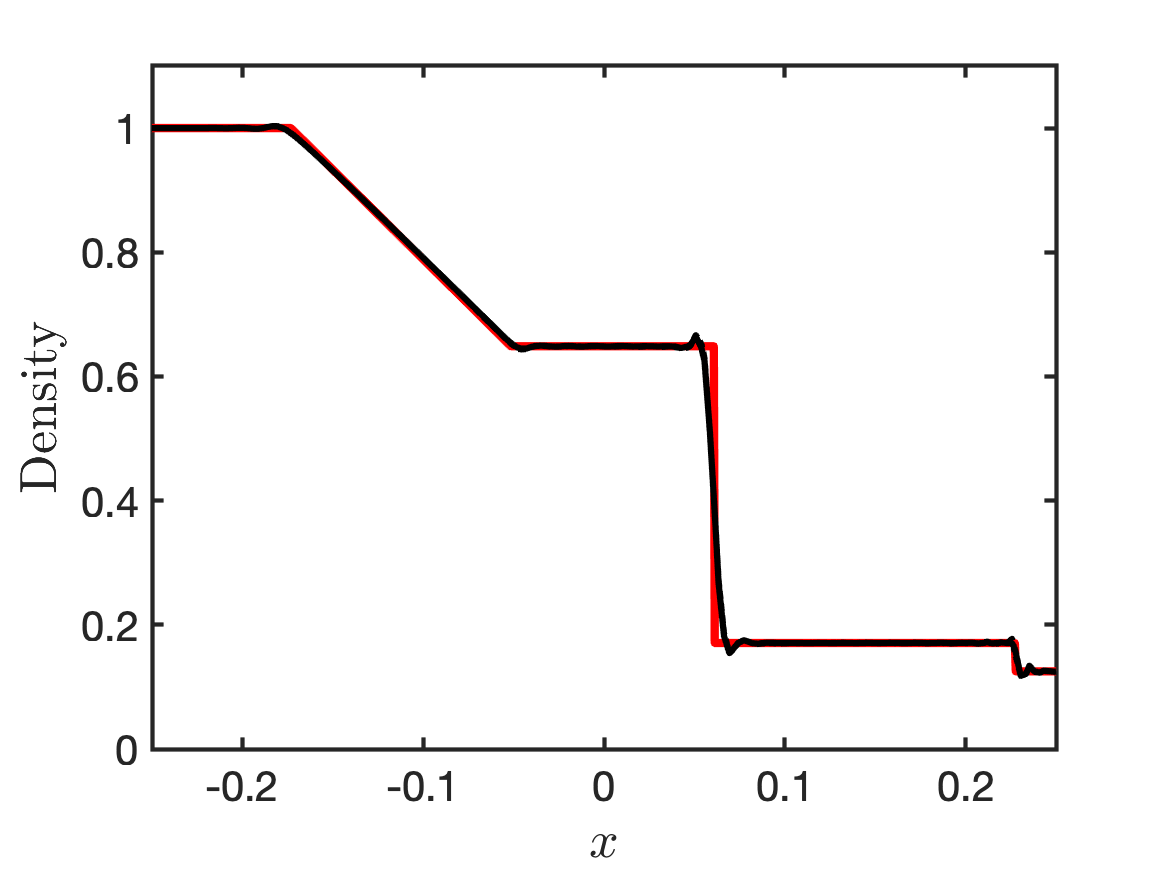}
			\label{ffig:RiemannProblem_CompareMoments.Density.1e4.128x08}
		\end{minipage}
	}
	\caption{Density at $t=0.1$ for the Riemann problem, obtained using $\nu=10^{4}$, and $N^{x}\times N^{v}=256\times4$ (left) and $N^{x}\times N^{v}=256\times8$ (right).  
	Results are plotted for the direct method (solid black line) and the mM method (dotted black line).  
	For reference, the exact solution to the Riemann problem in the Euler limit is also plotted (solid red line).}
	\label{fig:RiemannProblem_CompareMoments.Density.1e4}
	\end{centering}
\end{figure}

\subsection{Two-Stream Instability}
\label{sec:numerical_TwoStream}


Here we consider a collisionless test ($\nu=0$) with electric fields included --- namely a version of the two-stream instability problem (see, e.g., \cite{banksHittinger_2010,rossmanithSeal_2011}) --- to solve the VP system.  
Unless otherwise stated, the computational domain is $D^{x}=D^{v}=[-2\pi,2\pi]$.  
The initial distribution is
\begin{equation}
  f_{0}(x,v) = \big(\,1-\f{1}{2}\cos\big(\,\f{1}{2}x\,\big)\big)\f{v^{2}}{\sqrt{\pi}}\exp\big(\,-v^{2}\,\big).  
\end{equation}
We use periodic boundary conditions in the spatial domain, and zero-flux conditions at the velocity boundaries.  
The simulation is evolved until $t=10$.  
This test typically runs for longer times (e.g., $t=45$), at which point the distribution function exhibits a filamentary structure that requires relatively high resolution to resolve.  
Moreover, it becomes increasingly difficult as the simulation proceeds in time to maintain positivity of the distribution function \cite{rossmanithSeal_2011}.  
Maintaining positive distributions in the context of the mM method is an interesting topic, but outside the scope of the present paper.  
Here, our goal is to compare the direct and mM methods, and in particular to (1) demonstrate the total energy conservation property of the mM method, and (2) investigate any potential impacts of violating the constraints $\vint{\be g}=0$.  
Therefore, mainly to reduce the simulation wall-clock time, we run for a shorter time, with relatively coarse resolution.  

\begin{figure}[H]
	\captionsetup[subfigure]{justification=centering}
	\subfloat[$f(v,x,t=10)$]
	{\begin{minipage}{0.5\textwidth}
		\includegraphics[width=\linewidth]{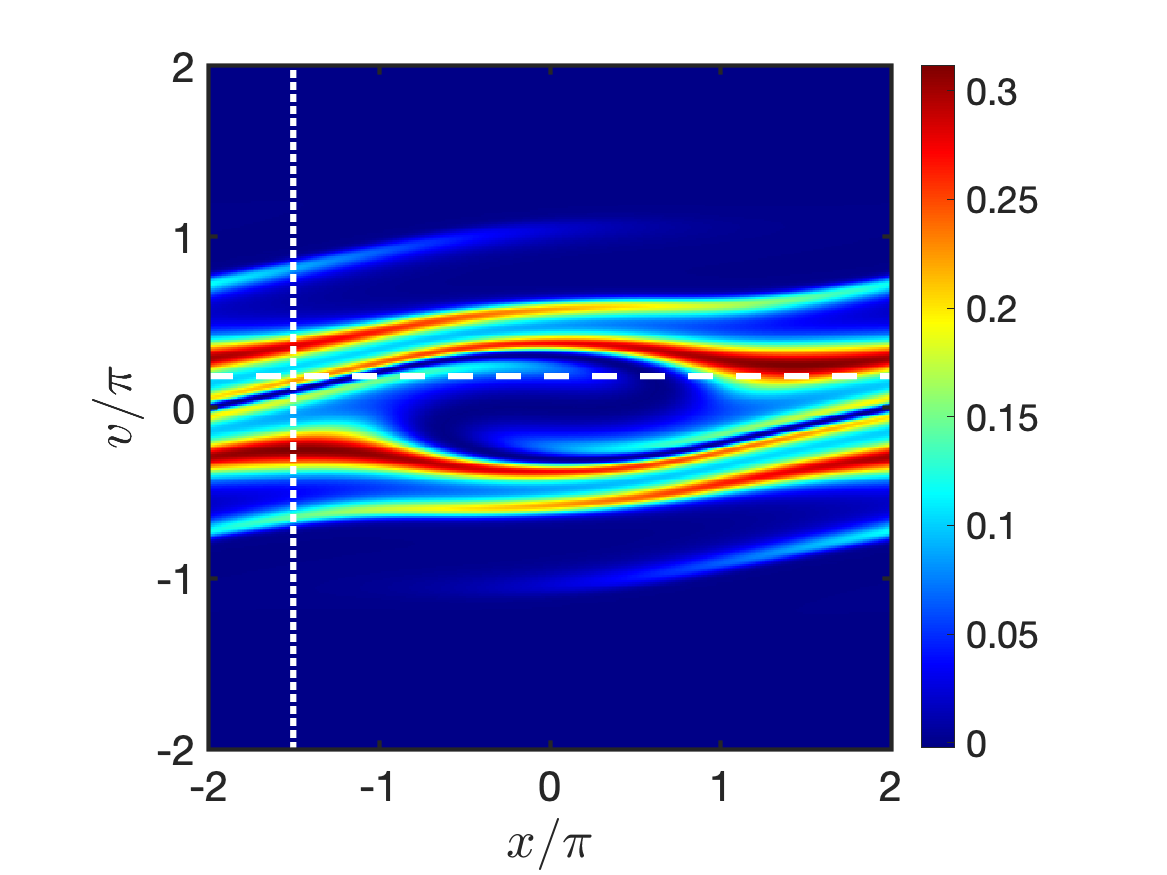}
	\end{minipage}
	}
	\subfloat[$f(v,x/\pi=-1.50,t=10)$]
	{\begin{minipage}{0.5\textwidth}
		\includegraphics[width=\linewidth]{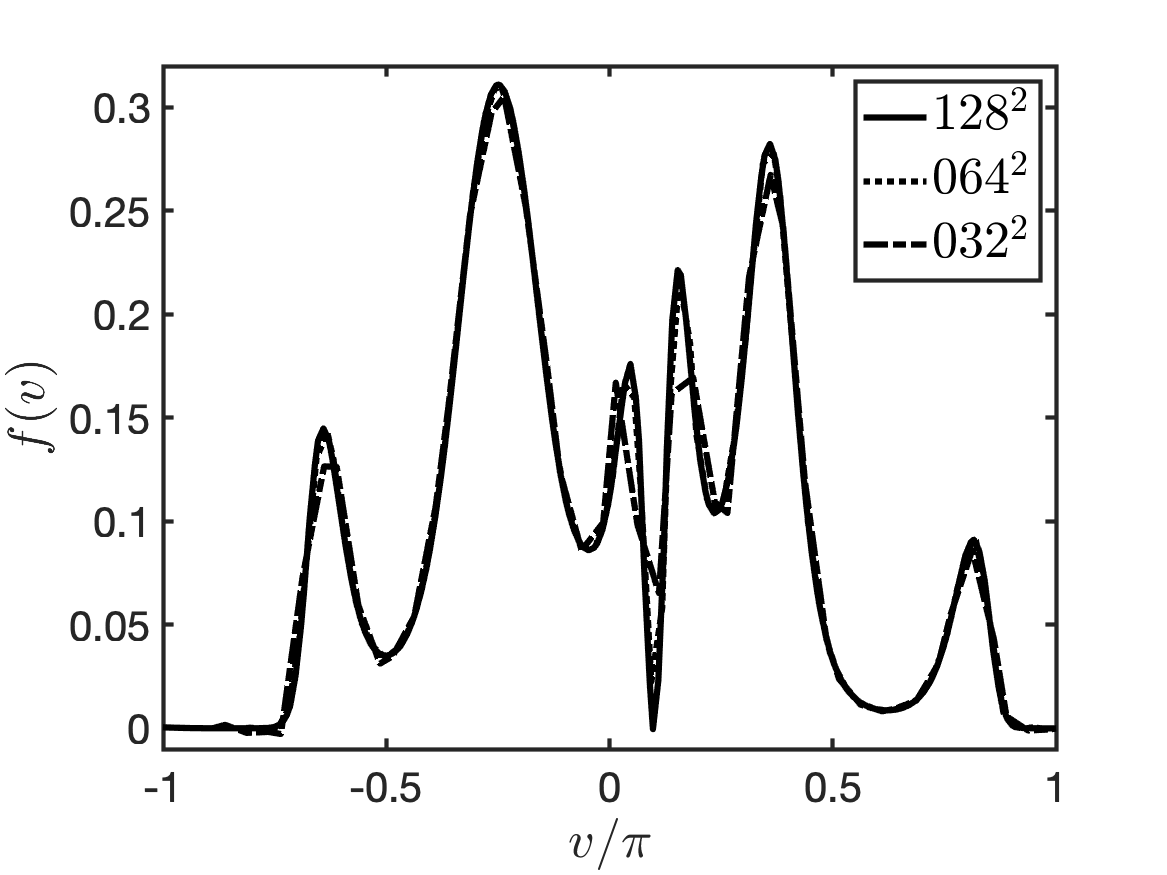}
	\end{minipage}
	} \\
	\subfloat[$f(v/\pi=0.18,x,t=10)$]
	{\begin{minipage}{0.5\textwidth}
		\includegraphics[width=\linewidth]{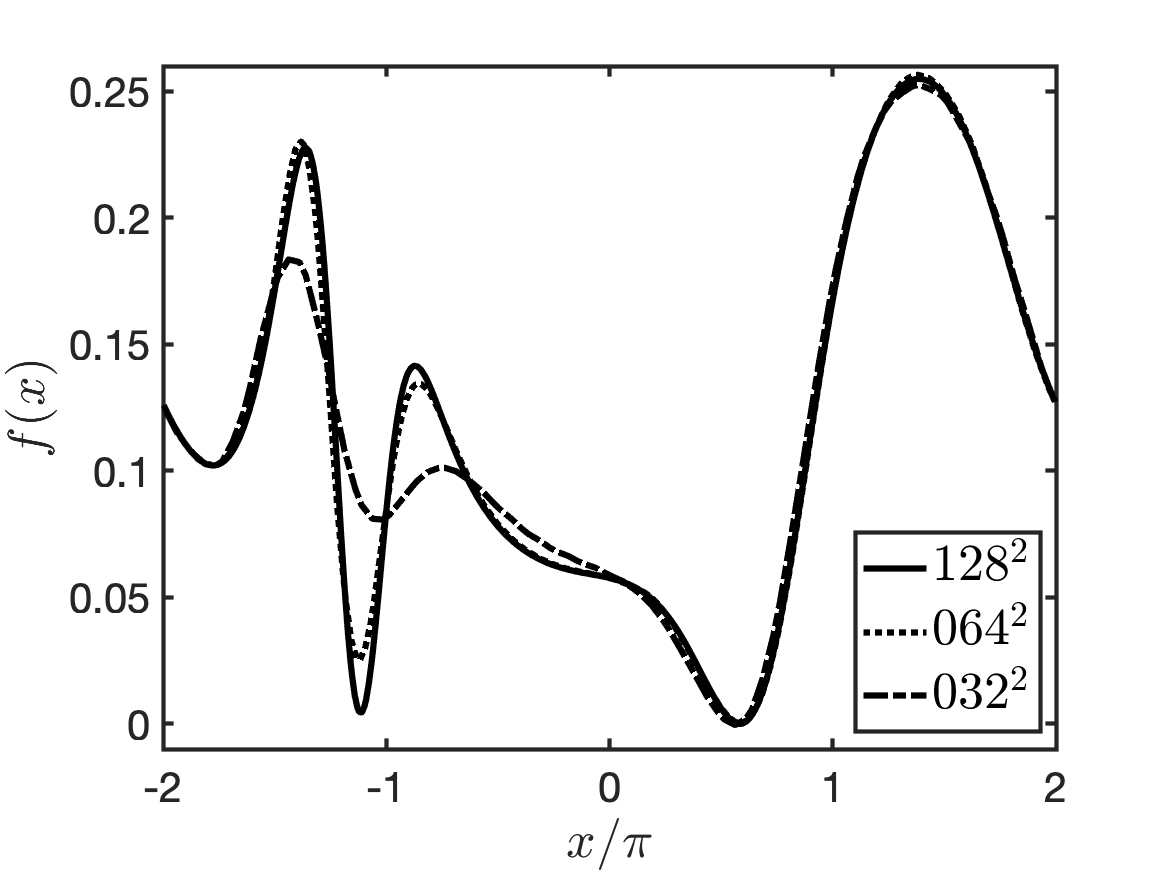}
	\end{minipage}
	}
	\subfloat[Total energy versus time]
	{\begin{minipage}{0.5\textwidth}
		\includegraphics[width=\linewidth]{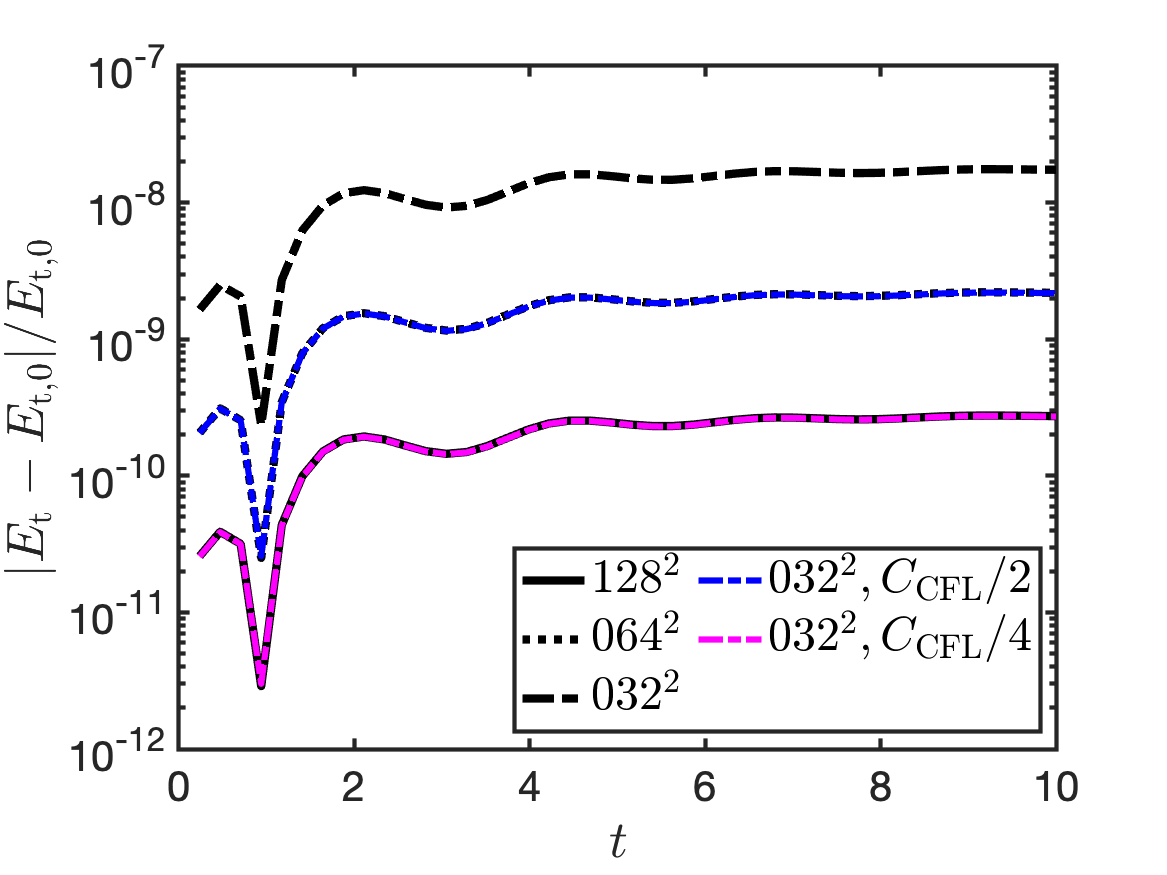}
	\end{minipage}
	}
	\caption{Results for the two-stream instability test, obtained with the mM method using various phase-space resolutions.  In the upper left panel, a snapshot of the distribution function at $t=10$ is shown for a model with $N^{x}\times N^{v}=128\times128$.  In the upper right panel, the distribution at $t=10$ is plotted versus $v$ (for $x/\pi\approx-1.50$; i.e., along the vertical dotted white line in the upper left panel), for models with $N^{x}\times N^{v}=32\times32$ (dash-dot), $64\times64$ (dotted), and $128\times128$ (solid).  Similarly, in the lower left panel, the distribution at $t=10$ is plotted versus $x$ (for $v/\pi\approx0.18$; i.e., along the horizontal dashed white line in the upper left panel).  In the lower right panel, the relative change in the total energy is plotted versus time.  In addition to showing the change in $E_{\rm{t}}$ from the models displayed in the other panels in this figure, we also show results for two models with $N^{x}\times N^{v}=32\times32$, where $C_{\rm{CFL}}$ has been reduced by a factor of two (dash-dot blue line) and four (dash-dot magenta line).}
	\label{fig:TwoStreamInstability.MM}
\end{figure}

Figure~\ref{fig:TwoStreamInstability.MM} displays results obtained with the mM method using the cleaning limiter for various phase-space resolutions.  
The upper left panel displays the distribution function $f=\cE[\bsrho_{f}]+g$ versus $x$ and $v$ at $t=10$, when the characteristic vortex-like structure is starting to form.  
In the upper right and lower left panels, we plot the the distribution function versus $v$ (for $x/\pi\approx-1.50$) and $x$ (for $v/\pi\approx0.18$), respectively.  
These panels indicate that the $64^{2}$ simulation captures the main features of the higher resolution simulation ($128^{2}$) reasonably well.  
In the lower right panel, we plot the relative change in the total energy
\begin{equation}
  E_{\rm{t}} = \f{1}{2}\int_{D^{x}}\big[\,\int_{D^{v}}f\,v^{2}\,dv + E^{2}\,\big]\,dx
  \label{eq:particlePlusPotentialEnergy}
\end{equation}
versus time.  
These results demonstrate that the mM method conserves total energy in the semi-discrete limit (cf. Remark~\ref{rem:momentumEnergyConservation}).  
Specifically, when the spatial resolution increases by a factor of two in each dimension, the time step decreases by a corresponding factor of two due to the time step restriction in Eq.~\eqref{eq:cflTimeStep}, and we find that the magnitude of the relative change in the total energy at $t=10$ is $1.74\times10^{-8}$, $2.19\times10^{-9}$, and $2.74\times10^{-10}$ for phase-space resolutions of $32^{2}$, $64^{2}$, and $128^{2}$, respectively.  
That is, the relative change in the total energy at $t=10$ decreases roughly as $\dt^{3}$, which is expected when the energy conservation is semi-discrete and a third-order accurate time-stepping method (i.e., SSP-RK3) is used.  
To elaborate further, we have computed two additional models with a phase-space resolution of $32^{2}$, but where $C_{\rm{CFL}}$ has been reduced by factors of two and four, so that the time step in these runs is identical to the $64^{2}$ and $128^{2}$ runs, respectively.  
For these runs, the relative change in the total energy is practically indistinguishable from that obtained in the runs with higher phase-space resolution with identical time step.  

\begin{figure}[H]
	\begin{centering}
	\includegraphics[width=0.99\linewidth]{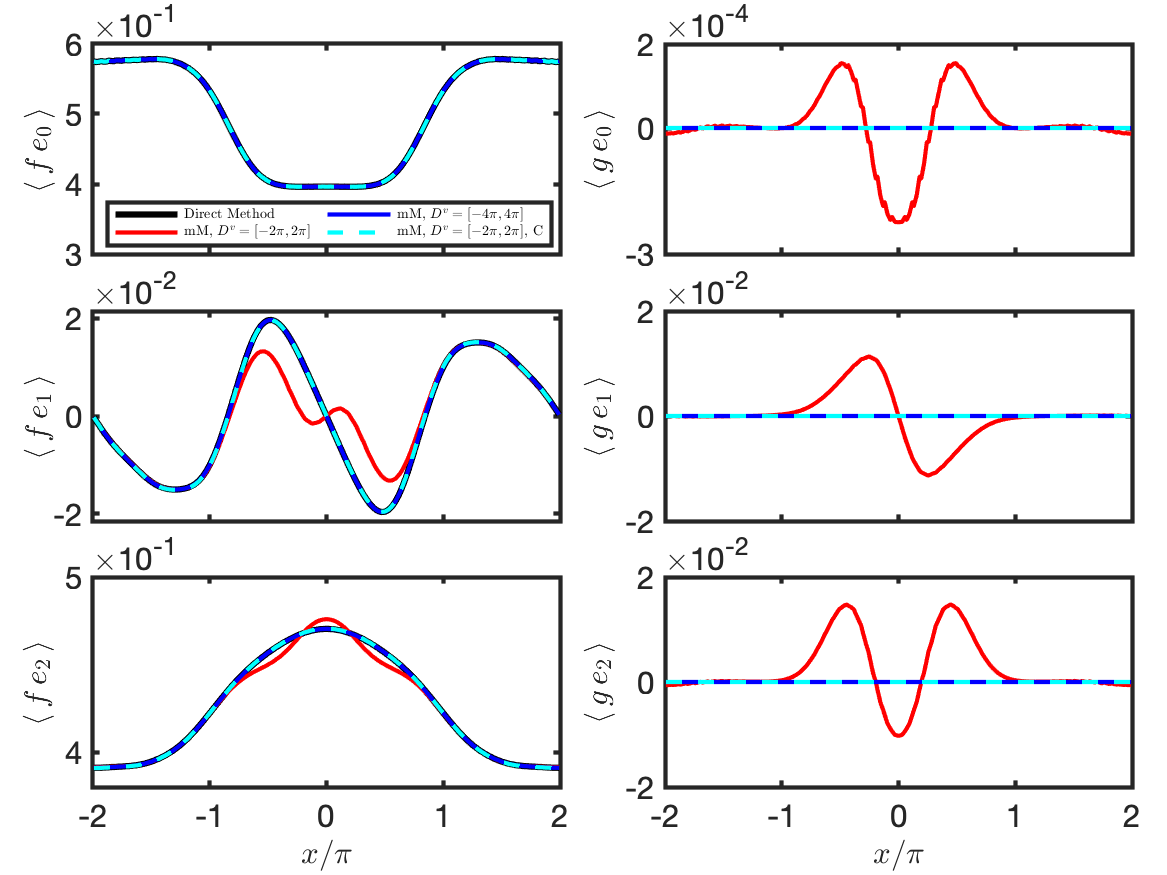}
	\caption{Moments of the distribution function versus position for the two-stream instability at $t=10$, obtained with a phase-space resolution of $N^{x}\times N^{v}=64\times64$.  
	The zeroth, first, and second moments of $f$ are plotted in the upper, middle, and lower panels of the left column, respectively.  
	In each panel, results are shown for the direct method (black lines), the mM method without cleaning limiter (red lines), mM method without cleaning limiter and extended velocity domain (blue lines), and mM method with cleaning limiter (dashed cyan lines).  
	In the right column, we plot the corresponding moments of the micro distribution as obtained with mM method; $\vint{ge_{0}}$ (upper), $\vint{ge_{1}}$ (middle), and $\vint{ge_{2}}$ (lower).}
	\label{fig:TwoStreamInstability.fgMoments}
	\end{centering}
\end{figure}


Next, we present results obtained with the mM method --- with and without the cleaning limiter --- using $N^{x}\times N^{v}=64^{2}$.  
We also compare with results obtained with the direct method.  
In the left column of Figure~\ref{fig:TwoStreamInstability.fgMoments} we plot the zeroth, first, and second moments of the distribution function (components of $\bsrho_{f}$ for the mM method) versus position at $t=10$.  
In each panel, we plot results obtained with the direct method, and results obtained with the mM method without the cleaning limiter, using the fiducial velocity domain $D^{v}=[-2\pi,2\pi]$ and the extended velocity domain $D^{v}=[-4\pi,4\pi]$.  
(For the model with the extended velocity domain we set $N^{v}=128$ so that $\dv$ is unchanged.)  
We also plot results obtained with the mM method using the fiducial domain and the cleaning limiter.  
(Here, for the models using the fiducial velocity domain, we reduced $C_{\rm{CFL}}$ by a factor of two to ensure that the time step is identical to the time step in the model with the extended velocity domain.)  
In the right column of Figure~\ref{fig:TwoStreamInstability.fgMoments} we plot the moments of the micro distribution for the mM models plotted in the left column.  

When the moments of the micro distribution remain small, the results obtained with the mM method are indistinguishable from those obtained with the direct method.  
However, the model computed with the fiducial velocity domain without cleaning deviates significantly from the other models.  
This is particularly evident in the plot of $\vint{f e_{1}}$ in the middle left panel in Figure~\ref{fig:TwoStreamInstability.fgMoments}, but also discernible in the plot of the second moment in the bottom left panel (solid red lines).  
Indeed, the magnitude of first and second moments of the micro distribution reach almost $2\times10^{-2}$ for this model, while magnitude of the zeroth moment is about two orders of magnitude smaller.  
By extending the velocity domain, or applying the cleaning limiter, the moments of the micro distribution remain small everywhere in the spatial domain.  
For the mM model computed on the extended velocity domain, the magnitude of $\vint{g e_{0}}$ is below $5\times10^{-13}$, while the magnitude of $\vint{g e_{1}}$ and $\vint{g e_{2}}$ is below $10^{-10}$.  
For the mM model computed with the cleaning limiter, the magnitude of all the moments of $g$ is at the level of machine precision ($10^{-16}$).  


Figure~\ref{fig:TwoStreamInstability.EnergyChange} shows the relative change in the total energy versus time for the models plotted in Figure~\ref{fig:TwoStreamInstability.fgMoments}.  
The model computed on the fiducial velocity domain without the cleaning limiter exhibits inferior total energy conservation properties, and the magnitude of the relative change in the total energy reaches about $10^{-3}$.  
For the models where the moments of the micro distribution remain small, the relative change in the total energy is practically identical to that obtained with the direct method.  
The negative impact of $\vint{g}\ne0$ on the evolution of the total energy is expected from the considerations in \ref{sec:energyConservationVP}, where we show that, for the mM method, the rate of change of the total energy in Eq.~\eqref{eq:particlePlusPotentialEnergy} is proportional to the spatial integral of $\vint{g} E$ (see Eq.~\eqref{eq:energyConservationVP_MM}).  

\begin{figure}[H]
	\begin{centering}
	\includegraphics[width=0.99\linewidth]{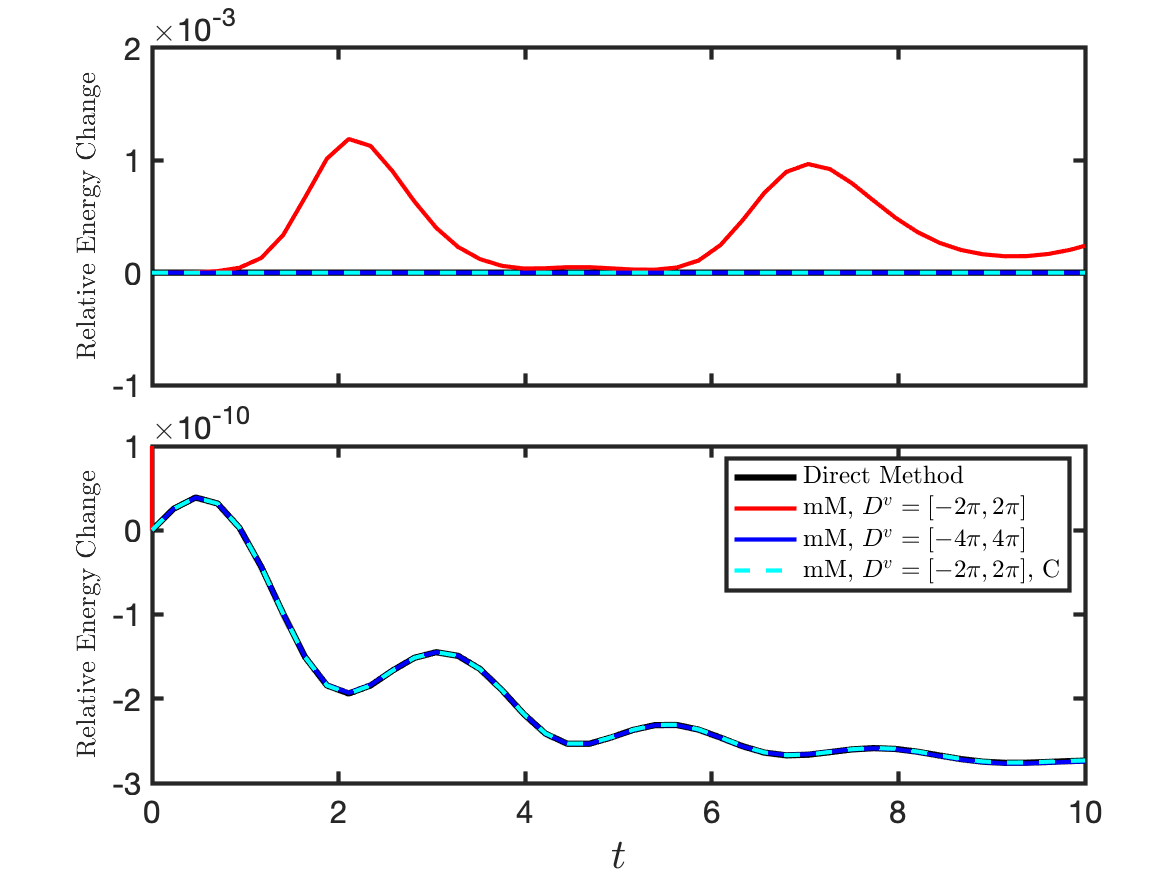}
	\caption{Relative change in total energy versus time for the same models as plotted in Figure~\ref{fig:TwoStreamInstability.fgMoments}.  
	Both panels display the same data, but use different ordinate ranges to more clearly show the behavior of the total energy in the different models.}
	\label{fig:TwoStreamInstability.EnergyChange}
	\end{centering}
\end{figure}

\subsection{Collisional Landau Damping}
\label{sec:numerical_Landau}

Finally, we consider the collisional Landau damping problem (see, e.g., \cite{crestetto_etal_2012,hakim_etal_2020}), which involves both collisions and electric fields.  
We let the computational domain be given by $D^{x}=[-2\pi,2\pi]$ and $D^{v}=[-6,6]$, and the initial distribution be given by
\begin{equation}
  f_{0}(x,v) = \big(\,1+10^{-4}\cos\big(\,\f{1}{2}x\,\big)\,\big)\,\f{1}{\sqrt{2\pi}}\exp\big(\,-\f{v^{2}}{2}\,\big).  
\end{equation}
We use periodic boundary conditions in the spatial domain, and zero-flux conditions at the velocity boundaries, and evolve until $t=50$.  
Our goals are to demonstrate (i) the performance of the mM method with respect to expected damping rates of the electrostatic potential energy and energy conservation properties, and (ii) that the mM method achieves improved accuracy relative to the direct method in the strongly collisional regime.  

\begin{figure}[H]
	\begin{centering}
	\captionsetup[subfigure]{justification=centering}
	\subfloat[$\nu = 0.0$]
	{\begin{minipage}{0.49\textwidth}
			\includegraphics[width=\linewidth]{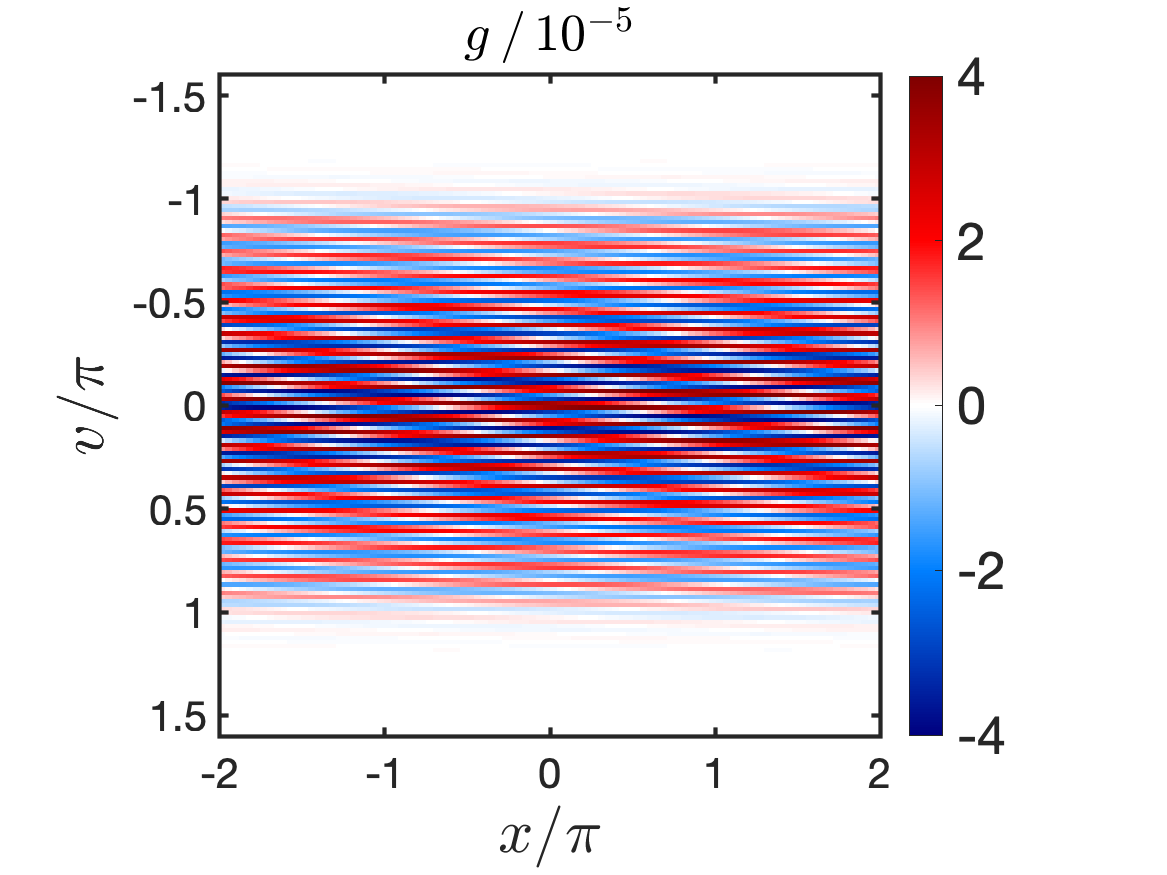}
			\label{fig:CollisionalLandauDamping.gSnapshot.0.00}
		\end{minipage}
	}
	\subfloat[$\nu = 0.25$]
	{\begin{minipage}{0.49\textwidth}
			\includegraphics[width=\linewidth]{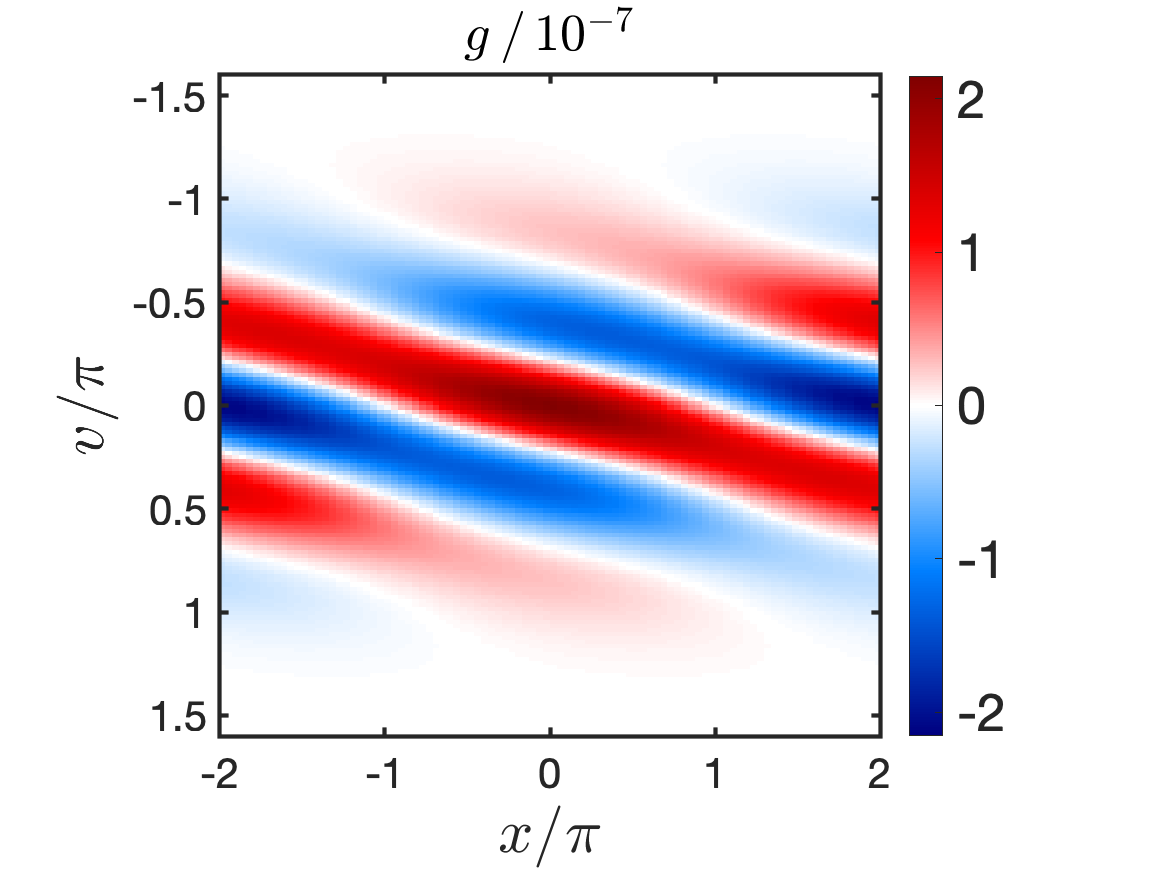}
			\label{fig:CollisionalLandauDamping.gSnapshot.0.25}
		\end{minipage}
	} \\
	\subfloat[$\nu = 1.0$]
	{\begin{minipage}{0.49\textwidth}
			\includegraphics[width=\linewidth]{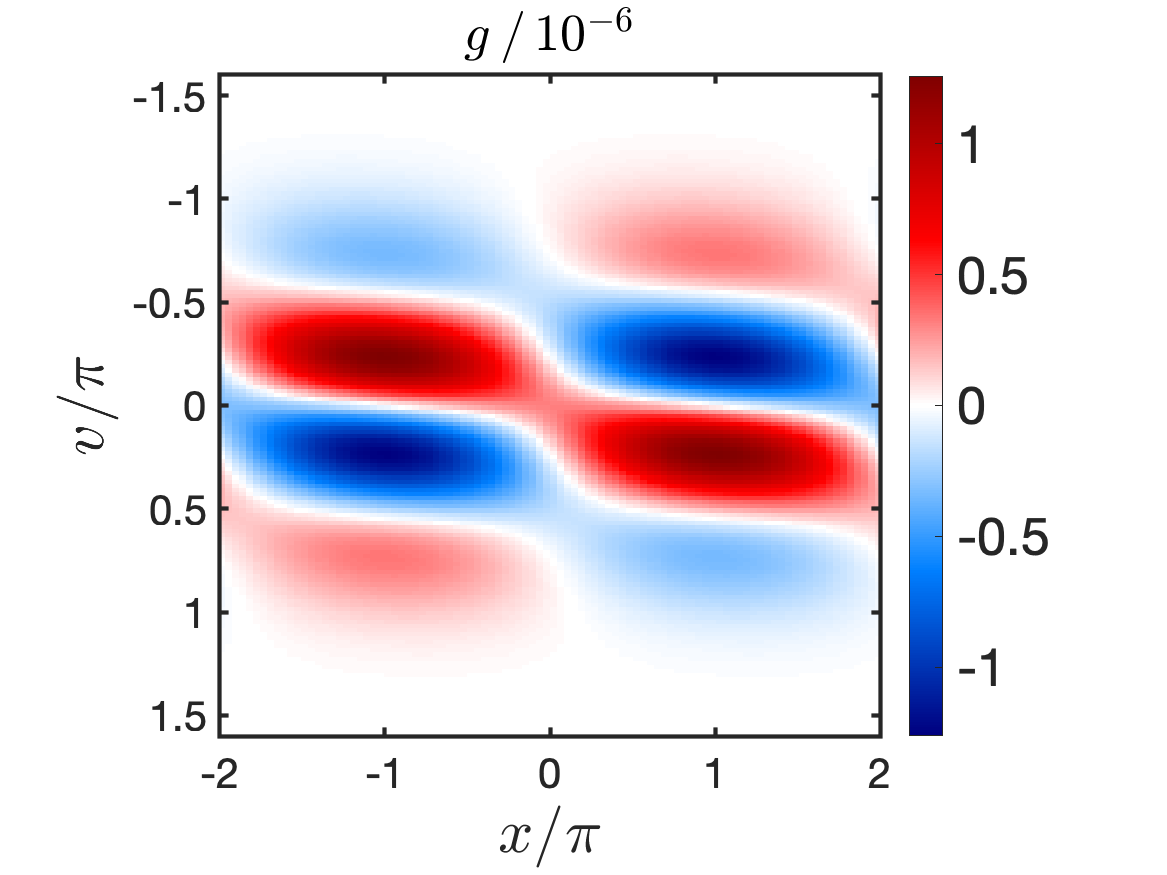}
			\label{fig:CollisionalLandauDamping.gSnapshot.1.00}
		\end{minipage}
	}
	\caption{Snapshots of the micro distribution at $t=50$ for the collisional Landau damping problem with various degrees of collisionality.  
	The upper left panel shows $g/10^{-5}$ for the collisionless case ($\nu=0.0$), the upper right panel shows $g/10^{-7}$ for $\nu=0.25$, while the bottom panel shows $g/10^{-6}$ for $\nu=1.0$.}
	\label{fig:CollisionalLandauDamping.gSnapshots}
	\end{centering}
\end{figure}

Figure~\ref{fig:CollisionalLandauDamping.gSnapshots} shows the micro distribution $g_h$ in the phase-space domain at $t=50$ for three simulations with various degrees of collisionality: $\nu=0.0$, $\nu=0.25$, and $\nu=1.0$.  
These simulations were performed using a phase-space resolution of $N^{x}\times N^{v}=32\times64$.  
In the collisionless case, the micro distribution has evolved into a filamentary structure that is marginally resolved by the velocity grid.  
(Further evolution of this model does not give good agreement with the theoretically predicted damping rate for the electrostatic potential energy; possibly due to the recurrence phenomenon discussed in, e.g., \cite{cheng_etal_2013}.)  
For the moderately collisional cases ($\nu=0.25$ and $\nu=1.0$), the phase-space resolution is adequate for resolving the structures that have developed in the micro distribution at $t=50$.  

\begin{figure}[H]
	\begin{centering}
	\includegraphics[width=1.0\linewidth]{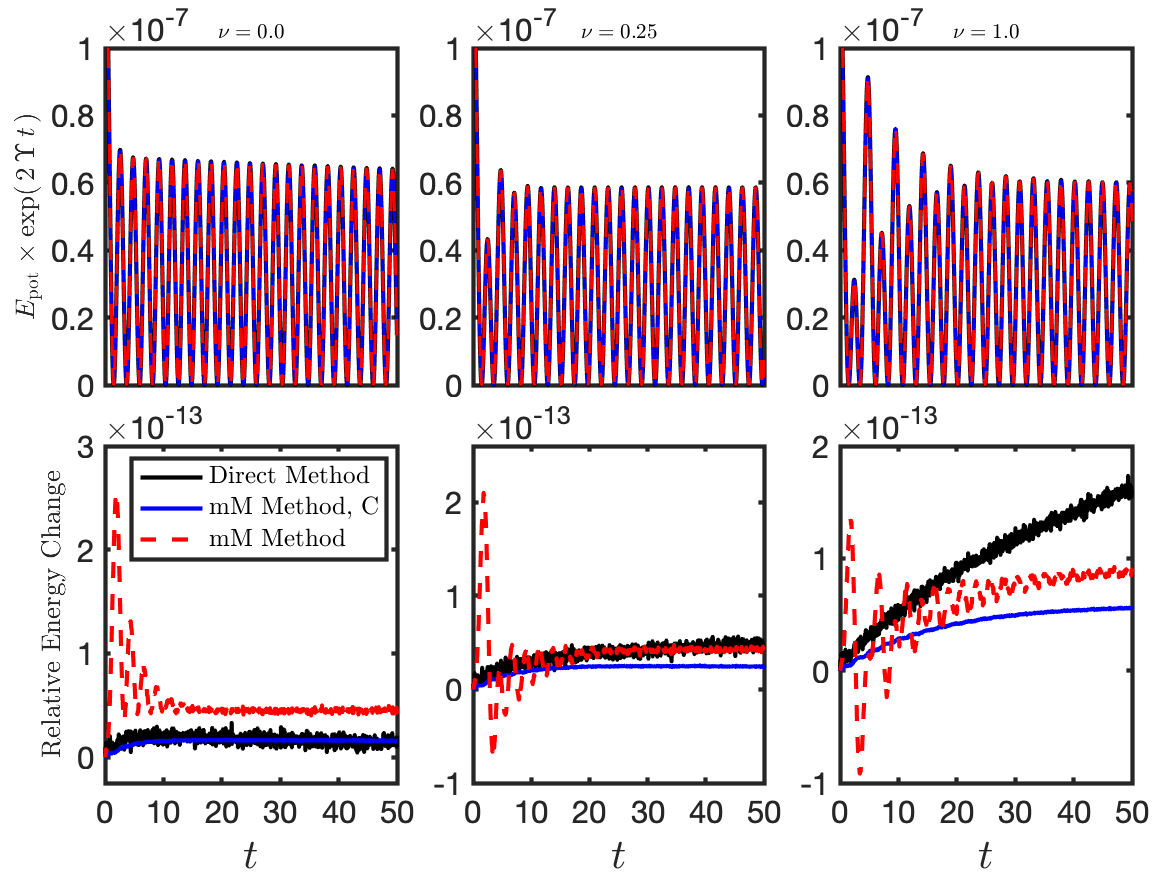}
	\caption{Energetics of the Landau damping problem for collisionless and moderately collisional cases; $\nu=0.0$ (left column), $\nu=0.25$ (middle column), and $\nu=1.0$ (right column).  
	Results obtained with the direct method (black lines) are compared with results obtained with the mM method with and without the cleaning limiter (blue and dashed red lines, respectively).  
	In the upper panels, the potential energy $E_{\rm{pot}}=\f{1}{2}\int_{D^{x}}E^{2}dx$, multiplied by the exponential factor $\exp(2\Upsilon t)$, is plotted versus time.  
	Damping rate estimates of $\Upsilon=0.1534$, $0.0746$, and $0.0312$ are used for the cases with $\nu=0.0$, $0.25$, and $1.0$, respectively (e.g.,~\cite{hakim_etal_2020}).  
	In the lower panels, the relative change in the total energy in Eq.~\eqref{eq:particlePlusPotentialEnergy} is plotted versus time.}
	\label{fig:CollisionalLandauDamping.PotentialAndTotalEnergies.all}
	\end{centering}
\end{figure}

Figure~\ref{fig:CollisionalLandauDamping.PotentialAndTotalEnergies.all} shows the time evolution of the electrostatic potential energy (top panels),
\begin{equation}
  E_{\rm{pot}} = \f{1}{2}\int_{D^{x}}E^{2}\,dx,
\end{equation}
and the relative change in the total energy (bottom panels) for the models with $\nu=0.0$, $0.25$, and $1.0$ (left, middle, and right columns, respectively.)  
Results obtained with the direct method are compared with results obtained with the mM method with and without cleaning (all computed with a phase-space resolution of $N^{x}\times N^{v}=32\times64$).  

The numerical results agree reasonably well with the theoretically predicted damping rates (given in the figure caption).  
Moreover, there is excellent agreement between the direct method and the mM method (with and without cleaning).  
In addition, the relative change in the total energy remains small in all the runs, and the results obtained with the different methods are comparable.  
In all cases, the relative change in the total energy is on the order of $10^{-13}$.  
Even the mM method without cleaning exhibits good energy conservation properties in this case.  
The reason is our consistent discretization of the micro and macro equations, combined with the fact that the velocity domain is sufficiently large so that additional contributions from integrating the Maxwellian beyond $D^{v}$ are negligible.  
Indeed, at $t=50$, the magnitude of the moments of the micro distribution, which are plotted versus position in Figure~\ref{fig:CollisionalLandauDamping.gMoments}, are on the order of $10^{-10}$ (or smaller) for the model with $\nu=0.0$, and on the order of $10^{-11}$ for the models with $\nu=0.25$ and $\nu=1.0$.  

\begin{figure}[H]
	\begin{centering}
	\includegraphics[width=0.9\linewidth]{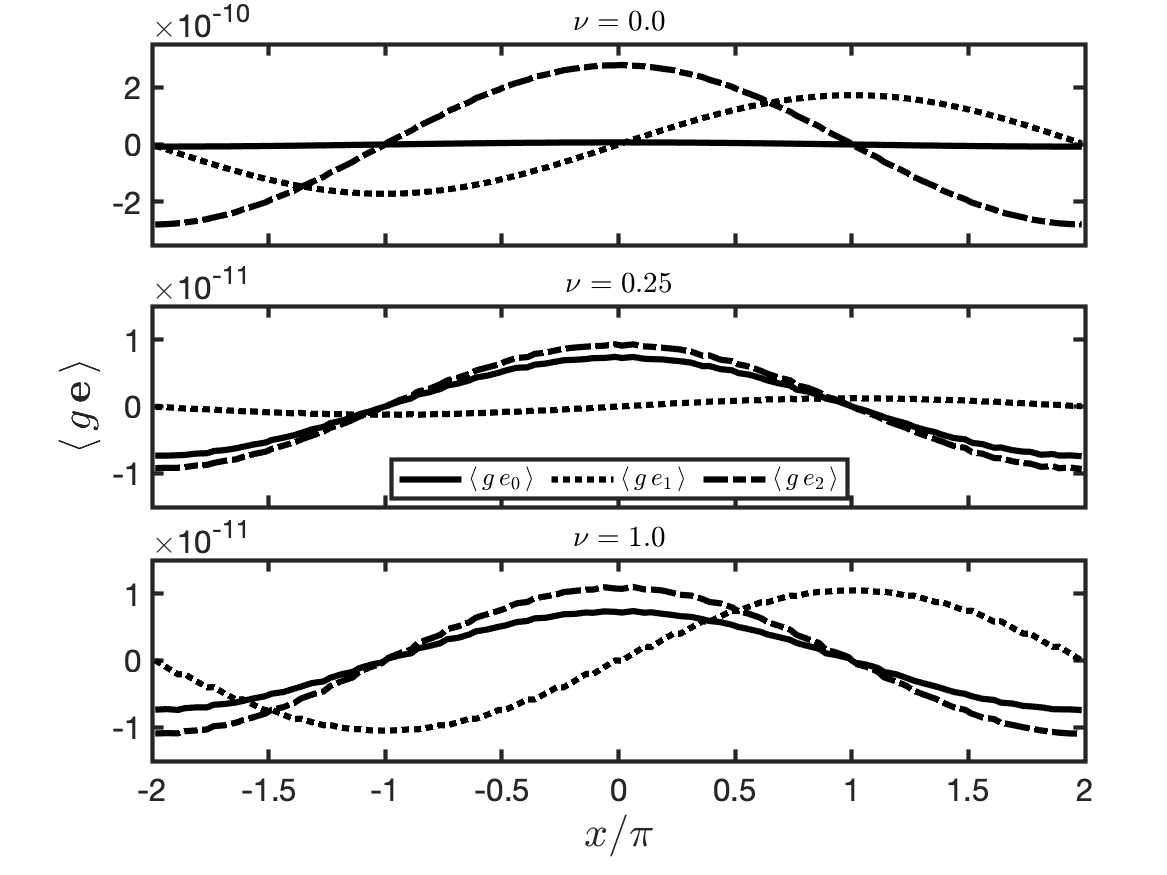}
	\caption{Plots of the moments of the micro distribution at $t=50$, obtained with the mM method without the cleaning limiter (displayed with dashed red lines in Figure~\ref{fig:CollisionalLandauDamping.PotentialAndTotalEnergies.all}).  
	In each panel, we plot $\vint{ge_{0}}$ (solid), $\vint{ge_{1}}$ (dotted), and $\vint{ge_{2}}$ (dash-dot).  
	Results for $\nu=0.0$, $0.25$, and $1.0$ are plotted in the top, middle, and bottom panels, respectively.}
	\label{fig:CollisionalLandauDamping.gMoments}
	\end{centering}
\end{figure}

\begin{figure}[H]
	\captionsetup[subfigure]{justification=centering}
	\subfloat[$N^{x}\times N^{v}=32\times16$]
	{\begin{minipage}{0.5\textwidth}
			\includegraphics[width=\linewidth]{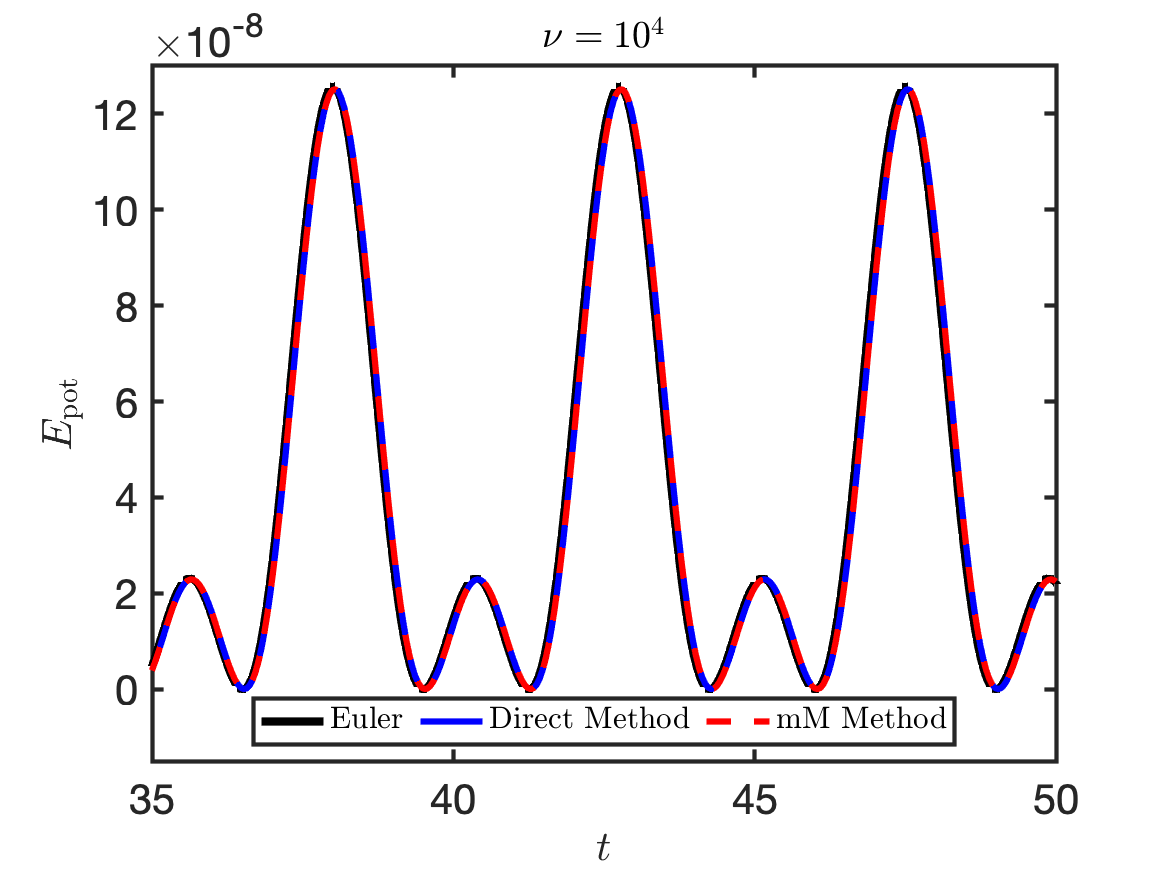}
			\label{fig:CollisionalLandauDamping.Nu_1e4.32x16}
		\end{minipage}
	}
	\subfloat[$N^{x}\times N^{v}=32\times4$]
	{\begin{minipage}{0.5\textwidth}
			\includegraphics[width=\linewidth]{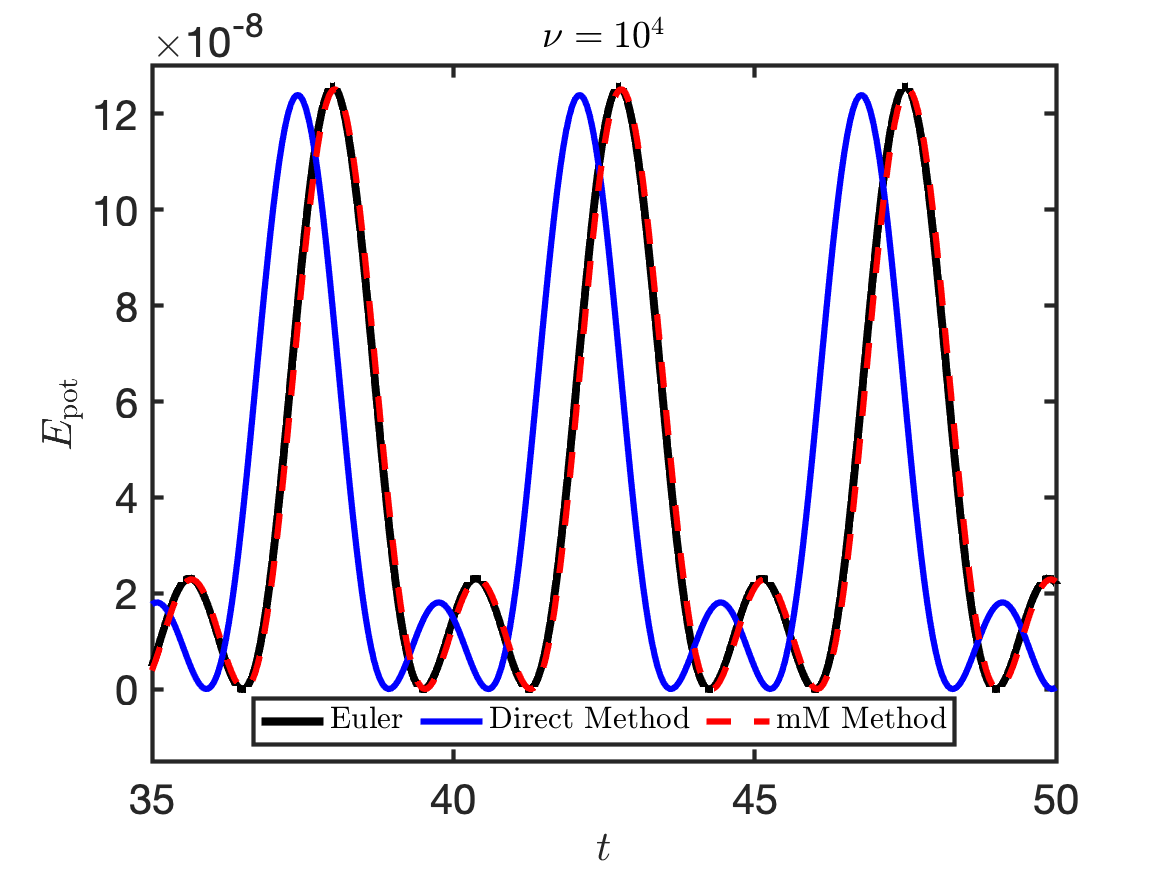}
			\label{fig:CollisionalLandauDamping.Nu_1e4.32x04}
		\end{minipage}
	}
	\caption{Plots of potential energy versus time in the strongly collisional regime ($\nu=10^{4}$).  In both panels, results obtained with the direct method (solid blue lines) and the mM method with the cleaning limiter (dashed red lines) are shown.  Results obtained with $N^{x}\times N^{v}=32\times16$ and $N^{x}\times N^{v}=32\times4$ are plotted in the left and right panels, respectively.  For reference, we also plot results obtained by solving the Euler--Poisson system with $N^{x}=128$.}
	\label{fig:CollisionalLandauDamping.Nu_1e4}
\end{figure}

Finally, we compare the performance of the direct and mM methods in the strongly collisional (fluid) regime.  
Specifically, we set $\nu=10^{4}$.  
As a reference, we use a numerical solution to the Euler--Poisson system, obtained with the DG method for the macro equations from Section~\ref{sec:dgMacro} (with $p=2$ and $g_{h}=0$), using $N^{x}=128$ and SSP-RK2 time-stepping.  
The left panel of Figure~\ref{fig:CollisionalLandauDamping.Nu_1e4} shows the potential energy versus time, as obtained with the Euler--Poisson solver, compared with the direct and mM methods, using a moderate phase-space resolution of $N^{x}\times N^{v}=32\times16$.  
With this resolution, there is very good agreement in the results obtained with the different methods.  
However, when the number of elements in velocity is reduced to $N^{v}=4$, the mM method continues to agree well with the Euler--Poisson solution, while the results obtained with the direct method suffers from a significant phase shift, as well as a slight reduction in peak amplitudes.  
These results, together with the results presented for the Riemann problem in Section~\ref{sec:numerical_Riemann}, demonstrate that the mM method offers improved accuracy with coarse velocity grids in the fluid regime.  
We note that similar findings were reported in \cite{crestetto_etal_2012}, where a particle method was used to solve for the micro distribution in a micro-macro method for solving the Vlasov--Poisson--BGK system.  

%% file: conclusion.tex
\section{Summary and Conclusions}
\label{sec:conclusions}

We have developed a numerical method for the VPLB system in one spatial and one velocity dimension, based on the mM decomposition (mM method), where $f:=\cE[\bsrho_{f}]+g$, that solves a coupled system of equations for macro and micro components ($\bsrho_{f}$ and $g$, respectively), instead of a single kinetic equation for the full distribution $f$ (direct method).  
The use of the mM decomposition is in part motivated by the following two propositions: (i) the macro component captures the dynamics in the collision dominated (fluid) regime with fewer degrees of freedom than the direct method, which suggests that the mM method provides a computational advantage in this regime; and (ii) the macro component evolves conservation laws for particle number, momentum, and energy directly --- as opposed to indirectly in the case of the direct method --- so that exact conservation of these quantities is guaranteed with the mM method, provided the macro equations are formulated in conservation form.  

We use the DG method to discretize the equations in phase-space, and evolve the resulting ODEs with IMEX time integration, where the phase-space advection terms are integrated explicitly and the collision term is integrated implicitly to avoid severe time steps restrictions for stability in collision dominated regimes.  
The discretization of the equations governing the micro and macro components is designed in a consistent manner in order to ensure that the constraints $\vint{\be g}=0$ remain satisfied throughout a simulation, provided they are satisfied initially.  
The constraint-preserving property is proved for the case of an infinite velocity domain, and we introduce a cleaning limiter to enforce $\vint{\be g}=0$ for practical simulations performed on a finite velocity domain.  
(We also discuss an alternate approach, afforded by the consistent discretization, that does not require the cleaning limiter and maintains the constraints to machine precision on a finite velocity domain.)  
For comparison, we also provide a corresponding DG-IMEX discretization to solve directly for $f$ (direct method).  
For the direct method, we prove that the implicit integration of the LB collision operator is conservative for number, momentum, and energy (similar to \cite{hakim_etal_2020} in the context of explicit time integration).  
We present numerical results that demonstrate the performance of the mM method on a set of standard test problems relevant to plasma physics applications: relaxation and Riemann problems, the two-stream instability, and collisional Landau damping.  

In the context of the Riemann problem (and to a certain degree for collisional Landau damping), we show that the mM method is more accurate than the direct method in the fluid regime when the velocity space resolution is coarse.  
However, this conclusion may only follow when the moments of the micro distribution remain small, which is the case for the consistent discretization developed here.  
To emphasize this point, we modified the discretization of the micro component to break the consistency in a way that resulted in unacceptably large violations of the moment constraints, which then resulted in artifacts in the numerical solution to the Riemann problem.  
The accuracy is restored with higher velocity space resolution, the cleaning limiter, or consistent discretization of the macro and micro components.  
The takeaway is that, to better leverage the mM method, the moment constraints on the micro distribution should be enforced.  

We have also demonstrated the conservation properties of the mM method.  
Here too enforcement of the moment constraints plays a crucial role.  
When $\vint{\be g}\ne0$, conserved quantities derived from $f(=\cE[\bsrho_{f}]+g)$ differ from those derived from the macro fields $\bsrho_{f}$, and the conservation properties of the method become ambiguous.  
By maintaining $\vint{\be g}=0$, the ambiguity is removed, and we achieve exact (to machine precision) conservation of particle number, momentum, and energy, when the equations for the macro component are formulated as local conservation laws, as is the case for the Riemann problem, where we set $E=0$.  
For problems with $E\ne0$ (i.e., two-stream instability and collisional Landau damping), the macro model is not formulated in conservative form, and exact conservation of momentum and energy is violated due to discretization errors, similar to the direct method.  
In this case, we demonstrated numerically that total (kinetic plus potential) energy conservation errors in the mM method decrease with decreasing time step as $\dt^{3}$ when using third-order time-stepping, as expected when the conservation is exact in the semi-discrete limit.  
We show in \ref{sec:energyConservationVP} that the total energy is exactly conserved in the semi-discrete limit, provided $\vint{\be g}=0$ holds, and this is consistent with our numerical results.  
Specifically, for the two-stream instability problem, we find that total energy conservation in the mM method is as good as in the direct method when $\vint{\be g}=0$, but significantly worse when the components of $\vint{\be g}$ are allowed to develop non-trivial amplitudes.  

We have considered the VP system, combined with the LB collision operator, in reduced phase-space dimensionality.  
Extensions to the present work includes full phase-space dimensionality and more realistic collision operators (e.g., as in \cite{gamba_etal_2019}).  
In full phase-space dimensionality, we expect the computational advantage of the mM method over the direct method in the fluid regime to be more pronounced (since only two more components are added to $\bsrho_{f}$), but for problems with a spatially and temporally varying degree of collisionality, some form of phase-space adaptivity \cite{hittingerBanks_2013} is desirable to better leverage the mM method.  
Maintaining positivity of the distribution function is desirable, and sometimes necessary.  
(For some of the simulations presented here, we have found that $f$ becomes negative in some regions of phase-space, but this did not prevent the simulations from running to completion.)  
Maintaining $f\ge0$ in the context of the mM method should, however, be considered.  
We hope to address some of these challenges in future work.

%% file: appendix.tex
\section{Energy Conservation in the Vlasov--Poisson Subsystem}
\label{sec:energyConservationVP}

In this appendix we consider energy conservation of the VP model in the semi-discrete setting (i.e., the problem is kept continuous in time).  
When applied to the VP system, both the direct and mM methods satisfy a conservation law for total (kinetic plus potential) energy in the semi-discrete setting.  
Here, we consider only the mM method (but see, e.g., \cite{hakim_etal_2019}, in the context of a direct method.)  
Taking the time derivative of Eq.~\eqref{eq:poissonFEM}, and using Eq.~\eqref{eq:electricField}, gives
\begin{equation}
  - \int_{D^{x}}(\p_{t}E_{h})\,(\p_{x}\psi_{h})\,dx 
  	= \int_{D^{x}}(\p_{t}n_{f,h})\,\psi_{h}\,dx
  \label{eq:app.poissonFEM}
\end{equation}
for $\psi_{h}\in V_{h}$.  
We proceed with a notation where the components of the moments evolved with the macro model are denoted by $\{(\rho_{f,h})_{\ell}\}_{\ell=0}^{2}$; e.g., $n_{f,h}:=(\rho_{f,h})_{0}$.

Next, Eq.~\eqref{eq:app.poissonFEM} is combined with the mM method for the Vlasov equation to derive the semi-discrete conservation statement.  
First, from the first component of Eq.~\eqref{eq:macroSemiDiscrete}, with $\varphi_{h}:=\psi_{h}\in V_{h}$, we obtain
\begin{align}
  \int_{D^{x}}(\p_{t}(\rho_{f,h})_{0})\,\psi_{h}\,dx 
  &= \int_{D^{x}}\big[\,{\rm F}_{0}(\bsrho_{f,h})+{\rm f}_{0}(g_{h})\,\big]\,(\p_{x}\psi_{h})\,dx, \nonumber \\
  &= \int_{D^{x}}\big[\,(\rho_{f,h})_{1}+\vint{g_{h}}_{D^{v}}\,\big]\,(\p_{x}\psi_{h})\,dx,
  \label{eq:app.particleConservationMM}
\end{align}
where ${\rm F}_{0}(\bsrho_{f,h})=(\rho_{f,h})_{1}$ and ${\rm f}_{0}(g_{h})=\vint{g_{h}}_{D^{v}}$ are the first components of $\bF(\bsrho_{f,h})$ and $\bff(g_{h})$, respectively.  
In Eq.~\eqref{eq:app.particleConservationMM} (and in Eq.~\eqref{eq:app.energyConservationMM} below), we assume a periodic spatial domain (or that the numerical fluxes vanish at the spatial domain boundaries).  
Then combining Eqs.~\eqref{eq:app.poissonFEM} and \eqref{eq:app.particleConservationMM}, with $\psi_{h}:=\Phi_{h}$, gives
\begin{equation}
  \int_{D^{x}}\big[\,(\p_{t}E_{h}) + (\rho_{f,h})_{1} + \vint{g_{h}}_{D^{v}}\,\big]\,E_{h}\,dx = 0.  
  \label{eq:poissonParticleConservationCombinedMM}
\end{equation}
From the third component of Eq.~\eqref{eq:macroSemiDiscrete}, with $\varphi_{h}:=1$, we obtain
\begin{equation}
  \int_{D^{x}}\p_{t}(\rho_{f,h})_{2}\,dx = \int_{D^{x}}E_{h}\,(\rho_{f,h})_{1}\,dx, 
  \label{eq:app.energyConservationMM}
\end{equation}
which, when combined with Eq.~\eqref{eq:poissonParticleConservationCombinedMM}, gives
\begin{equation}
  \int_{D^{x}}\big[\,\p_{t}(\rho_{f,h})_{2} + \f{1}{2}\p_{t}E_{h}^{2}\,\big]\,dx = - \int_{D^{x}}\vint{g_{h}}_{D^{v}}\,E_{h}\,dx.  
  \label{eq:energyConservationVP_MM}
\end{equation}
Here, the zeroth moment of $g_{h}$, which emanates from the particle conservation equation in Eq.~\eqref{eq:app.particleConservationMM}, appears on the right-hand side.  
Thus, if $\vint{g_{h}}_{D^{v}}\ne0$, conservation of total energy is compromised, and this conservation error is in addition to errors introduced when the system is discretized in time.  
We investigate the effects of losing conservation in the context of the two-stream instability test in Section~\ref{sec:numerical_TwoStream}.